\documentclass[preprint,12pt,authoryear]{elsarticle}

\usepackage{amssymb,amsthm,amsmath,amsfonts,mathrsfs}
\usepackage{bm}
\usepackage{cancel}
\usepackage{enumerate}
\usepackage{url}
\usepackage{chessboard}
\usepackage{pst-all}
\usepackage{hyperref}
\usepackage{wasysym}
\usepackage{mathtools}
\usepackage{graphicx}
\usepackage{contour}
\usepackage{CJKutf8}
\usepackage{bbm}

\newcommand{\customTiny}{\scalebox{0.23}{
\begin{tikzpicture}
\clip(4.83,0.365) rectangle (5.92,1.452);
\fill[line width=2.pt,fill=black,fill opacity=1.0] (5.212307692307691,0.3709401709401755) --
(5.227098504388038,0.4109165222391566) -- (5.244356740184811,0.4626804457451724) --
(5.260695509797721,0.5186561018862608) -- (5.273760554573898,0.5705732327687449) --
(5.2837430321406975,0.616781456772672) -- (5.293211594495013,0.6690752465306267) --
(5.30077046042753,0.7213397005973032) -- (5.305581482521969,0.7643279442575954) --
(5.308600200833782,0.7997323673195916) -- (5.310920237599428,0.8381937289079981) --
(5.312211337771162,0.8750972065724719) -- (5.312592592592592,0.9088319088319139) --
(5.312211337771162,0.9425666110913559) -- (5.310920237599428,0.9794700887558296) --
(5.308600200833782,1.017931450344236) -- (5.305581482521969,1.0533358734062324) --
(5.30077046042753,1.0963241170665246) -- (5.293211594495013,1.1485885711332011) --
(5.2837430321406975,1.2008823608911556) -- (5.273760554573898,1.2470905848950828) --
(5.260695509797721,1.299007715777567) -- (5.244356740184811,1.3549833719186553) --
(5.227098504388038,1.4067472954246711) -- (5.212307692307691,1.4467236467236524) --
(5.531396011396009,1.4467236467236524) -- (5.516605199315662,1.4067472954246711) --
(5.499346963518889,1.3549833719186553) -- (5.483008193905979,1.299007715777567) --
(5.469943149129802,1.2470905848950828) -- (5.459960671563002,1.2008823608911556) --
(5.450492109208687,1.1485885711332011) -- (5.44293324327617,1.0963241170665246) --
(5.438122221181731,1.0533358734062324) -- (5.435103502869918,1.017931450344236) --
(5.432783466104271,0.9794700887558296) -- (5.431492365932538,0.9425666110913559) --
(5.4311111111111074,0.9088319088319139) -- (5.431492365932538,0.8750972065724719) --
(5.432783466104271,0.8381937289079981) -- (5.435103502869918,0.7997323673195916) --
(5.438122221181731,0.7643279442575954) -- (5.44293324327617,0.7213397005973032) --
(5.450492109208687,0.6690752465306267) -- (5.459960671563002,0.616781456772672) --
(5.469943149129802,0.5705732327687449) -- (5.483008193905979,0.5186561018862608) --
(5.499346963518889,0.4626804457451724) -- (5.516605199315662,0.4109165222391566) --
(5.531396011396009,0.3709401709401755) -- cycle;
\fill[line width=2.pt,fill=black,fill opacity=1.0] (5.909743589743589,0.7492877492877543) --
(5.8697672384446085,0.7640785613681009) -- (5.818003314938593,0.7813367971648741) --
(5.7620276587975034,0.7976755667777837) -- (5.7101105279150195,0.8107406115539613) --
(5.663902303911093,0.8207230891207606) -- (5.611608514153138,0.8301916514750758) --
(5.559344060086461,0.837750517407593) -- (5.516355816426169,0.8425615395020323) --
(5.480951393364173,0.8455802578138449) -- (5.442490031775766,0.8479002945794916) --
(5.405586554111292,0.8491913947512248) -- (5.371851851851851,0.8495726495726554) --
(5.338117149592409,0.8491913947512248) -- (5.301213671927935,0.8479002945794916) --
(5.262752310339529,0.8455802578138449) -- (5.227347887277532,0.8425615395020323) --
(5.18435964361724,0.837750517407593) -- (5.132095189550563,0.8301916514750758) --
(5.0798013997926095,0.8207230891207606) -- (5.033593175788682,0.8107406115539613) --
(4.981676044906198,0.7976755667777837) -- (4.92570038876511,0.7813367971648741) --
(4.873936465259094,0.7640785613681009) -- (4.833960113960112,0.7492877492877543) --
(4.833960113960112,1.0683760683760717) -- (4.873936465259094,1.053585256295725) --
(4.92570038876511,1.0363270204989519) -- (4.981676044906198,1.0199882508860423) --
(5.033593175788682,1.0069232061098647) -- (5.0798013997926095,0.9969407285430654) --
(5.132095189550563,0.9874721661887502) -- (5.18435964361724,0.979913300256233) --
(5.227347887277532,0.9751022781617937) -- (5.262752310339529,0.9720835598499811) --
(5.301213671927935,0.9697635230843344) -- (5.338117149592409,0.9684724229126012) --
(5.371851851851851,0.9680911680911706) -- (5.405586554111292,0.9684724229126012) --
(5.442490031775766,0.9697635230843344) -- (5.480951393364173,0.9720835598499811) --
(5.516355816426169,0.9751022781617937) -- (5.559344060086461,0.979913300256233) --
(5.611608514153138,0.9874721661887502) -- (5.663902303911093,0.9969407285430654) --
(5.7101105279150195,1.0069232061098647) -- (5.7620276587975034,1.0199882508860423) --
(5.818003314938593,1.0363270204989519) -- (5.8697672384446085,1.053585256295725) --
(5.909743589743589,1.0683760683760717) -- cycle;
\end{tikzpicture}}
} 

\newcommand{\customMiny}{\scalebox{0.23}{
\begin{tikzpicture}
\clip(4.83,0.365) rectangle (5.92,1.452);
\fill[line width=2.pt,fill=black,fill opacity=1.0] (5.909743589743589,0.7492877492877543) -- (5.8697672384446085,0.7640785613681009) -- (5.818003314938593,0.7813367971648741) -- (5.7620276587975034,0.7976755667777837) -- (5.7101105279150195,0.8107406115539613) -- (5.663902303911093,0.8207230891207606) -- (5.611608514153138,0.8301916514750758) -- (5.559344060086461,0.837750517407593) -- (5.516355816426169,0.8425615395020323) -- (5.480951393364173,0.8455802578138449) -- (5.442490031775766,0.8479002945794916) -- (5.405586554111292,0.8491913947512248) -- (5.371851851851851,0.8495726495726554) -- (5.338117149592409,0.8491913947512248) -- (5.301213671927935,0.8479002945794916) -- (5.262752310339529,0.8455802578138449) -- (5.227347887277532,0.8425615395020323) -- (5.18435964361724,0.837750517407593) -- (5.132095189550563,0.8301916514750758) -- (5.0798013997926095,0.8207230891207606) -- (5.033593175788682,0.8107406115539613) -- (4.981676044906198,0.7976755667777837) -- (4.92570038876511,0.7813367971648741) -- (4.873936465259094,0.7640785613681009) -- (4.833960113960112,0.7492877492877543) -- (4.833960113960112,1.0683760683760717) -- (4.873936465259094,1.053585256295725) -- (4.92570038876511,1.0363270204989519) -- (4.981676044906198,1.0199882508860423) -- (5.033593175788682,1.0069232061098647) -- (5.0798013997926095,0.9969407285430654) -- (5.132095189550563,0.9874721661887502) -- (5.18435964361724,0.979913300256233) -- (5.227347887277532,0.9751022781617937) -- (5.262752310339529,0.9720835598499811) -- (5.301213671927935,0.9697635230843344) -- (5.338117149592409,0.9684724229126012) -- (5.371851851851851,0.9680911680911706) -- (5.405586554111292,0.9684724229126012) -- (5.442490031775766,0.9697635230843344) -- (5.480951393364173,0.9720835598499811) -- (5.516355816426169,0.9751022781617937) -- (5.559344060086461,0.979913300256233) -- (5.611608514153138,0.9874721661887502) -- (5.663902303911093,0.9969407285430654) -- (5.7101105279150195,1.0069232061098647) -- (5.7620276587975034,1.0199882508860423) -- (5.818003314938593,1.0363270204989519) -- (5.8697672384446085,1.053585256295725) -- (5.909743589743589,1.0683760683760717) -- cycle;
\end{tikzpicture}}
} 

\newcommand{\GL}{{G^\mathcal{L}}}
\newcommand{\GR}{{G^\mathcal{R}}}
\newcommand{\HL}{{H^\mathcal{L}}}

\newcommand{\N}{\mathcal{N}}
\renewcommand{\P}{\mathcal{P}}
\renewcommand{\L}{\mathcal{L}}
\newcommand{\R}{\mathcal{R}}

\newcommand{\cgfuzzy}{\mathrel\|}
\newcommand{\Npc}{\mathbbm{N}\hspace{-0.5pt}\mathbbm{p}}
\newcommand{\Np}{\mathbbm{N}\hspace{-0.5pt}\mathbbm{p}\hspace{-0.5pt}^\infty\!}
\newcommand{\Nc}{\mathbbm{N}\hspace{-0.5pt}\mathbbm{p}\hspace{-0.5pt}^{\emph{C}}\!}
\newcommand{\moon}{\scalebox{1.1}{\ensuremath{\leftmoon}}}
\newcommand{\cgtiny}[1]{\customTiny_{#1}}
\newcommand{\cgminy}[1]{\customMiny_{#1}}

\newtheorem{theorem}{Theorem}

\newtheorem{corollary}[theorem]{Corollary}
\newtheorem{lemma}[theorem]{Lemma}
\newtheorem{definition}[theorem]{Definition}

\newtheorem{observation}[theorem]{Observation}

\newtheorem{example}[theorem]{Example}

\emergencystretch=\maxdimen
\newcommand{\custombrace}[2]{
  \left\{
  \begin{array}{@{}c@{}}
    \rule{0pt}{#1}\\
    #2
  \end{array}
  \right.
}

\newcommand{\custombracee}[2]{%
  \left\}
  \begin{array}{@{}c@{}}
    \rule{0pt}{#1}\\
    #2
  \end{array}
  \right.
}

\newcommand{\customslash}[2]{%
  \makebox[0pt][l]{\rule[-#1]{#2}{#1}}|
}

\journal{TO DECIDE}

\begin{document}

\begin{frontmatter}

\title{Affine Normal Play}

\author{Urban Larsson}
\ead{larsson@iitb.ac.in}
\affiliation{organization={Indian Institute of Technology Bombay},
            country={India}}

\author{Richard J. Nowakowski}
\ead{r.nowakowski@dal.ca}
\affiliation{organization={Dalhousie University},
            country={Canada}}

\author{Carlos P. Santos\corref{correspondingauthor}}
\fntext[fn1]{Carlos Santos' work is funded by national funds through the FCT - Funda\c{c}\~{a}o para a Ci\^{e}ncia e a Tecnologia, I.P., under the scope of the projects UIDB/00297/2020 and UIDP/00297/2020 (Center for Mathematics and Applications).}
\ead{cmf.santos@fct.unl.pt}
\affiliation{organization={Center for Mathematics and Applications (NovaMath), FCT NOVA},
            country={Portugal}}

\begin{abstract}
There are many combinatorial games in which a move can terminate the game, such as a checkmate in {\sc chess}. These moves give rise to diverse situations that fall outside the scope of the classical normal play structure.  To analyze these games, an algebraic extension is necessary, including infinities as elements. In this work, affine normal play, the algebraic structure resulting from that extension, is analyzed. We prove that it is possible to compare two affine games using only their forms. Furthermore, affine games can still be reduced, although the reduced forms are not unique. We establish that the classical normal play is order-embedded in the extended structure, constituting its substructure of invertible elements. Additionally, as in classical theory, affine games born by day $n$ form a lattice with respect to the partial order of games.
\end{abstract}


\end{frontmatter}

\clearpage
\section{Introduction}
\label{sec:intro}
Combinatorial Game Theory (CGT) is the branch of mathematics that studies games with perfect information and no chance, where the players take turns making moves. The required theoretical background for this document pertains to the normal play convention, where the player who cannot move loses. We assume the reader has a decent level of knowledge in CGT, while our interest is solely in non-loopy two-player short games. We refrain from delving into basic concepts such as game forms, options, followers, numbers, nimbers, infinitesimals, temperature, and so forth, even though this document is mostly self-contained. All fundamental concepts are presented and discussed in the references \cite{ANW007,BCG82,Con76,Sie013}.

Since the classical normal play structure $\Npc$  does not cover the possibility of checkmates, i.e., absorbing \emph{terminating moves} that immediately conclude the game, awarding victory to one of the players, it is imperative to consider an extension. Terminating moves are represented by $\infty$ and $\overline{\infty}$, where the first is positive (Left\footnote{In CGT, Left is a female and positive player, while Right is a male and negative player.} wins) and the second is negative (Right wins). The sole goal of this document is to extend $\Npc$ with these elements, exploring the algebraic properties of the resulting structure, which we denote as $\Np$. The idea is analogous to completing the real number line with infinities, leading to the affinely extended real number line. Nevertheless, in the context of games, the analysis of the resulting algebra is considerably more sophisticated. Inspired by this analogy, we refer to $\Np$ as \emph{affine normal play}\footnote{Designation inspired by geometry, where an affine space includes points at infinity to represent classes of parallel lines.}, and its elements are referred to as \emph{affine normal play games} (or just affine games).

Still in this introductory part, Subsection \ref{ssec:motivation} elaborates on the motivation for analyzing the mentioned extension. Examples from game practice are shown, making use of some classical rulesets.\footnote{In strict formal terms, in CGT, the word ``game'' denotes an abstract algebraic object, and the word ``ruleset'' denotes a real-world game. Nevertheless, we maintain a certain flexibility, occasionally using ``game'' to also refer to a real-world game. In such cases, the meaning is evident from the context.} We assume the reader knows the rules of the historical {\sc chess} and {\sc go}. In relation to the second, what is required is familiarity with the rules of \textsc{atari go}, also known as \textsc{first capture go} \citep{watari013,wchess013}. It is also important that the reader is familiar with rulesets analyzed in the specialized literature, particularly {\sc blue-red-hackenbush}, {\sc nim}, {\sc nimstring}, and {\sc amazons} \citep{BCG82}. Having pieces or other elements in many colors, Left is b\textbf{L}ue and Right is \textbf{R}ed. If the colors are black and white, Left is b\textbf{L}ack and Right is White (clea\textbf{R}).

In Section \ref{sec:structure}, fundamental definitions related to the structure under analysis are presented. In this section, Theorem \ref{thm:almostmonoid} addresses ``good'' properties of the disjunctive sum, and Theorem \ref{thm:relations} establishes that game equivalence and game order are indeed an equivalence relation and a partial order relation, respectively.

A relevant result from Section \ref{sec:ftanp} is Theorem \ref{thm:ftnp} (Fundamental Theorem of Affine Normal Play), which establishes that, just like in classical theory, a game $G$ is greater than or equal to zero if and only if Left wins $G$ playing second. As a consequence of this result, we have Corollary \ref{cor:ftnp2} (Order-Outcome Bijection), indicating that there is a total matching between the order relation and the outcome classes.

Section \ref{sec:optionsonly} contains the first of three main results in this document. To better understand this result, we recall that the classical structure $\Npc$ is an abelian group in which the inverse of a game $G$ is its conjugate $-G$, i.e., the game where the players switch roles. It is known that it is possible to compare two games $G$ and $H$ without resorting to the universal statement in the definition of order. This is achieved through an options-only procedure, which consists of playing $G+(-H)$, as $G\geqslant H$ if and only if Left wins $G+(-H)$ playing second. The following definition formalizes ``Left wins playing second'', and has the significant advantage of being useful when analyzing structures that are not groups, as we will see is the case of $\Np$.

\begin{definition}\label{def:maintenance}
Let $G$ and $H$ be two elements of $\Npc$. The pair $(G,H)$ \emph{satisfies the maintenance property}  if
\begin{enumerate}
  \item $\forall G^R$, $\exists G^{RL}$ such that $G^{RL}\geqslant H$ or $\exists H^{R}$ such that $G^R\geqslant H^R$;
  \item $\forall H^L$, $\exists H^{LR}$ such that $G\geqslant H^{LR}$ or $\exists G^L$ such that $G^L\geqslant H^L$.
\end{enumerate}
In that case, we write $\mathbf{M}(G,H)$.
\end{definition}

In other words, in $\Npc$, we have $G\geqslant H$ if and only if $\mathbf{M}(G,H)$. What is new in $\Np$ is that, due to the possibility of forcing sequences of moves, it is possible to have $G\geqslant H$ without the maintenance property as stated in Definition \ref{def:maintenance} being satisfied. Nevertheless, we prove that it is still possible to have an options-only procedure to compare $G$ with $H$. We have $G\geqslant H$ if and only if, playing second, Left wins $G \uplus (-H)$, where $\uplus$ refers to {\sc inquisitor game} that will be formalized later. This is equivalent to saying that, in $\Np$, $G$ is greater than or equal to $H$ if and only if a weaker version of the maintenance property is satisfied, which is the content of Theorem~\ref{thm:constructive}. Note that this is a theoretical fact with little practical relevance, as we estimate that in the vast majority of cases, the comparison can be made using only the maintenance property as stated in the previous definition. Moreover, cases where this cannot be done are immediately recognizable.

In Section \ref{sec:reductions}, it is proven that there are three ways to reduce a game form, with absorbing reversibility being the only one that does not exist in $\Npc$ (Theorem \ref{thm:areversibility}). The most interesting fact is that the reduced forms cease to be unique.

Section \ref{sec:cembedding} contains the second of three main results in this document. In addition to the circumstance that, with an appropriate interpretation, one can prove that $\Npc$ is order-embedded in $\Np$, Theorem \ref{thm:invertible} states the appealing fact that the classical structure is effectively the substructure of the invertible elements of affine normal play. It is not merely a technical proof of a result with little practical consequences. Alongside a straightforward result contained in the third section, this theorem allows us to conclude that when $H$ is a traditional Conway value, comparing $G$ with $H$ can be done by playing $G+(-H)$, just as in classical theory.

Section \ref{sec:classification} contains a classification of affine normal games. One notable feature is the possibility for a component in which, as long as the current player is not mortally threatened in another summand, the first thing they must do is deliver a check there. This possibility led to the formalization of the important concept of hammerzug, and the definition of the pathetic infinitesimals $\cgtiny{\infty}$ and $\cgminy{\infty}$. In a way, $\infty$ and $\overline{\infty}$ close the game line, while $\cgtiny{\infty}$ and $\cgminy{\infty}$ close the descending scales of infinitesimals.

Section \ref{sec:lattice} contains our third main result, establishing that the affine games born by day $n$ form a lattice structure (Theorem \ref{thm:lattice}). Furthermore, it is noteworthy that, while this does not occur in $\Npc$, in $\Np$ a restriction can alter the game equivalence.

Finally, given that the analysis of combinatorial rulesets is a primary application of CGT, the document concludes with Section \ref{sec:example}, exploring an {\sc atari go} endgame by determining the affine game values of its components.

\subsection{Motivational prelude}
\label{ssec:motivation}

\vspace{0.3cm}
The need to use infinities is directly related to the notion of ``urgency'', as measured through the CGT concept of temperature. As an example, consider the {\sc amazons} position~$G$, illustrated in Figure \ref{fig1}.

\begin{figure}[!htb]
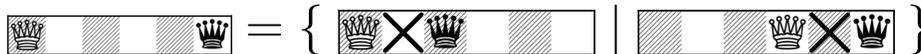

\noindent
\scalebox{0.7}{
\setchessboard{setpieces=\mylist}
\def\mylist{Qa1,qf1}
\chessboard[label=false,maxfield=f1,showmover=false]}\hspace{-0.3cm}\raisebox{0.15cm}{\scalebox{1.6}{$=\{$}}
\hspace{-0.7cm}\scalebox{0.8}{
\setchessboard{setpieces=\mylist}
\def\mylist{Qa1,qc1}
\chessboard[pgfstyle=cross,color=black,markfields=b1,backfields=b1,label=false,maxfield=f1,showmover=false]}
\hspace{-0.5cm}\raisebox{0.15cm}{\scalebox{1.6}{$|$}}
\hspace{-0.7cm}\scalebox{0.8}{
\setchessboard{setpieces=\mylist}
\def\mylist{Qd1,qf1}
\chessboard[pgfstyle=cross,color=black,markfields=e1,backfields=e1,label=false,maxfield=f1,showmover=false]}
\hspace{-0.52cm}\raisebox{0.15cm}{\scalebox{1.6}{$\}$}}

\vspace{-0.3cm}
\caption{A hot {\sc amazons} position.}
\label{fig1}
\end{figure}

While the players have other options, the two depicted in Figure~\ref{fig1} dominate all the others. If Left plays first, she is in time to ``conquer'' territory (by moving to $3$); if Right plays first, then it is Right who can play to $-3$. We have $G=\{3\,|\,-3\}=\pm 3$, with a temperature equal to $3$. Often, in a disjunctive sum $G+X$, both players aim to play in $G$ to gain territory. However, it is important to note the use of ``often'' and not ``always''. Consider the disjunctive sum $G+X=\pm 3+\pm 5$, shown in Figure \ref{fig2}. The component $X$ (at the top of the board) is a component where Left has a move to $5$ and Right has a move to $-5$; therefore, $X$ is hotter than $G$. Of course, in this position, it is now more urgent to play in $X$ than in $G$.

\begin{figure}[!htb]
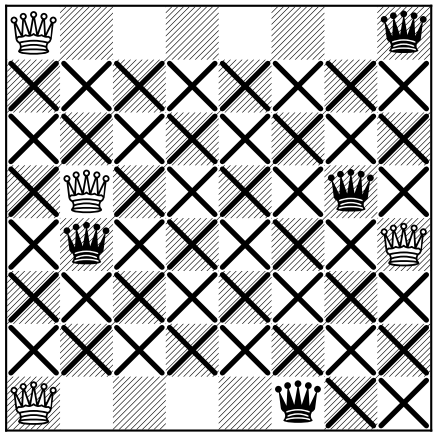

\begin{center}
\setchessboard{setpieces=\mylist}
\def\mylist{Qa1,Qa8,Qb5,Qh4,qh8,qf1,qg5,qb4}
\chessboard[pgfstyle=cross,color=black,
markfields=a2,markfields=a3,markfields=a4,markfields=a5,markfields=a6,markfields=a7,
markfields=b2,markfields=b3,markfields=b6,markfields=b7,
markfields=c2,markfields=c3,markfields=c4,markfields=c5,markfields=c6,markfields=c7,
markfields=d2,markfields=d3,markfields=d4,markfields=d5,markfields=d6,markfields=d7,
markfields=e2,markfields=e3,markfields=e4,markfields=e5,markfields=e6,markfields=e7,
markfields=f2,markfields=f3,markfields=f4,markfields=f5,markfields=f6,markfields=f7,
markfields=g1,markfields=g2,markfields=g3,markfields=g4,markfields=g6,markfields=g7,
markfields=h1,markfields=h2,markfields=h3,markfields=h5,markfields=h6,markfields=h7,
label=false,maxfield=h8,showmover=false]
\end{center}

\vspace{-0.8cm}
\caption{First things first.}
\label{fig2}
\end{figure}

In the mathematical structure $\Npc$, there are no absolutely mandatory moves. More
precisely, given a game $G$, it is always possible to find a disjunctive sum $G+X$
with a more urgent component $X$. However, certain rulesets include terminating moves -- moves that are infinitely advantageous for the players making them, leading to an immediate victory.

The existence of terminating moves enables the existence of threats against which a player must defend. These are commonly referred to as \emph{checks}. Such situations have infinite temperature (the urgency is infinitely extreme). It is worthwhile to re-iterate this point, since it is central to all the analysis that follows: due to the fact that a terminating move has an absorbing nature, a player in check must always try to defend himself. In particular, this allows a player to interpose checks to ``re-arrange the position'' before responding to
the opponent's last move with a quiet move. Even world-class chess players can forget the forcing nature of a \emph{zwischenzug}\footnote{``Zwischenzug'' (German for ``intermediate move'') is a chess tactic where a player, instead of playing the expected move (commonly a recapture), first  threatens the opponent with a check, and only after their response plays the expected move. Such a move is also called ``in-between check'' or \emph{intermezzo}.}, as can be seen in the example illustrated in Figure \ref{fig3}.

\begin{figure}[!htb]
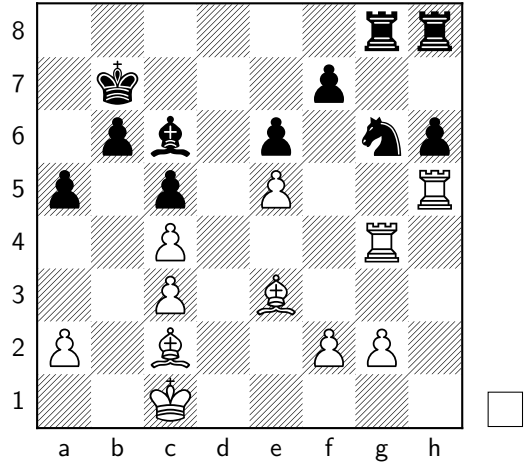

\begin{center}
\setchessboard{setpieces=\mylist}
\def\mylist{Kc1, Pa2, Pc3, Pc4, Pe5, Pf2, Pg2, Bc2, Be3, Rg4, Rh5, pb6, pa5, pc5, pf7, pe6, ph6, bc6, ng6, kb7, rg8, rh8}
\chessboard
\end{center}

\vspace{-0.4cm}
\caption{Carlsen-Anand, World Championship, RUS, 15 Nov 2014, Round 6. An eventual capture of the white pawn at e5 by the black knight is simply bad since that knight is captured by the white rook at h5 after two ``automatic moves'' resulting from the exchange of a pair of rooks at g8. Due to that, Carlsen thought he had the position under control and made the very bad move 26.Kd2?. Incredibly, Anand missed the opportunity and replied also with the bad move 26...a4?. Anand could have explored the opportunity with a pair of zwischenzugs: 26.Kd2 Ne5! 27.Rg8 Nc4 (zwischenzug) 28.Kd3 Nb2 (one more zwischenzug) 29.Kd2 Rg8 (only now Black recaptures the white rook at g8, with a winning position).}
\label{fig3}
\end{figure}

Checks and forcing sequences cannot be represented by elements of $\Npc$. Consequently, the $\Np$ framework is needed to analyze some rulesets. In that framework, $\infty+\overline{\infty}$ is not defined. However, if $X \neq \overline{\infty}$, then $\infty + X = \infty$, and if $X \neq \infty$, then $\overline{\infty} + X = \overline{\infty}$. These follow from the absorbing nature of the terminating moves. There are three important situations of play that require $\Np$ for their analysis. Examples are given in the next subsections. In the first, the checks are overt, but they are hidden in the second and third.

\subsubsection{Terminating moves and infinite heat components}
\label{sssec:terminating}

\vspace{0.3cm}
{\sc chess} was mentioned because almost all readers are familiar with checks and checkmates within the context of this ruleset. However, it is a loopy ruleset that allows ties. Additionally, {\sc chess} positions do not tend to break down into disjoint components. In this work, from now on, we will use {\sc atari go} to exemplify some concepts, as it is a ruleset for which the use of $\Np$ is particularly suitable.

 In Figure \ref{fig4}, the black stones marked with triangles are ``alive''. This is a term in {\sc go} indicating that they cannot be captured; possibly, they are connected to some group with ``eyes'', as indicated by the arrow in the right corner. Thus, the region bounded by these stones is a disjoint component.

\begin{figure}[!htb]
\begin{center}
\includegraphics[scale=1]{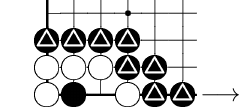}
\end{center}

\vspace{-0.4cm}
\caption{Infinite heat {\sc atari go} component.}
\label{fig4}
\end{figure}

Left, playing first, can capture a white stone, winning the game; Right, playing first, can capture a black stone, winning the game. Thus, $H=\{\infty\,|\,\overline{\infty}\}=\pm\infty$ has infinite temperature; the player who has the move wins the game. For $X$ different from $\infty$ and $\overline{\infty}$, in $H +X$, $X$ is irrelevant, regardless of what $X$ may be. This is in sharp contrast to the {\sc amazons} position in Figure \ref{fig1}.\\

Consider now the {\sc atari go} position $W$ shown in Figure \ref{fig5}. The Left option $W^L$ corresponding to placing a black stone on point ``a'' is a check since it threatens an immediate capture\footnote{There is a much better move for Left than this check, which is left as an exercise for the reader.}.

\begin{figure}[!htb]
\begin{center}
\includegraphics[scale=1]{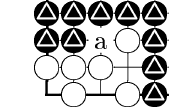}
\end{center}

\vspace{-0.4cm}
\caption{The Left option $W^{L}$ is a check.}
\label{fig5}
\end{figure}

Right can defend with the option $W^{LR}$, as shown in Figure \ref{fig6}.

\begin{figure}[!htb]
\begin{center}
\includegraphics[scale=1]{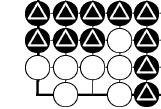}
\end{center}

\vspace{-0.4cm}
\caption{$W^{LR}=0$.}
\label{fig6}
\end{figure}

Note that in $W^{LR}$, a Left option $W^{LRL}=\overline{\infty}$ is a \emph{suicidal} move, that is, an option that Left never chooses unless it is the only thing she can do. Having this option or not makes no difference. Sometimes, in certain rulesets, such as {\sc go} or {\sc atari go}, these moves are considered illegal and cannot be made. For the sake of formalization, even in such cases, $\overline{\infty}$ is still included in the game form as if the move could be made, since this procedure does not change the nature of the ruleset. A single Left option $\overline{\infty}$ is the equivalent in $\Npc$ to having an empty set of options; if there are other Left options, $\overline{\infty}$ is always dominated. On the other hand, the Right options in $W^{LR}$ are moves that allow Left to immediately play to $\infty$. These moves take the form $\{\infty\,|\,\ldots\}$, and Right never chooses them unless it is the only thing he can do. In practice, these Right options can be replaced by $\infty$ since they are also a form of suicide, just taking one more move. Thus, $W^{LR}=\{\overline{\infty}\,|\,\infty\}$ is a game in which neither player has moves except those that self-inflict defeat. As we will see later, this game is zero, the identity of $\Np$. It is now easy to understand that the initial check is $W^L=\{\infty\,|\,0\}$.

In summary, if a Left option of a game $G$ is $G^L=\infty$, then $G^L$ is a checkmate; if a Left option is $G^L=\{\infty\,|\,G^{L\mathcal{R}}\}$, then $G^L$ is a check; if a Left option is $G^L=\overline{\infty}$, then $G^L$ is a suicidal move; if a Left option is $G^L=\{G^{L\mathcal{L}}\,|\,\overline{\infty}\}$, then the move allows being checkmated and can be replaced by $\overline{\infty}$. The same logic applies to Right options.

\subsubsection{Entailing moves}
\label{sssec:entailing}

\vspace{0.3cm}
\emph{Entailing moves}\footnote{The designation ``entailing move'' was proposed in \cite{BCG82} precisely to convey the idea that a player is compelled (``entailed'') to play in a certain way.} are moves that \emph{force} the opponent to respond locally within a certain range of options. It may be possible for the opponent to respond locally in more than one way; nevertheless, the opponent is restricted to those local options by the entailing move. Of course, this disrupts the usual logic of disjunctive sums. Consider a disjunctive sum $G_1+G_2+\ldots+G_k$. If Left  is not lethally threatened in any component and makes an entailing move in $G_1$ to $G_1^L$, then the imposed restriction on Right prevents him from freely choosing any component from $G_1^L+G_2+\ldots+G_k$ for his reply.

As an example, consider {\sc nim} played on sums of piles of stones $G_1+\ldots+G_k$, but with the following change in the rules: whenever a player leaves three stones in one of the piles, the opponent must play in that very same pile on the next move. This ruleset does not have mathematical interest, since a straightforward  argument shows that all initial positions with at least a pile of more than 3 stones are $\mathcal{N}$-positions. The representation of positions through game forms is what is of interest. A pile of four stones in traditional {\sc nim} is described through a game form of $\Npc$, more specifically through $\{0,*,*2,*3\,|\,0,*,*2,*3\}$.
In the modified version, a pile of four stones should be described as $\{0,*,*2,*3^{\text{\scalebox{0.7}{\(\swarrow\)}}}\,|\,0,*,*2,*3^{\text{\scalebox{0.7}{\(\swarrow\)}}}\}$, where $*3^{\text{\scalebox{0.7}{\(\swarrow\)}}}$ means ``next player is compelled to play in this component, even when the game is not played in isolation''. It is because there are no game forms of this type in $\Npc$ that $\Np$ is necessary. The question that arises is how one can use a form of $\Np$ to represent the desired situation.

Any {\sc chess} player knows that a check compels the opponent to defend. In other words, a check is, by its nature, an entailing move. Since some of elements of $\Np$ are checks, they can be employed as ``gadgets'' to describe situations similar to the modified {\sc nim}. For example, $\{0,*,*2,\{\infty\,|\,0,*,*2\}\,|\,0,*,*2,\{\,0,*,*2\,|\,\overline{\infty}\}\}$ perfectly represents a pile of four stones in the modified ruleset. As in the movie Godfather, checks are ``offers'' (threats) that the opponent cannot refuse (cannot avoid defending). This explains why these gadgets work. When Left moves to $\{\infty\,|\,0,*,*2\}$, Right has to defend by playing to $0$, $*$, or $*2$. Since $0$, $*$, or $*2$ are the options Right would have if forced to play in $*3$, the Left option $\{\infty\,|\,0,*,*2\}$ is, in practice, a move that works as one that compels Right to play in $*3$.

In summary, if a Left entailing option compels Right to respond with an element from $A$, the way to represent this entailing move through a game form of $\Np$ is to use a check $G^L$ where $G^{L\mathcal{R}}$ is the set $A$, i.e., $G^L = \{\infty\,|\,A\}$. Obviously, Right entailing options are represented in a similar manner.

\subsubsection{Carry-on moves}
\label{sssec:carry}

\vspace{0.3cm}
Rulesets like {\sc nimstring} allow the possibility of \emph{carry-on} moves, moves that compel the player making them to retain the right to play. In this particular ruleset, when a player closes a box, the player must move again. Figure \ref{fig7} shows a {\sc nimstring} position that consists of a disjunctive sum $G+H$, where $G$ is the component on the left and $H$ is the component on the right. For example, if Left closes a box in the left component, Left \emph{has to play again}, being able to \emph{choose either of the two components to continue the play}.

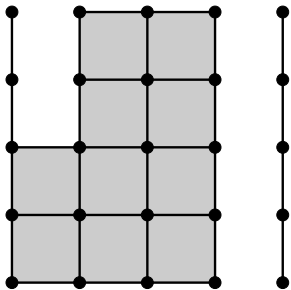
\begin{figure}[!htb]
\begin{center}
\scalebox{.9}{
\definecolor{cqcqcq}{rgb}{0.7529411764705882,0.7529411764705882,0.7529411764705882}
\begin{tikzpicture}
\clip(9.5,6.5) rectangle (14.5,11.5);
\fill[gray!40!white] (13.,11.) -- (11.,11.) -- (11.,9.) -- (10.,9.) -- (10.,7.) -- (13.,7.) -- cycle;
\draw [line width=1.pt] (10.,11.)-- (10.,10.);
\draw [line width=1.pt] (10.,10.)-- (10.,9.);
\draw [line width=1.pt] (10.,9.)-- (11.,9.);
\draw [line width=1.pt] (11.,9.)-- (11.,10.);
\draw [line width=1.pt] (11.,10.)-- (11.,11.);
\draw [line width=1.pt] (12.,11.)-- (12.,7.);
\draw [line width=1.pt] (13.,11.)-- (13.,7.);
\draw [line width=1.pt] (14.,11.)-- (14.,7.);
\draw [line width=1.pt] (11.,11.)-- (13.,11.);
\draw [line width=1.pt] (11.,10.)-- (13.,10.);
\draw [line width=1.pt] (11.,9.)-- (13.,9.);
\draw [line width=1.pt] (10.,9.)-- (10.,7.);
\draw [line width=1.pt] (11.,9.)-- (11.,7.);
\draw [line width=1.pt] (10.,8.)-- (13.,8.);
\draw [line width=1.pt] (10.,7.)-- (13.,7.);
\begin{scriptsize}
\draw [fill=black] (10.,11.) circle (2.5pt);
\draw [fill=black] (10.,10.) circle (2.5pt);
\draw [fill=black] (10.,9.) circle (2.5pt);
\draw [fill=black] (10.,8.) circle (2.5pt);
\draw [fill=black] (10.,7.) circle (2.5pt);
\draw [fill=black] (11.,11.) circle (2.5pt);
\draw [fill=black] (11.,10.) circle (2.5pt);
\draw [fill=black] (11.,9.) circle (2.5pt);
\draw [fill=black] (11.,8.) circle (2.5pt);
\draw [fill=black] (11.,7.) circle (2.5pt);
\draw [fill=black] (12.,7.) circle (2.5pt);
\draw [fill=black] (12.,8.) circle (2.5pt);
\draw [fill=black] (12.,9.) circle (2.5pt);
\draw [fill=black] (12.,10.) circle (2.5pt);
\draw [fill=black] (12.,11.) circle (2.5pt);
\draw [fill=black] (13.,11.) circle (2.5pt);
\draw [fill=black] (13.,10.) circle (2.5pt);
\draw [fill=black] (13.,9.) circle (2.5pt);
\draw [fill=black] (13.,8.) circle (2.5pt);
\draw [fill=black] (13.,7.) circle (2.5pt);
\draw [fill=black] (14.,7.) circle (2.5pt);
\draw [fill=black] (14.,8.) circle (2.5pt);
\draw [fill=black] (14.,9.) circle (2.5pt);
\draw [fill=black] (14.,10.) circle (2.5pt);
\draw [fill=black] (14.,11.) circle (2.5pt);
\end{scriptsize}
\end{tikzpicture}
}
\end{center}

\vspace{-0.4cm}
\caption{ A disjunctive sum in \textsc{nimstring}.}
\label{fig7}
\end{figure}

In \cite{BCG82}, carry-on moves were called \emph{complimenting moves}. In this work, we prefer the designation carry-on move as we believe it to be much more descriptive. In Figure \ref{fig8}, we show the component $G$ of the previous example, along with its two options, one of which is a carry-on move. If the top bar is drawn, the next player continues, but if the middle bar is drawn, then the current player has to carry on.

\begin{figure}[!htb]
\begin{center}
\scalebox{0.9}{
\begin{tikzpicture}
\draw [line width=1.pt] (5.04,-4.02)-- (5.04,-5.04);
\draw [line width=1.pt] (5.04,-5.04)-- (5.04,-6.04);
\draw [line width=1.pt] (4.04,-5.04)-- (4.04,-6.04);
\draw [line width=1.pt] (4.04,-4.04)-- (4.04,-5.04);
\draw [line width=1.pt] (4.04,-6.04)-- (5.04,-6.04);
\begin{scriptsize}
\draw [fill=black] (4.04,-4.04) circle (2.5pt);
\draw [fill=black] (5.04,-4.02) circle (2.5pt);
\draw [fill=black] (5.04,-5.04) circle (2.5pt);
\draw [fill=black] (5.04,-6.04) circle (2.5pt);
\draw [fill=black] (4.04,-5.04) circle (2.5pt);
\draw [fill=black] (4.04,-6.04) circle (2.5pt);
\end{scriptsize}
\end{tikzpicture}\hspace{2 cm}
\begin{tikzpicture}
\draw [line width=1.pt] (5.04,-4.02)-- (5.04,-5.04);
\draw [line width=1.pt] (5.04,-5.04)-- (5.04,-6.04);
\draw [line width=1.pt] (4.04,-5.04)-- (4.04,-6.04);
\draw [line width=1.pt] (4.04,-4.04)-- (4.04,-5.04);
\draw [line width=1.pt] (4.04,-6.04)-- (5.04,-6.04);
\draw [line width=1.pt] (4.04,-4.02)-- (5.04,-4.02);
\begin{scriptsize}
\draw [fill=black] (4.04,-4.04) circle (2.5pt);
\draw [fill=black] (5.04,-4.02) circle (2.5pt);
\draw [fill=black] (5.04,-5.04) circle (2.5pt);
\draw [fill=black] (5.04,-6.04) circle (2.5pt);
\draw [fill=black] (4.04,-5.04) circle (2.5pt);
\draw [fill=black] (4.04,-6.04) circle (2.5pt);
\end{scriptsize}
\end{tikzpicture}
\hspace{1 cm}
\begin{tikzpicture}
\fill[gray!40!white] (4.04,-6.04) rectangle (5.04,-5.04);
\draw [line width=1.pt] (5.04,-4.02)-- (5.04,-5.04);
\draw [line width=1.pt] (5.04,-5.04)-- (5.04,-6.04);
\draw [line width=1.pt] (4.04,-5.04)-- (4.04,-6.04);
\draw [line width=1.pt] (4.04,-4.04)-- (4.04,-5.04);
\draw (5.1,-5.2) node[anchor=north west] {\rm{Carry-on}};
\draw [line width=1.pt] (4.04,-6.04)-- (5.04,-6.04);
\draw [line width=1.pt] (4.04,-5.04)-- (5.04,-5.04);
\begin{scriptsize}
\draw [fill=black] (4.04,-4.04) circle (2.5pt);
\draw [fill=black] (5.04,-4.02) circle (2.5pt);
\draw [fill=black] (5.04,-5.04) circle (2.5pt);
\draw [fill=black] (5.04,-6.04) circle (2.5pt);
\draw [fill=black] (4.04,-5.04) circle (2.5pt);
\draw [fill=black] (4.04,-6.04) circle (2.5pt);
\end{scriptsize}
\end{tikzpicture}
}
\end{center}

\vspace{-0.4cm}
\caption{ A \textsc{nimstring} component $G$ with its two options, a ``double-box'' and a carry-on.}
\label{fig8}
\end{figure}
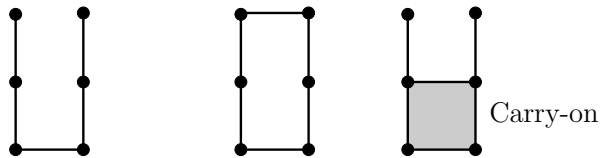

A carry-on move can be seen as a particular case of an entailing move. As previously mentioned, even when entailed, a player may have more than one option to respond. This happens when in a game form like $G^L =\{\infty\,|\,A\}$ the set $A$ has more than one element.

When $A$ has a single option, there is only one way for the player to cease being entailed. Suppose that in $G^L =\{\infty\,|\,A\}$, Right's only option is $G^{LR}$, i.e., $G^L =\{\infty\,|\,G^{LR}\}$. In this case, when Left opts for the entailing move, it is as if Left is ``jumping'' directly from $G$ to $G^{LR}$, since there will be an opportunity to \emph{play again} after replacing $G$ by $G^{LR}$ in the disjunctive sum. After the ``jump'', Left can play again on any component of the sum.

Returning to Figure \ref{fig8}, suppose Left closes a box by drawing the middle bar in $G$. The remaining position is a $3$-sided box, which is a second player win, i.e., a $\mathcal{P}$-position (zero). Since Left retains the right to play, she has to carry on in a position equal to zero. These moves, for Left and Right, are represented by the options $\{\infty\,|\,0\}$ (Left's carry-on move) and $\{0\,|\,\overline{\infty}\}$ (Right's carry-on move). If a player draws the top bar in $G$, then they place a $\mathcal{P}$-position (zero) in the disjunctive sum, but this time, passing the turn to the opponent. Thus, the game form that describes $G$ is $\{0,\{\infty\,|\,0\}\,|\,0,\{0\,|\,\overline{\infty}\}\}$.

The analysis of partizan rulesets with entailing moves, which can be of various types or only carry-on moves, is absent from the specialized literature. For example, one can consider {\sc whackenbush}, which prompts an interesting mathematical analysis. It is played like the classic {\sc blue-red-hackenbush} with an extra rule: whenever a player makes a move that drops one or more opponent's edges, they have to play again. Figure~\ref{fig9} illustrates an intriguing {\sc whackenbush} position that cannot be adequately described through an element of $\Npc$; thus, $\Np$ is invoked. The game value of the position is $\{\{\infty\,|\,0\}\,|\,0,1\}=\{\{\infty\,|\,0\}\,|\,0\}$. This negative value does not exist in $\Npc$.

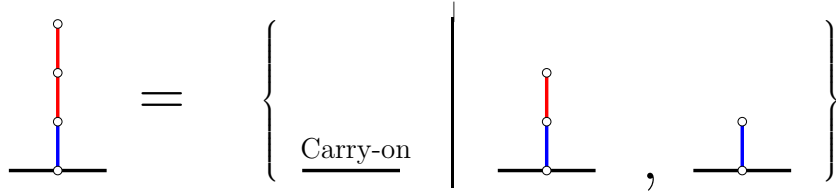
\begin{figure}[!htb]
\begin{center}
\scalebox{0.65}{
\definecolor{ffqqqq}{rgb}{1.,0.,0.}
\definecolor{qqqqff}{rgb}{0.,0.,1.}
\begin{tikzpicture}
\clip(1.5,1.5) rectangle (18.9,5.5);
\draw [line width=2.pt] (2.,2.)-- (4.,2.);
\draw [line width=2.pt] (8.,2.)-- (10.,2.);
\draw [line width=2.pt] (12.,2.)-- (14.,2.);
\draw [line width=2.pt] (16.,2.)-- (18.,2.);
\draw [line width=2.pt,color=qqqqff] (3.,2.)-- (3.,3.);
\draw [line width=2.pt,color=ffqqqq] (3.,3.)-- (3.,4.);
\draw [line width=2.pt,color=ffqqqq] (3.,4.)-- (3.,5.);
\draw [line width=2.pt,color=qqqqff] (13.,2.)-- (13.,3.);
\draw [line width=2.pt,color=ffqqqq] (13.,3.)-- (13.,4.);
\draw [line width=2.pt,color=qqqqff] (17.,2.)-- (17.,3.);
\begin{scriptsize}
\draw [color=black,fill=white] (3.,2.) circle (2.5pt);
\draw [color=black,fill=white] (3.,3.) circle (2.5pt);
\draw [color=black,fill=white] (3.,4.) circle (2.5pt);
\draw [color=black,fill=white] (3.,5.) circle (2.5pt);
\draw [color=black,fill=white] (13.,2.) circle (2.5pt);
\draw [color=black,fill=white] (13.,3.) circle (2.5pt);
\draw [color=black,fill=white] (13.,4.) circle (2.5pt);
\draw [color=black,fill=white] (17.,2.) circle (2.5pt);
\draw [color=black,fill=white] (17.,3.) circle (2.5pt);
\end{scriptsize}

\draw (4.5,3.75) node[anchor=north west] {\scalebox{3}{$=$}};
\draw (7,5.3) node[anchor=north west] {$\custombrace{3.3cm}{}$};
\draw (18.5,5.3) node[anchor=north west] {$\custombracee{3.3cm}{}$};
\draw (10.9,5.6) node[anchor=north west] {$\customslash{3.5cm}{1.5pt}$};
\draw (14.8,2.1) node[anchor=north west] {\scalebox{3}{$,$}};
\draw (7.82,2.8) node[anchor=north west] {\scalebox{1.4}{Carry-on}};
\end{tikzpicture}}
\end{center}

\vspace{-0.4cm}
\caption{ An intriguing {\sc whackenbush} position.}
\label{fig9}
\end{figure}

In summary, if, with a carry-on move, Left ``jumps'' directly from $G$ to $H$, maintaining the right to play, then this option can be described as $G^L=\{\infty\,|\,H\}$. This is because it can be seen as a Left option (a ``gadget'') that compels Right to respond with $H$, returning the right to play again to Left. In game practice, what happens  frequently is that Left moves from $G$ to $H$, playing again without any intervention from Right. With the gadget, Left moves from $G$ to $G^L$, and Right is forced to respond with $H$, allowing Left to play again. Of course, although Right intervenes, it is as if he had not, given the forced manner in which he had to do it. The carry-on option $G^L=\{\infty\,|\,H\}$ is commonly symbolized by $\circlearrowright^{H}$. Analogously, Right carry-on options $\{H\,|\,\overline{\infty}\}$ are denoted as $\circlearrowleft^{H}$. For example, the \textsc{nimstring} position shown in Figure \ref{fig7} can be represented as $\{0,\circlearrowright^{0}\,|\,0,\circlearrowleft^{0}\}$. The {\sc whackenbush} position shown in Figure \ref{fig8} can be represented as $\{\circlearrowright^{0}\,|\,0\}$.

\section{Mathematical structure $\Np$ of affine game forms}
\label{sec:structure}

In this section, we present the fundamental definitions and some initial results related to affine normal play. The first step, before delving into the structure, is to define rigorously the objects under study, which are, naturally, games. Like many others in CGT, it is a recursive definition.

\begin{definition}[Affine Games]\label{def:recursion}
Let $\mathbb{A}_{-1}=\{\infty,\overline{\infty}\}$ be the base case, the set of \emph{affine} games born on day $-1$; these are the only games in which players have no options. For $i>-1$, $\mathbb{A}_{i}$ is the set of games $\{G^\mathcal{L}\,|\,G^\mathcal{R}\}$, where $G^\mathcal{L}$ and $G^\mathcal{R}$ are non-empty finite sets of games contained in $\cup_{j<i} \mathbb{A}_j$; these are the affine games born by day $i>-1$. Moreover, if $i>-1$, $G\in \mathbb{A}_{i}$, and $G\not\in\mathbb{A}_{i-1}$, then $G$ is born on day $i$; it is said to have a \textit{formal birthday} of $i$, denoted by $\tilde{b}(G) = i$. The set of affine games is $\mathbb{A}=\cup_{i \geqslant -1} \mathbb{A}_i$.
\end{definition}

The definition of affine games deserves some considerations. We observe that, except for games of day $-1$, both players have options. In other words, affine games are dicotic, meaning either both players have options, or neither does -- this latter scenario pertains to cases where the game is a terminating game, i.e., it is either $\infty$ or $\overline{\infty}$. One might ask if the structure we are about to propose is limited by not being able to encompass rulesets with positions where one player has moves and the other does not. The answer to the question is easy: \emph{the structure also covers those rulesets}. The only thing that needs to be done is a modification by choosing $G^\mathcal{L}=\{\overline{\infty}\}$ or $G^\mathcal{R}=\{\infty\}$ for the player without moves. In practice, if the empty set is replaced with a single suicidal move, the same effect is achieved, as the existence of that move is equivalent to having no moves at all. It should also be mentioned that the dicotic nature of the structure streamlines the thinking process in many proofs, as one can assume for a non-terminating game that both players have moves. Note also that all games end with a move to $\infty$ or to $\overline{\infty}$, i.e., all games end with a ``checkmate'', which can be self-inflicted or not. In the case where the move is a suicidal move, it may seem that the name ``normal play'' is poorly chosen, as in that case the last player loses. However, this is not the case, as it is again a matter of interpretation. Before making the suicidal move, in practice, the player was deprived of satisfactory moves, and that is the fundamental reason for their loss. In other words, the structure preserves the logic of the normal play convention.

The reason why the recursion begins on day $-1$ instead of day $0$ is purely instrumental. In Section \ref{sec:cembedding}, we will prove that the classical structure $\Npc$ is order-embedded in $\Np$. Starting the recursion on day $-1$ allows a game's formal birthday to match in both structures if the game belongs to both. In a way, the infinities play a role analogous to the empty set in the classical structure, with a difference: in the classical structure, the empty set is always ``horrible'' for a player, whereas in the structure we are proposing, an infinity can be ``horrible'' or ``wonderful'', depending on its sign. Somehow, in the classical structure, the empty set ``exists'' on day $-1$, before day zero, as it is used to make the only game of day zero. Similarly, infinities exist on day $-1$. Since there are two infinities and not just one empty set, there is more than one game born by day zero. There is another difference, but it is not very crucial. In the classical structure, the empty set is a useful \emph{atom} for constructing game forms, but it is not a game itself. The infinities are games. However, in a certain sense, the infinities also serve as atoms for building game forms. This subject will be revisited in Section \ref{sec:lattice}.

Moving on to the definition of disjunctive sum, it makes sense to establish an initial parallel with what happens in the extended real number line. The disjunctive sum is well-defined except in the case where there are infinities of different signs in the sum. Regarding the game practice, this poses no issue because there is never more than one terminating move in a game, given its absorbing nature; the game ends as soon as one occurs.

\begin{definition}[Disjunctive Sum]\label{def:disjunctive}
The disjunctive sum of two affine games $G$ and $H$ is given by
\[G+ H=  \begin{cases}
\textrm{not defined}, \quad\textrm{ if }(G=\infty\textrm{ and }H=\overline{\infty})\textrm{ or }(G=\overline{\infty} \textrm{ and }H=\infty);\\
\infty, \quad\quad\quad\quad\,\,\,\,\textrm{ if }(G=\infty\textrm{ and }H\neq\overline{\infty})\textrm{ or }(G\neq\overline{\infty} \textrm{ and }H=\infty);\\
\overline{\infty}, \quad\quad\quad\quad\,\,\,\,\textrm{ if }(G=\overline{\infty}\textrm{ and }H\neq\infty)\textrm{ or }(G\neq\infty \textrm{ and }H=\overline{\infty});\\
\{G^\mathcal{L}+H,G+H^\mathcal{L}\mid G^\mathcal{R}+H,G+H^\mathcal{R}\}, \quad\textrm{ otherwise.}
\end{cases}
\]
\end{definition}

\begin{observation}
The definition of disjunctive sum is recursive and is presented with some abuse of language: $G^\mathcal{L}+H,G+H^\mathcal{L}$ means $\{G^L+H\,|\,G^L\in G^\mathcal{L}\}\cup \{G+H^L\,|\,H^L\in H^\mathcal{L}\}$ (and the same for Right's options).
\end{observation}

A second parallel concerns the following theorem. The extended real number line is ``almost'' an abelian group, with only one issue when infinities are involved. Similarly, the affine normal play structure is ``almost'' an abelian monoid.

\begin{theorem}\label{thm:almostmonoid} Let $G$, $H$, and $J$ be affine games. We have the following:
\begin{itemize}
  \item $G + (H + J)$ and $(G + H) + J$ are either equal or both undefined;
  \item $G + H$ and $H + G$ are either equal or both undefined.
\end{itemize}
\end{theorem}

\begin{proof} The items are established as follows.
\begin{itemize}
  \item If one of the summands is $\infty$ and one of the others is $\overline{\infty}$, it is easy to verify that both $G + (H + J)$ and $(G + H) + J$ are undefined. If one of the summands is $\infty$ without either of the others being $\overline{\infty}$, then both $G + (H + J)$ and $(G + H) + J$ are equal to $\infty$. If one of the summands is $\overline{\infty}$ without either of the others being $\infty$, then both $G + (H + J)$ and $(G + H) + J$ are equal to $\overline{\infty}$. Thus, assume that none of $G$, $H$, and $J$ are infinities. In that case,
      \begin{align*}
    \left(G+(H+J)\right)^\mathcal{L} & =\left(G^\mathcal{L}+(H+J)\right)\cup \left(G+(H+J)^\mathcal{L}\right) \quad &&\\
         & =\left(G^\mathcal{L}+(H+J)\right)\cup \left(G+(H^\mathcal{L}+J)\right)\cup \left(G+(H+J^\mathcal{L})\right) \quad &&\\
    \text{(Induction)} &=\left((G^\mathcal{L}+H)+J\right)\cup \left((G+H^\mathcal{L})+J\right)\cup \left((G+H)+J^\mathcal{L}\right)&& \\
 &=\left((G+H)^\mathcal{L}+J\right)\cup \left((G+H)+J^\mathcal{L}\right)&& \\
     &= \left((G+H)+J\right)^\mathcal{L}. \quad &&
\end{align*}
Note that the induction works well, given that the sums are defined because none of $G$, $H$, and $J$ are infinities. The fact that $\left(G+(H+J)\right)^\mathcal{R}=\left((G+H)+J\right)^\mathcal{R}$ is proven in a similar way. Hence, $G + (H +J)=\{(G+(H+J))^\mathcal{L}\,|\,(G+(H+J))^\mathcal{R}\}=\{((G+H)+J)^\mathcal{L}\,|\,((G+H)+J)^\mathcal{R}\}=(G+H)+J$.
  \item If one of the summands is $\infty$ and the other is $\overline{\infty}$, then both $G + H$ and $H + G$ are undefined. If one of the summands is $\infty$ without the other being $\overline{\infty}$, then both $G + H$ and $H + G$ are equal to $\infty$. If one of the summands is $\overline{\infty}$ without the other being $\infty$, then both $G + H$ and $H + G$ are equal to $\overline{\infty}$. Thus, assume that neither $G$ nor $H$ are infinities. In that case,
      \begin{align*}
    G+H &= \{G^\mathcal{L}+H,G+H^\mathcal{L}\mid G^\mathcal{R}+H,G+H^\mathcal{R}\} \quad &&\\
     &= \{H^\mathcal{L}+G,H+G^\mathcal{L}\mid H^\mathcal{R}+G,H+G^\mathcal{R}\} \quad &&\text{(Induction)} \\
     &= H + G. \quad &&
\end{align*}
Note that the induction works well, given that the sums are defined because neither $G$ nor $H$ are infinities.
\end{itemize}
\end{proof}

In order to define an equivalence relation and a partial order of games, it is necessary to first define the concept of outcomes. This is done in the usual manner.

\begin{definition}[Individualized Outcomes]\label{def:individualized}
The \emph{individualized outcomes} of $G$ when Left plays first and when Right plays first are recursively defined by
\[o^L(G)=  \begin{cases}
\mathbf{L}, \quad\textrm{ if }G=\infty\textrm{ or }\exists G^L\in G^\mathcal{L}\textrm{ such that }o^R(G^L)=\mathbf{L};\\
\mathbf{R}, \quad\textrm{ otherwise.}
\end{cases}
\]
\[o^R(G)=  \begin{cases}
\mathbf{R}, \quad\textrm{ if }G=\overline{\infty}\textrm{ or }\exists G^R\in G^\mathcal{R}\textrm{ such that }o^L(G^R)=\mathbf{R};\\
\mathbf{L}, \quad\textrm{ otherwise.}
\end{cases}
\]
\end{definition}

\begin{observation}
$\mathbf{L}$ and $\mathbf{R}$ mean ``Left wins'' and ``Right wins'', respectively. Note that Definition \ref{def:individualized} formalizes the idea of \emph{optimal play} in that, having that move at their disposal, a player chooses a winning move. In CGT, the number of moves in which a player wins a game is not relevant; what matters is winning. In other words, speed is irrelevant.
\end{observation}

\begin{definition}[Outcomes]\label{def:outcomes}
The {\em combined outcome}, or simply the \emph{outcome} of a game $G$, is the pair $(o^L(G),o^R(G))$ which can be one of the four possibilities $(\mathbf{L},\mathbf{L})$, $(\mathbf{L},\mathbf{R})$, $(\mathbf{R},\mathbf{L})$, and $(\mathbf{R},\mathbf{R})$. Traditionally, these four possibilities are denoted as $\mathscr{L}$, $\mathscr{N}$, $\mathscr{P}$, and $\mathscr{R}$.
\end{definition}

From Left's perspective\footnote{In CGT, Left is the ``positive player'', and Right is the ``negative player'': Left always aims for the highest possible, while Right seeks the lowest possible. Given that in two-player games, each pulls in their direction, this convention aligns with game practice.}, $\mathscr{L}$ is the best outcome (she wins, regardless of whether playing first or second) and $\mathscr{R}$ is the worst (she loses, regardless of whether playing first or second). On the other hand, regarding $\mathscr {N}$ and $\mathscr{P}$, the victory depends on playing first or second, so these outcomes are not comparable. These considerations explain how the partial order of outcomes is conventionally established:
\begin{center}
\begin{tikzpicture}[node distance=1cm]
  \node (L) at (0,2) {$\mathscr{L}$};
  \node (P) at (-1,1) {$\mathscr{P}$};
  \node (N) at (1,1) {$\mathscr{N}$};
  \node (R) at (0,0) {$\mathscr{R}$};

  \draw (L) -- (P);
  \draw (L) -- (N);
  \draw (P) -- (R);
  \draw (N) -- (R);
\end{tikzpicture}
 \end{center}

 The outcome function $o(G)$ is used to denote the outcome of $G$. The {\em outcome classes} $\L, \N, \R, \P$ are the sets of all games with the indicated outcome, so that we can write $G\in \L$ when $o(G)=\mathscr{L}$.

 In accordance with what has already been mentioned, when Left wins the game, the last move is a move to $\infty$, made by one of the two players. Of course, Right only makes this move when he has no other option. Analogously, when Right wins the game, the last move is a move to $\overline{\infty}$.

 With a partial order of outcomes defined, it makes sense to write $o(G)=o(H)$, $o(G)\geqslant o(H)$ or $o(G)\,\|\, o(H)$. Up to this point, the symbol $=$ has been used exclusively to mean ``by definition''; from now on, it will also be used to indicate the equality of outcomes. The distinction will be made by the context.

 With the concept of outcome established, it is now possible to define the equivalence and partial order of games.

\begin{definition}[Equivalence and Order]\label{def:orderequivalence}
Let $G$ and $H$ be affine games. We have the following:
\begin{itemize}
  \item $G\sim H  \textrm{ if and only if } o(G+X)= o(H+X) \textrm{ for all games } X\in \mathbb{A}\setminus \{\infty,\overline{\infty}\}.$
  \item $G\gtrsim H  \textrm{ if and only if } o(G+X)\geqslant o(H+X) \textrm{ for all games } X\in \mathbb{A}\setminus \{\infty,\overline{\infty}\};$
\end{itemize}
\end{definition}

\begin{observation}
The relation $G\gtrsim H$ means that replacing $H$ by $G$ can never hurt Left, no matter what the disjunctive sum is; the relation $G\sim H$ means that replacing $H$ by $G$ can never hurt Left or Right in any disjunctive sum.
\end{observation}

\begin{theorem}\label{thm:relations} The relations $\sim$ and $\gtrsim$ are, respectively, an equivalence relation and a partial order relation. \end{theorem}

\begin{proof}
The facts that $\sim$ is reflexive and symmetric are direct consequences of Definition~\ref{def:orderequivalence}. Suppose now that $G\sim H$ and $H\sim J$. For every $X\in \mathbb{A}\setminus \{\infty,\overline{\infty}\}$, it is a consequence of $G\sim H$ that $o(G+X)= o(H+X)$. On the other hand, it is a consequence of $H\sim J$ that $o(H+X)= o(J+X)$. Consequently, as we have $o(G+X)=o(H+X)= o(J+X)$, it follows that $G\sim J$ and $\sim $ is also transitive.

The fact that $\gtrsim$ is reflexive is a direct consequence of Definition \ref{def:orderequivalence}. Suppose now that $G\gtrsim H$ and $H\gtrsim G$. For every $X\in \mathbb{A}\setminus \{\infty,\overline{\infty}\}$, it is a consequence of $G\gtrsim H$ that $o(G+X)\geqslant o(H+X)$. On the other hand, it is a consequence of $H\gtrsim G$ that $o(H+X)\geqslant o(G+X)$. Consequently, as we must have $o(G+X)= o(H+X)$, it follows that $G\sim H$ and $\gtrsim$ is antisymmetric. Finally, suppose that $G\gtrsim H$ and $H\gtrsim J$. For every $X\in \mathbb{A}\setminus \{\infty,\overline{\infty}\}$, it is a consequence of $G\gtrsim H$ that $o(G+X)\geqslant o(H+X)$. On the other hand, it is a consequence of $H\gtrsim J$ that $o(H+X)\geqslant o(J+X)$. Thus, as we have $o(G+X)\geqslant o(H+X)\geqslant o(J+X)$, it follows that $G\gtrsim J$ and $\gtrsim$ is transitive.
\end{proof}

\noindent
\textbf{Notation}: From now on, the structure $(\mathbb{A},+,\sim,\gtrsim)$ is simply designated as $\Np$.  For ease, we will use the symbol $\geqslant$ instead of $\gtrsim$, and $=$ instead of $\sim$; different situations determine if the symbols are being used for definitions, outcomes or games. Furthermore, $G>H$ means that $G\geqslant H$ and $G\neq H$. When two games have exactly the same game form, they are \emph{isomorphic}, denoted as $G\cong H$. Naturally, if $G$ is isomorphic to $H$, $G$ is equal to $H$ in terms of game equivalence, but the opposite direction may not be true.

\section{Fundamental Theorem of Normal Play and the partial order of games}
\label{sec:ftanp}

We begin this section dedicated to some initial results on the partial order of games by clarifying what is the identity of $\Np$ with respect to the disjunctive sum.

\begin{theorem}\label{thm:identity} The game $\{\overline{\infty}\,|\,\infty\}$, denotated as $0$, is the identity of $\Np$. \end{theorem}

\begin{proof}
According to Theorem \ref{def:disjunctive}, $G+\{\overline{\infty}\,|\,\infty\}$ commutes, thus it suffices to prove that $G+\{\overline{\infty}\,|\,\infty\}=G$.

If $G=\infty$, according to Definition \ref{def:disjunctive}, $G+\{\overline{\infty}\,|\,\infty\}=G=\infty$. If $G=\overline{\infty}$, according to Definition \ref{def:disjunctive}, $G+\{\overline{\infty}\,|\,\infty\}=G=\overline{\infty}$. Hence, assume that $G\neq\infty$ and $G\neq\overline{\infty}$. Let $X\in\Np\setminus\{\infty,\overline{\infty}\}$.

Suppose that, playing first, a player wins $G+X$. Without loss of generality, assume that the player is Left, and she wins with some option $(G+X)^L$. Then, she also wins $G+X+\{\overline{\infty}\,|\,\infty\}$ with the option $(G+X)^L+\{\overline{\infty}\,|\,\infty\}$. Essentially, she mimics the strategy used when $G+X$ is played alone, since neither player can make a move in $\{\overline{\infty}\,|\,\infty\}$ without immediately leading to defeat. On the other hand, if, playing first, Left wins $G+X+\{\overline{\infty}\,|\,\infty\}$, this must be done with an option $(G+X)^L+\{\overline{\infty}\,|\,\infty\}$. She also wins $G+X$ with the option $(G+X)^L$, since, once more, one can never touch the component $\{\overline{\infty}\,|\,\infty\}$, and the mimicry can be employed again. The argument is entirely analogous from Right's perspective, thus $o(G+X+\{\overline{\infty}\,|\,\infty\})= o(G+X)$, implying $G+\{\overline{\infty}\,|\,\infty\}=G$.
\end{proof}

The next theorem follows directly from the absorbing nature of the infinities.

 \begin{theorem}\label{thm:inflarge}
If $G\in\Np$, then $\infty\geqslant G$ and $G\geqslant\overline{\infty}$.
\end{theorem}

\begin{proof}
If $X\in \Np\setminus\{\infty,\overline{\infty}\}$ then, according to Definition \ref{def:disjunctive}, $\infty+X=\infty$. Hence, by Definition \ref{def:individualized}, we have $o(\infty+X)=o(\infty)=\mathscr{L}$. Thus, for every $X\in \Np\setminus\{\infty,\overline{\infty}\}$,  we have $o(\infty+X)\geqslant o(G+X)$, implying $\infty\geqslant G$. Proving that $G\geqslant\overline{\infty}$ is analogous.
\end{proof}

Regarding the classical structure $\Npc$, it is a well-known fact that $G\geqslant 0$ if and only if $G\in\L\cup\P$ -- here, we refer to this result as the Fundamental Theorem of Normal Play. The ultimate reason for this theorem is that if Left has a winning strategy in a game $X$, she also has it in $G+X$. She can use a \emph{local response strategy}, meaning that in $G+X$, Left responds to Right's moves in a component with a move in that same component, as if she were playing it in isolation. Under the normal play convention, Left makes the last move in both components and, consequently, in the disjunctive sum as a whole. The following theorem states that this result remains valid in $\Np$, providing yet another reason to support the choice of the name ``affine normal play''.

\begin{theorem}[Fundamental Theorem of Affine Normal Play]\label{thm:ftnp}
Let $G\in \Np$. We have that $G\geqslant 0$ if and only if $G\in \mathcal{L}\cup\mathcal{P}$.
\end{theorem}

\begin{proof} Assume that $G\geqslant 0$. We have $0 \in \mathcal{P}$, and so, by order of outcomes, $G\in \mathcal{L}\cup\mathcal{P}$.

Suppose now that $G\in \mathcal{L}\cup\mathcal{P}$. If $G=\infty$, according to Theorem \ref{thm:inflarge}, $G\geqslant 0$; hence, assume $G\neq\infty$ and let $X\in\Np\setminus\{\infty,\overline{\infty}\}$. If, playing first, Left wins $X$ with the option $X^L$, then she also wins $G+X$ with the option $G+X^L$. Essentially, she mimics the strategy used when $X$ is played alone, answering locally when Right plays in $G$. Due to the assumption $G\in \mathcal{L}\cup\mathcal{P}$, this is a winning strategy for Left in $G+X$. If Left, playing second, wins $X$. Then, on $G+X$, she can respond to each of Right's moves locally, with a winning move on the same component, because $G\in \mathcal{L}\cup\mathcal{P}$. Thus Left can win $G+X$ playing second. Consequently, $o(G+X)\geqslant o(X)$, implying $G\geqslant 0$.
\end{proof}

Once again, as is the case in $\Npc$, Theorem \ref{thm:ftnp} is sufficient to prove that there is a total matching between the order relation and the outcome classes.

\begin{corollary}[Order-Outcome  Bijection]\label{cor:ftnp2}
If $G\in \Np$ then
\begin{itemize}
  \item $G>0$ if and only if $G\in \mathcal{L}$;
  \item $G=0$ if and only if $G\in \mathcal{P}$;
  \item $G\cgfuzzy 0$ if and only if $G\in \mathcal{N}$;
  \item $G<0$ if and only if $G\in \mathcal{R}$.
\end{itemize}
\end{corollary}

\begin{proof} In this proof, we also use the dual version of Theorem~\ref{thm:ftnp}: ``$G\leqslant0$ if and only if $G\in\mathcal{R}\cup\mathcal{P}$''.

Suppose that $G>0$. By Theorem \ref{thm:ftnp}, $ G\in\mathcal{L}\cup\mathcal{P}$. But, we cannot have $G\in\mathcal{P}$, for otherwise $G\in\mathcal{R}\cup\mathcal{P}$ and $G\leqslant 0$. Therefore, $G\in\mathcal{L}$. Conversely, suppose that $G\in\mathcal{L}$. By Theorem \ref{thm:ftnp}, we have $G\geqslant 0$. But, we cannot have $G=0$, for otherwise $G\leqslant 0$, and  $G\in\mathcal{R}\cup\mathcal{P}$. Hence, $G> 0$. Thus, the first equivalence holds.

The proof of the fourth equivalence is analogous.

For the second equivalence, we have that if $G=0$, then $G\geqslant 0 \;\wedge\; G\leqslant 0$. Therefore, $G\in(\mathcal{L}\cup\mathcal{P})\cap (\mathcal{R}\cup\mathcal{P})=\mathcal{P}$.

The third equivalence is a consequence of eliminating all other possibilities.
\end{proof}

The subsequent theorems offer a set of valuable results that relate the partial order of games to the disjunctive sum.

\begin{theorem}\label{thm:compadd}
If $G,H\in \Np$ and $J\in \Np\setminus\{\infty,\overline{\infty}\}$, then $$G\geqslant H\implies G+J\geqslant H+J.$$
\end{theorem}

\begin{proof}
Consider any $X\in \Np\setminus\{\infty,\overline{\infty}\}$ and let $X'=J+X$. Since $X$ and $J$ are neither $\infty$ nor $\overline{\infty}$, we have that $X'$ is neither $\infty$ nor $\overline{\infty}$.  Definition~\ref{def:orderequivalence} implies $o(G+X')\geqslant o(H+X')$, that is $o(G+J+X)\geqslant o(H+J+X)$. Thus, the arbitrariness of $X$ implies $G+J\geqslant H+J$.
\end{proof}

\begin{theorem}\label{thm:compinvertible}
If $G,H,J\in\Np$ and $J$ is an invertible element, then $$G\geqslant H\iff G+J\geqslant H+J.$$
\end{theorem}

\begin{proof}
Firstly, note that the fact that $J$ is invertible implies that $J$ is neither $\infty$ nor $\overline{\infty}$, since, given the absorbing nature of infinities, there cannot be any element that, when added to these elements, equals $0$. Thus, the implication $G\geqslant H\implies G+J\geqslant H+J$ directly follows from Theorem \ref{thm:compadd}.

Regarding the reciprocal implication, consider any $X\in \Np\setminus\{\infty,\overline{\infty}\}$ and let\linebreak $X'=-J+X$ ($J$ is invertible, i.e. $-J $ exists and $J+(-J)=0$). Since $-J$ is invertible, $-J$ is neither $\infty$ nor $\overline{\infty}$, and so, $X'$ is neither $\infty$ nor $\overline{\infty}$. By Definition~\ref{def:orderequivalence}, $o(G +J+X')\geqslant o(H+J+X')$, that is $o(G+J-J+X)\geqslant o(H+J-J+X)$. Hence, $o(G+X)\geqslant o(H+X)$, and so, given the arbitrariness of $X$, we have $G\geqslant H$. \end{proof}

The following theorem is highly relevant as it establishes that the comparison between two games, one of which is invertible, can be conducted by playing, exactly in the same manner as in the classical structure. In Section \ref{sec:cembedding}, an elegant characterization of invertible elements will be presented, reinforcing the significance of the theorem.

\begin{theorem}\label{thm:invcomparison}
If $G,H\in \Np$ and $H$ is an invertible element, then $$G\geqslant H\iff G+(-H)\in\mathcal{L}\cup\mathcal{P} \text{ and }G=H\iff G+(-H)\in\mathcal{P}.$$
\end{theorem}

\begin{proof} By Theorem \ref{thm:compinvertible}, $G\geqslant H \iff G+(-H) \geqslant H+(-H)$. Therefore, we have $G\geqslant H\iff G+(-H)\geqslant 0$. By Theorem \ref{thm:ftnp}, this leads to $G\geqslant H\iff G+(-H) \in \mathcal{L} \cup \mathcal{P}$. Finally, $G=H\iff G+(-H)\in\mathcal{P}$, since $G\geqslant H$ and $H\leqslant G$.
\end{proof}

The two following theorems are straightforward consequences of the fact that the order and equivalence of affine games constitute a order relation and an equivalence relation.

\begin{theorem}\label{thm:disjunctiveorder}
If $G,H,J,W\in \Np\setminus\{\infty,\overline{\infty}\}$, then we have the following:
\begin{itemize}
  \item $G\geqslant H\text{ and }J\geqslant W \implies G+J\geqslant H+W$;
  \item $G=H\text{ and }J=W \implies G+J=H+W$.
\end{itemize}
\end{theorem}

\begin{proof} Starting with the first implication, it follows from Theorem \ref{thm:compadd} that $G+J\geqslant H+J$ and $H+J\geqslant H+W$. Next, it follows from the fact that we have an order relation that $G+J\geqslant H+W$. The second implication is a trivial consequence of the first one.
\end{proof}

\begin{theorem}\label{thm:strict}
Let $G,H\in \Np$. If $G> 0$ and $H\geqslant 0$ then $G+H> 0$.
\end{theorem}

\begin{proof}
If either $G$ or $H$ is $\infty$, the theorem holds trivially. Assuming that this is not the case, by Theorem~\ref{thm:disjunctiveorder}, we already know that $G+H\geqslant 0$. So, it is enough to show that $G+H\neq 0$.
Since $G> 0$ then, without loss of generality, we may assume that
 $o^L(G+X)=\mathbf{L}$ and $o^L(X)=\mathbf{R}$ for some $X$ different from $\infty$ and $\overline{\infty}$. Because $H\geqslant 0$,
 we have $o^L(G+X+H)\geqslant o^L(G+X+0)=o^L(G+X)=\mathbf{L}$. Since  $o^L(G+H+X)=\mathbf{L}$ and $o^L(X)=\mathbf{R}$, we have $G+H\neq 0$.
\end{proof}

We conclude this section with a very simple theorem that encapsulates the idea that having extra options can never be a disadvantage.

\begin{theorem}[Monotone Principle]\label{thm:greediness}
If $G\in \Np\setminus\{\infty,\overline{\infty}\}$, then, for any $H\in\Np$, we have $\{\GL\cup \{H\}\mid \GR\}\geqslant G$.
\end{theorem}

\begin{proof} Let $G'=\{\GL\cup \{H\}\mid \GR\}$ and consider any $X\in \Np\setminus\{\infty,\overline{\infty}\}$. If Left has a winning strategy in $G+X$ (playing first or second), then Left also has a winning strategy in $G'+X$ since she never has to use the extra option, simply sticking to the winning strategy that exists in $G+X$.
\end{proof}

\section{Options-only comparison}
\label{sec:optionsonly}

As mentioned earlier, it is possible to compare two elements $G$ and $H$ of the classical structure $\Npc$ using only their followers. This is done through an options-only procedure, which consists of playing $G+(-H)$ to check if $G\geqslant H$. It is known that $G\geqslant H$ if and only if, playing second, Left wins $G+(-H)$.

In this section, we will see that the options-only comparison in $\Np$ is not as straightforward. In fact, unlike in the majority of CGT monoids, it is possible to have $G,H\in\Np$ such that $G\geqslant H$, without the maintenance property, as established in Definition \ref{def:maintenance}, being satisfied. This fact is due to the forcing nature of checks and sequences of checks.

Nevertheless, it is possible to obtain an options-only procedure using a weaker version of the maintenance property, in order to adress the effect of the forcing sequences. Thus, we begin this section by presenting an example of two games $G$ and $H$ such that $G\geqslant H$ without the traditional maintenance property being satisfied. Subsequently, still at an intuitive level, we propose the {\sc inquisitor game}, a compound game that allows comparing $G$ with $H$ by playing, using only their followers, as done in the classical structure. We conclude the section with a suitable formalization and rigorous proofs.

\subsection{A challenging example}
\label{ssec:challenging}

\vspace{0.3cm}

Let $G=\left\{*3,\{\infty\,|\,\{0\,|\,\overline{\infty}\},\{*2\,|\,\overline{\infty}\}\}\,\|\,0\right\}$ and $H=\left\{\{1\,|\,\{0\,|\,\overline{\infty}\},\{*2,*3\,|\,\overline{\infty}\}\}\,\|\,0\right\}$. In these game forms, $0$ is the game $\{\overline{\infty}\,|\,\infty\}$, $1$ is the game $\{0\,|\,\infty\}$, and the nimbers are the games defined recursively from $0$ in the traditional way.\\

Claim 1: We have $G\geqslant H$.\\

\emph{Proof of Claim 1}: Let $X\in \Np\setminus\{\infty,\overline{\infty}\}$. If Right wins $G+X$ by moving to $G+X^R$, then, by induction, he wins $H+X$ by moving to $H+X^R$. On the other hand, if Right wins $G+X$ by moving to $0+X$, then the exact same move is available in $H+X$. Therefore, for every $X\in \Np\setminus\{\infty,\overline{\infty}\}$, if Right wins $G+X$ when playing first, then Right also wins $H+X$.\\

If Left wins $H+X$ by moving to $H+X^L$, then, by induction, she wins $G+X$ by moving to $G+X^L$. The sensible option happens when Left wins $H+X$ by moving to $H^L+X$. In that case, she must have a winning move against both Right-checks $\{0\,|\,\overline{\infty}\}+X$ and $\{*2,*3\,|\,\overline{\infty}\}+X$. The first implies that $0+X$ is a winning move for Left, while the second implies that either $*2+X$ or $*3+X$ is a winning move for Left. Hence, we conclude that either both $0+X$ and $*2+X$ are winning moves for Left or both $0+X$ and $*3+X$ are winning moves for Left. In the latter case, Left, playing first, secures a win in $G+X$ by moving to $*3+X$. In the former case, Left, playing first, wins $G+X$ by moving to $\{\infty\,|\,\{0\,|\,\overline{\infty}\},\{*2\,|\,\overline{\infty}\}\}+X$. Thus, for every $X\in \Np\setminus\{\infty,\overline{\infty}\}$, if Left wins $H+X$ when playing first, then she also wins $G+X$.\\

Claim 2: The maintenance property (Definition \ref{def:maintenance}) is not satisfied.\\

\emph{Proof of Claim 2}: Given the games $G$ and $H$, we have\\

$G^{L_1}=*3$, and $G^{L_2}=\{\infty\,|\,\{0\,|\,\overline{\infty}\},\{*2\,|\,\overline{\infty}\}\}$,\\

$H^{LR_1}=\{0\,|\,\overline{\infty}\}$, and $H^{LR_2}=\{*2,*3\,|\,\overline{\infty}\}$.\\

We will show that $G^{L_1}\not\geqslant H^L$, $G^{L_2}\not\geqslant H^L$, $G\not\geqslant H^{LR_1}$, and $G\not\geqslant H^{LR_2}$.\\

$G^{L_1}\not\geqslant H^L$ because, playing first, Left loses $G^{L_1}-1=*3-1$,\\

\vspace{-0.3cm}
\hspace*{\fill} and wins $H^L-1=\{1\,|\,\{0\,|\,\overline{\infty}\},\{*2,*3\,|\,\overline{\infty}\}\}-1.$\\

$G^{L_2}\not\geqslant H^L$ because, playing second, Left loses\\

\vspace{-0.3cm}
\hspace{2.5cm} $G^{L_2}+\{0\,|\,*2\}=\{\infty\,|\,\{0\,|\,\overline{\infty}\},\{*2\,|\,\overline{\infty}\}\}+\{0\,|\,*2\},$\\

\vspace{-0.3cm}
 \hspace*{\fill} and wins $H^L+\{0\,|\,*2\}=\{1\,|\,\{0\,|\,\overline{\infty}\},\{*2,*3\,|\,\overline{\infty}\}\}+\{0\,|\,*2\}.$\\

$G\not\geqslant H^{LR_1}$ because, playing first, Left loses\\

\vspace{-0.3cm}
\hspace{3.5cm}  $G=\left\{*3,\{\infty\,|\,\{0\,|\,\overline{\infty}\},\{*2\,|\,\overline{\infty}\}\}\,\|\,0\right\},$\\

\vspace{-0.3cm}
 \hspace*{\fill} and wins $H^{LR_1}=\{0\,|\,\overline{\infty}\}.$

$G\not\geqslant H^{LR_2}$ because, playing first, Left loses\\

\vspace{-0.3cm}
\hspace{2.5cm}
$G+*2=\left\{*3,\{\infty\,|\,\{0\,|\,\overline{\infty}\},\{*2\,|\,\overline{\infty}\}\}\,\|\,0\right\}+*2,$\\

\vspace{-0.3cm}
 \hspace*{\fill} and wins $H^{LR_2}+*2=\{*2,*3\,|\,\overline{\infty}\}+*2.$

This finishes the proof of Claim 2.\\

\subsection{{\sc inquisitor game}}
\label{ssec:inquisitor}

\vspace{0.3cm}
The results that will be proved next establish that in affine normal play, to check if $G\geqslant H$ ($G$ and $H$ not infinities), it is enough to verify that Left, playing second, wins $G \uplus (-H)$, where $\uplus$ is ``a kind of'' disjunctive sum, but with two extra rules. Just like in a disjunctive sum, players choose one of the two components to play, leaving the other untouched. The two extra rules are related to two specificities of affine normal play, namely the possibility of forcing sequences, and will now be discussed at an intuitive level. Note that $\uplus$  is not the disjunctive sum as defined in Definition \ref{def:disjunctive}; it is a compound operator whose sole purpose is to compare games.

The \emph{conjugate} of a given game switches the roles of the players. In the game that will be proposed, $-H$ is the conjugate of $H$. We will see in Section \ref{sec:cembedding} that $\Np$ does not deviate from the usual rule: when a game is invertible, its inverse is its conjugate.

\begin{definition}[Conjugate]\label{def:conjugate}
The conjugate of $H\in \Np$ is
\[ -H\, =
\begin{cases}
\overline{\infty}, &\mbox{if $H=\infty$}\\
\infty, &\mbox{if $H=\overline{\infty}$}\\
\{-H^\mathcal{R}\,|\,-H^\mathcal{L}\} &\mbox{otherwise},
\end{cases}
\]
where $-H^\mathcal{R}$ denotes the set of games
$-H^R$, for $H^R\in H^\mathcal{R}$, and similarly for $-H^\mathcal{L}$.
\end{definition}

Recall that in order to have $G\geqslant H$, it is necessary to have $o^L(G + X)\geqslant o^L(H + X)$ and $o^R(G + X)\geqslant o^R(H + X)$ for every game $X$ different from $\infty$ and $\overline{\infty}$. In other words, if Left, playing first, wins $H+X$, she must also win $G+X$; if Right, playing first, wins $G+X$, he must also win $H+X$. Often, moves in $X$ are covered by simple induction. Therefore, the sensible moves are in $G$ or in $H$.

Now, suppose we are playing $G \uplus (-H)$, with Right playing first. The first additional rule stipulates that whenever Right moves to $\overline{\infty}$ in one of the components, Left still has \emph{one last opportunity} to respond with $\infty$ in the other component. If Left can still do so on the immediately next move, she wins the game; otherwise, she loses the game. To understand why this rule is imperative, let us assume that Right has an option $G^R=\overline{\infty}$. Of course, playing first, Right wins $G+X$. In this case, for Right to also win $H+X$ for any $X$, it is absolutely necessary that there exists a Right option $H^R=\overline{\infty}$ as well. A symmetrical situation occurs when there is a Left option $H^L=\infty$; it is absolutely necessary that there also exists a Left option $G^L=\infty$. As an example, consider again the infinitely hot game $G=\{\infty\,|\,\overline{\infty}\}$. It is clear that if $H=\{\infty\,|\,\overline{\infty}\}$, then $G\geqslant H$; in fact, in this case, $G$ and $H$ are equal and isomorphic! In the game $G \uplus (-H) = \{\infty\,|\,\overline{\infty}\} \uplus \{\infty\,|\,\overline{\infty}\}$, Left only wins against Right's move to $\overline{\infty} \uplus \{\infty\,|\,\overline{\infty}\}$ due to the extra rule and to the final move to $\overline{\infty} \uplus \infty$.

Understanding the second extra rule at an intuitive level is a bit more delicate and requires some preliminary definitions.

\begin{definition}[Check Games] If $\infty$ is a Left option of a game, then that game is a \emph{Left-check}. A Left option of a game $G$ that is a Left-check is denoted by $G^{\overrightarrow{L}}$. Moreover, a \emph{Right-check} and the notation $G^{\overleftarrow{R}}$ are defined symmetrically.  If a game is either a Left-check or a Right-check, then it may be simply referred to as a \emph{check}.
\end{definition}

\begin{definition}[Quiet Games]\label{def:quiet}
Let $G\in\Np$. If $G\neq \infty$ and $G$ is not a Left-check, then $G$ is \emph{Left-quiet}. The definition of a \emph{Right-quiet} game is similar. If $G$ is Left-quiet and Right-quiet then $G$ is \emph{quiet}.
\end{definition}

\begin{definition}[Silent Games]\label{def:silent}
Let $G\in\Np$. If $G$ is not a Left-check, then $G$ is \emph{Left-silent}. The definition of a \emph{Right-silent} game is done in a similar manner. If $G$ is Left-silent and Right-silent then $G$ is \emph{silent}.
\end{definition}

\begin{observation}
A quiet-game cannot be $\infty$ or $\overline{\infty}$. A silent-game may be $\infty$ or $\overline{\infty}$. A Left option $G^L$ can be either a check or a silent game; a Left option $G^{\overrightarrow{L}}$ is a Left-check.
\end{observation}

Let us suppose that Left, playing first, wins $H+X$ with a move $H^L+X$. It may happen that Right has forcing sequences with checks in the first component, against which Left has to defend. Left's responses at each moment take the form $H^{L\overrightarrow{R}\ldots\overrightarrow{R}L}+X$ and \emph{must be winning responses}, otherwise, the initial move $H^L+X$ would not be a winning move. Right may have many forcing sequences $H^{L\overrightarrow{R_1}\ldots}+X$, $H^{L\overrightarrow{R_2}\ldots}+X$,$\,\ldots$, and Left always must be able to respond victoriously. It is crucial to observe that both players may have a choice. Against each check, Left must chose a winning response; she has a victorious behavior against checks, a strategy of responses. However, \emph{it is Right who may choose the moment to stop giving checks}; he may do so immediately with a silent move $H^{LR}+X$, but he eventually may choose to do so after some checks with a silent move $H^{L\overrightarrow{R}\ldots\overrightarrow{R}LR}+X$.

Now, in order to have $G\geqslant H$, there must also be a winning move for Left in $G+X$. Once again, it can be directly a silent move to $G^L+X$, but it can also be a silent move $G^{\overrightarrow{L}R\ldots\overrightarrow{L}RL}+X$ after a sequence of Left checks and the responses that Right gives, always avoiding a Left winning move on the second component (if Right is unable to respond in that way, no issue is raised). We are now ready to mention the fundamental difference between classical normal play and affine normal play. In the first case, all moves are silent moves. Thus, for each $H^L$, either there is a silent move $H^{LR}$ such that $G\geqslant H^{LR}$ or there is a silent move $G^L$ such that $G^L\geqslant H^L$. In affine normal play, it is Left who interrupts the forcing sequences in $G$, and it is Right who interrupts the forcing sequences in $H^L$. It is sufficient that there is always a pair of games $G^{\overrightarrow{L}R\ldots\overrightarrow{L}R}$ and $H^{L\overrightarrow{R}\ldots\overrightarrow{R}L}$ where the forcing sequences end with a silent move $G^{\overrightarrow{L}R\ldots\overrightarrow{L}R}\geqslant H^{L\overrightarrow{R}\ldots\overrightarrow{R}LR}$ or with a silent move $G^{\overrightarrow{L}R\ldots\overrightarrow{L}RL}\geqslant H^{L\overrightarrow{R}\ldots\overrightarrow{R}L}$, for $G\geqslant H$ to hold. This is a weaker version of the classical maintenance property: \emph{it is not necessary to have immediately $G\geqslant H^{LR}$ or $G^L\geqslant H^L$. Achieving the effect after sequences of checks is sufficient}.

Notice that saying it is Right who chooses the moment to stop giving checks in $H^L$ is the same as saying it is Left who chooses the moment to stop giving checks in $-H^L$. This explains the second extra rule. Let $G'$ and $-H^{L'}$ be followers of $G$ and $-H^L$, respectively. Whenever, in a position $G'\uplus(-H^{L'})$, on her turn, Left has the opportunity to initiate sequences of consecutive checks in either of the two components, \emph{she can inquire about Right's responses}. More precisely, Right must commit to a strategy of responses $(G',\ldots,G'^{\overrightarrow{L}R\ldots\overrightarrow{L}R},\ldots)$ to all forcing sequences from $G'$ and to a strategy of responses $(-H^{L'},\ldots, -H^{L'\overrightarrow{R}L\ldots\overrightarrow{R}L},\ldots)$ to all forcing sequences from $-H^{L'}$. This is done in advance, and, after that, Left may jump directly from $G' \uplus (-H^{L'})$ to some $G'^{\overrightarrow{L}R\ldots\overrightarrow{L}R} \uplus (-H^{L'\overrightarrow{R}L\ldots\overrightarrow{R}L})$ continuing the play in that position with a silent move, which can be $G'^{\overrightarrow{L}R\ldots\overrightarrow{L}RL} \uplus (-H^{L'\overrightarrow{R}L\ldots\overrightarrow{R}L})$ or $G'^{\overrightarrow{L}R\ldots\overrightarrow{L}R} \uplus (-H^{L'\overrightarrow{R}L\ldots\overrightarrow{R}LR})$. Hence, the {\sc inquisitor game} is a compound game $G_1 \uplus G_2$ where

\begin{enumerate}
\item[] (Disjunctive sum rule) Players make their moves as if it were a disjunctive sum;
\item[] (Last opportunity rule) After a move by Right to $\overline{\infty}$ in one of the components, Left has one last opportunity to win by playing to $\infty$ in the other component;
\item[] (Inquiry rule) On her turn, and being able to initiate forcing sequences \emph{in both components}, Left can inquire in advance about Right's behaviour in each of them. After that, Left can jump directly to a pair of responses, continuing the game from there.
\end{enumerate}

Let $G,H\in\Np\setminus\{\infty,\overline{\infty}\}$. What we will prove is that $G \geqslant H$ if and only if Left, playing second, wins the {\sc inquisitor game} $G \uplus (-H)$. Of course, there is a dual formalization $G_1\,$\rotatebox[origin=c]{180}{$\uplus$}$\,G_2$ from Right's perspective.\\

As an example, let us revisit the games $G$ and $H$ from the previous subsection. To check if $G \geqslant H$, we have to play $G \uplus (-H)$. Consider the following Right's move:
\begin{center}
\begin{tikzpicture}
\clip(10.,8.3) rectangle (31.,12.8-2);
\draw [line width=1.pt] (12.,12.-2) -- (15.,9.);
\draw (10.36,12.7-2) node[anchor=north west] {$G \uplus (-H)$};
\draw (12,8.86) node[anchor=north west] {$\left\{*3,\{\infty\,|\,\{0\,|\,\overline{\infty}\},\{*2\,|\,\overline{\infty}\}\}\,\|\,0\right\} \uplus \{\{\infty\,|\,0\},\{\infty\,|\,*2,*3\}\,|\,-1\}$};
\end{tikzpicture}
\end{center}

\vspace{0.5cm}
Now it is Left's turn, and there are Left-checks in both components. Regarding the first one, suppose that, after inquiring Right, she finds out that the response to the check is to $\{0\,|\,\overline{\infty}\}$. Regarding the second component, against the first check, Right commits to the answer $0$, and against the second check, Right commits to the answer $*2$. We have the following situation:
\begin{center}
\begin{tikzpicture}
\clip(10.,7.5-0.5) rectangle (31.,12.8-2);
\draw [line width=1.pt] (12.,12.-2) -- (15.,9.);
\draw (10.36,12.7-2) node[anchor=north west] {$G \uplus (-H)$};
\draw (12,8.86) node[anchor=north west] {$\left\{*3,\{\infty\,|\,\{0\,|\,\overline{\infty}\},\{*2\,|\,\overline{\infty}\}\}\,\|\,0\right\} \uplus \{\{\infty\,|\,0\},\{\infty\,|\,*2,*3\}\,|\,-1\}$};
\draw (13.8,8.1) node[anchor=north west] {$\left(G,\{0\,|\,\overline{\infty}\}\right)$};
\draw (19,8.1) node[anchor=north west] {$\left(-H^L,0,*2\right)$};
\end{tikzpicture}
\end{center}

Given these strategies of responses, Left chooses $\{0\,|\,\overline{\infty}\}$ and $0$, securing the victory by playing from $\{0\,|\,\overline{\infty}\} \uplus 0$ to $0 \uplus 0$.
\begin{center}
\begin{tikzpicture}
\clip(10.,1+3) rectangle (31.,12.8-2);
\draw [line width=1.pt] (12.,12.-2) -- (15.,9.);
\draw (10.36,12.7-2) node[anchor=north west] {$G \uplus (-H)$};
\draw (12,8.86) node[anchor=north west] {$\left\{*3,\{\infty\,|\,\{0\,|\,\overline{\infty}\},\{*2\,|\,\overline{\infty}\}\}\,\|\,0\right\} \uplus \{\{\infty\,|\,0\},\{\infty\,|\,*2,*3\}\,|\,-1\}$};
\draw (13.8,8.1) node[anchor=north west] {$\left(G,\bm{\mathsf{\{0\,|\,\overline{\infty}\}}}\right)$};
\draw (19,8.1) node[anchor=north west] {$\left(-H^L,\bm{\mathsf{0}},*2\right)$};
\draw (15,6+1) node[anchor=north west] {$\{0\,|\,\overline{\infty}\}\uplus 0$};
\draw [line width=1.pt] (15.,5.+1) -- (12.,2.+3);
\draw (10.9,1.85+3) node[anchor=north west] {$0 \uplus 0$};
\end{tikzpicture}
\end{center}

Regardless of other strategies of responses, Left still wins. Suppose that, regarding the check in the first component, Right's commitment is the response to $\{*2\,|\,\overline{\infty}\}$. Regarding the second component, against the first check, Right commits to the answer $0$, and against the second check, Right commits to the answer $*3$. Given this strategy of responses, Left chooses $G$ and $*3$, securing victory by playing from $G \uplus *3$ to $*3 \uplus *3$.
\begin{center}
\begin{tikzpicture}
\clip(10.,1+3) rectangle (31.,12.8-2);
\draw [line width=1.pt] (12.,12.-2) -- (15.,9.);
\draw (10.36,12.7-2) node[anchor=north west] {$G \uplus (-H)$};
\draw (12,8.86) node[anchor=north west] {$\left\{*3,\{\infty\,|\,\{0\,|\,\overline{\infty}\},\{*2\,|\,\overline{\infty}\}\}\,\|\,0\right\} \uplus \{\{\infty\,|\,0\},\{\infty\,|\,*2,*3\}\,|\,-1\}$};
\draw (13.8,8.1) node[anchor=north west] {$\left(\bm{G},\{*2\,|\,\overline{\infty}\}\right)$};
\draw (19,8.1) node[anchor=north west] {$\left(-H^L,0,\bm{\mathsf{*3}}\right)$};
\draw (15,6+1) node[anchor=north west] {$G\uplus *3$};
\draw [line width=1.pt] (15.,5.+1) -- (12.,2.+3);
\draw (10.9,1.85+3) node[anchor=north west] {$*3 \uplus *3$};
\end{tikzpicture}
\end{center}

Once it is possible to verify that in all cases, playing second, Left wins $G \uplus (-H)$, it follows that $G \geqslant H$.

\subsection{Strategies of responses and game trees}
\label{ssec:strategies}

\vspace{0.3cm}
In the same way that in classic normal play, the statement ``$G\geqslant H$ if and only if $o^R(G+(-H))=\mathbf{L}$'' can be transformed into ``$G\geqslant H$ if and only if $\mathbf{M}(G,H)$'', in affine normal play, ``$G\geqslant H$ if and only if $o^R(G\uplus(-H))=\mathbf{L}$'' can be transformed into ``$G\geqslant H$ if and only if $\mathbf{M}^\infty(G,H)$''. In this last sentence, $\mathbf{M}^\infty(G,H)$ is a weaker form of maintenance property that will be formalized next. For now, so that the proofs can be done, it is necessary to define the meaning of ``strategy of responses'' to forcing sequences.

\begin{definition}[Strategy of Responses]\label{def:strategy}
Let $G\in \Np$. A \emph{strategy of Right responses} in $G$ is a network of decisions in response to forcing sequences carried out by Left, constructed as follows:
\begin{itemize}
  \item For $i=0$, $\mathbb{R}^0(G)=\{G\}$ is the singular set that contains $G$.
  \item For $i>0$, $\mathbb{R}^i(G)$ contains a choice of exactly one $K^{\overrightarrow{L}R}\neq\infty$ for each Left-check $K^{\overrightarrow{L}}$, where $K\in \mathbb{R}^{i-1}(G)$ (whenever such a choice can be made). $\mathbb{R}^i(G)$ may be empty; when this happens, for every $k>i$, $\mathbb{R}^k(G)$ is also empty, indicating the end of the decision-making process.
\end{itemize}

The set $\mathbb{R}(G)=\cup_{i \geqslant 0} \mathbb{R}^i(G)$ is the strategy of Right responses. A \emph{strategy of Left responses} $\mathbb{L}(G)$ is defined analogously.\\

The set of all strategies of Right responses in $G$ is denoted by $\mathfrak{R}(G)$. Analogously, the set of all strategies of Left responses in $G$ is denoted by $\mathfrak{L}(G)$.
\end{definition}

Given a game $G$, a strategy of Right responses is a set of Right choices in response to all forcing sequences by Left (a strategy of Left responses is the symmetrical concept). When comparing games, there is the possibility for Left to play in $G$, so $G$ must belong to the strategy of Right responses. That explains the base case $\mathbb{R}^0(G)=\{G\}$. While Left is forcing with checks, Right must answer, and the answers must belong to the strategy. That explains the recursion. What is new in $\Np$ is that the Left-silent options in $G$ are not the only Left-silent options that matter. We can have relevant Left-silent options after a sequence of checks, and these options depend on the strategy chosen by Right.

\begin{example}
Consider once again $G$ and $H^L$ as given in Subsection \ref{ssec:challenging}. Starting with $G=\left\{*3,\{\infty\,|\,\{0\,|\,\overline{\infty}\},\{*2\,|\,\overline{\infty}\}\}\,\|\,0\right\}$, there are two strategies of Right responses:
\begin{itemize}
  \item $\mathbb{R}_1(G)=\{G,\{0\,|\,\overline{\infty}\}\}$
  \item $\mathbb{R}_2(G)=\{G,\{*2\,|\,\overline{\infty}\}\}$
\end{itemize}

In $H^L=\{1\,|\,\{0\,|\,\overline{\infty}\},\{*2,*3\,|\,\overline{\infty}\}\}$, there are two strategies of Left responses:
\begin{itemize}
  \item $\mathbb{L}_1(H^L)=\{H^L,0,*2\}$
  \item $\mathbb{L}_2(H^L)=\{H^L,0,*3\}$
\end{itemize}
\end{example}

\noindent
\textbf{Notation:} $\overrightarrow{G}$ denotes a typical element of $\mathbb{R}(G)$ and $\overleftarrow{G}$ denotes a typical element of $\mathbb{L}(G)$. Hence, in $\mathbb{R}_1(G)$ in the previous example, $\overrightarrow{G}_1=G$ and $\overrightarrow{G}_2=\{0\,|\,\overline{\infty}\}$.\\

The choice of the term ``strategy'' was not accidental, as this same term is used in Economic Game Theory \citep{Ste021}. In fact, the concept we just described is an interesting case in which research in Combinatorial Game Theory borrows a procedure used in Economic Game Theory.

For a better understanding of this observation, we need to resort to game trees. A game tree of an affine game is drawn exactly the same way as in classical theory, with two nuances. Whenever a player's only option is a suicidal move, no edge is placed in the tree (as having or not having that move is equivalent); this is the first nuance. Whenever a player has at their disposal a winning terminating move, a small arrow is placed in the game tree; this is the second nuance. That said, as in classical theory, the followers of a game are identified with the nodes of its game tree, as these nodes are the roots of the subtrees that represent them. The games $\infty$ and $\overline{\infty}$ do not have a game tree, as in these cases, the play has already ended. Figure \ref{fig10} shows some examples.\\

\begin{figure}[!htb]
\begin{center}
\includegraphics[scale=1]{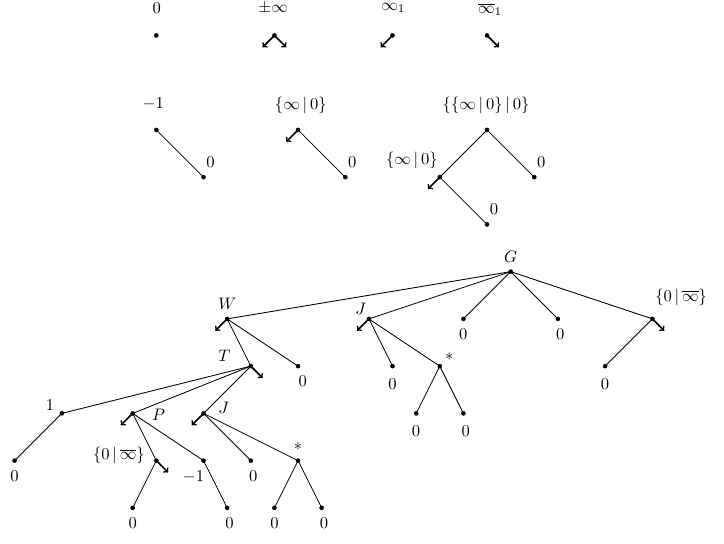}
\end{center}

\vspace{-0.4cm}
\caption{On top, the game trees of the four games $0=\{\overline{\infty}\,|\,\infty\}$, $\pm \infty=\{\infty\,|\,\overline{\infty}\}$, $\infty_1=\{\infty\,|\,\infty\}$, and $\overline{\infty}_1=\{\overline{\infty}\,|\,\overline{\infty}\}$ are displayed. In the center, the game trees of the games $-1$, $\{\infty\,|\,0\}$, and $\{\{\infty\,|\,0\}\,|\,0\}$ are easy to understand. Regarding the game tree below, we have $J=\{\infty\,|\,0,*\}$, $P=\{\infty\,|\,\{0\,|\,\overline{\infty}\},-1\}$, $T=\{1,P,J\,|\,\overline{\infty}\}$, and $W=\{\infty\,|\,T,0\}$. Since $G=\{W,J,0\,|\,0,\{0\,|\,\overline{\infty}\}\}$, the game tree of $G$ pertains to the game form $\{\{\infty\,|||\,\{1,\{\infty\,|\,\{0\,|\,\overline{\infty}\},-1\},\{\infty\,|\,0,*\}\,||\,\overline{\infty}\},0\},\{\infty\,|\,0,*\},0\,||||\,0,\{0\,|\,\overline{\infty}\}\}$.}
\label{fig10}
\end{figure}

Returning to the procedure used in Economic Game Theory, given an affine game $K$, it is possible to consider the \emph{Left-forcing-tree} of $K$ with all the information about the forcing sequences that can be carried out by Left. To do this, the following four steps are taken: (1) the root, identified with the game $K$ itself, is not removed from the game tree; (2) all nodes related to the elements of $K^{\mathcal{R}}$ and their followers are removed from the game tree; (3) regarding the Left moves, only the nodes identified with followers of the type $K^{\overrightarrow{L}R\overrightarrow{L}R\ldots\overrightarrow{L}}$ are retained; (4) regarding the Right moves, only the nodes identified with the responses to the Left-checks from the previous item are retained.

We observe that a Left-forcing-tree \emph{is not} a game tree. It is an object constructed from a game tree, keeping games as labels for some chosen nodes. Figure \ref{fig11} illustrates the Left-forcing-tree of the game $G$ shown in Figure \ref{fig10}.\\

\begin{figure}[!htb]
\begin{center}
\scalebox{0.5}{
\includegraphics[scale=1]{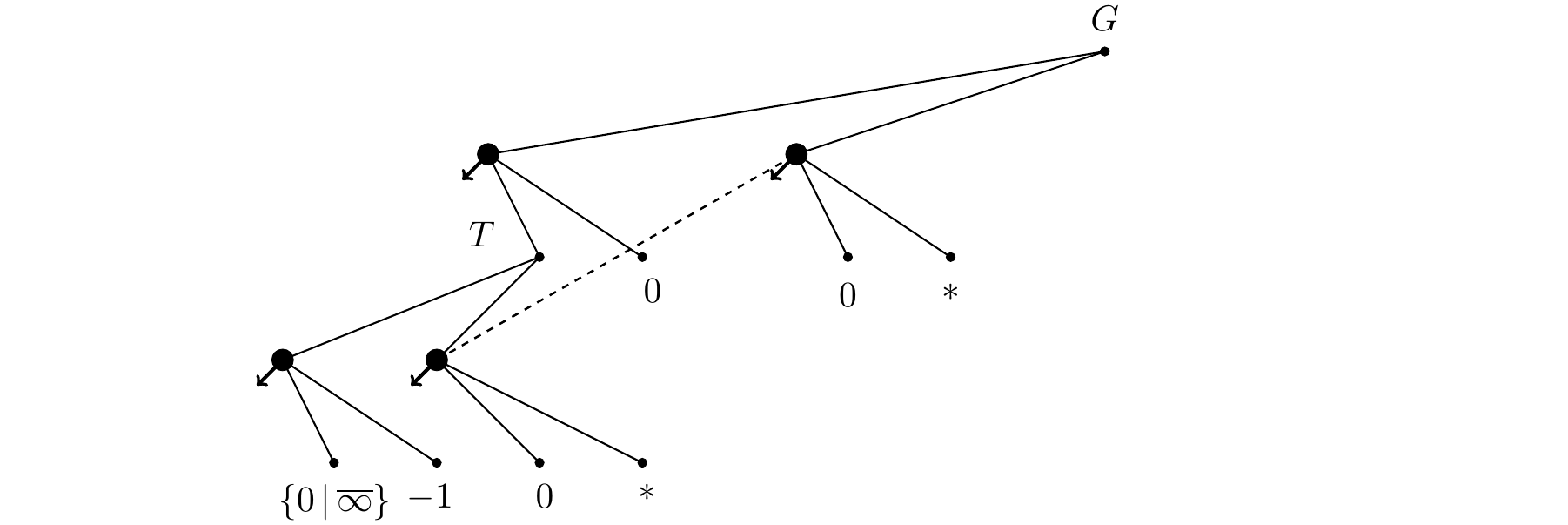}}
\end{center}

\vspace{-0.4cm}
\caption{Left-forcing-tree of the game $G$ shown in Figure 10.}
\label{fig11}
\end{figure}

\vspace{0.8cm}
It is possible to think of $\mathfrak{R}(G)$ by considering the Left-forcing-tree of $G$ as if it were the extensive form of an economic game without payoffs. Note that in Economic Game Theory, node labels are typically payoffs, and edge labels help identifying strategies. Since, in our case, there are no payoffs, the node labels remain the games identified with those nodes. The fundamental idea is to not use labels on the edges, identifying the strategies directly with the games (node labels). The larger ones (in black) are decision nodes, nodes where Right can be faced with a decision. The smaller nodes are possible choices and a strategy of Right responses is a set of choices.

The dashed line, also used in Economic Game Theory, indicates that Right's decision can be considered the same in both nodes (the check is identical). Making different decisions against the same check, only because it occurs at different places of the game tree, merely increases the number of elements in the strategy, benefiting Left. Therefore, assuming that Right always responds in the same way against the same check reduces the work in determining strategies. Having that in mind, there are $6$ strategies of Right responses in $G$:

\begin{itemize}
  \item $\mathbb{R}_1(G)=\{G,0\}$ -- Right moves to $0$ against both checks.
  \item $\mathbb{R}_2(G)=\{G,0,*\}$ -- Right moves to $0$ against the first check and to $*$ against the second.
  \item $\mathbb{R}_3(G)=\{G,T,\{0\,|\,\overline{\infty}\},0\}$ -- Against the first check, Right moves to $T$ with the intention of continuing with $\{0\,|\,\overline{\infty}\}$ or with $0$. The response to the second check does not need to be specified since that check has already been analyzed.
  \item $\mathbb{R}_4(G)=\{G,T,\{0\,|\,\overline{\infty}\},*\}$ -- Against the first check, Right moves to $T$ with the intention of continuing with $\{0\,|\,\overline{\infty}\}$ or with $*$. The response to the second check does not need to be specified.
  \item $\mathbb{R}_5(G)=\{G,T,-1,0\}$ -- Against the first check, Right moves to $T$ with the intention of continuing with $-1$ or with $0$. The response to the second check does not need to be specified.
  \item $\mathbb{R}_6(G)=\{G,T,-1,*\}$ -- Against the first check, Right moves to $T$ with the intention of continuing with $-1$ or with $*$. The response to the second check does not need to be specified.\\
\end{itemize}

The set of strategies of Right responses in $G$ is $\mathfrak{R}(G)=\{\mathbb{R}_1,\mathbb{R}_2,\mathbb{R}_3,\mathbb{R}_4,\mathbb{R}_5,\mathbb{R}_6\}$.\\

\subsection{Options-only comparison in $\Np$}
\label{ssec:comparison}

\vspace{0.3cm}
To formalize the maintenance property appropriate to affine normal play, it is necessary to combine two strategies of responses, one for each player.

\begin{definition}[Locked Pairs]\label{def:locked} Let $G,H\in \Np$. The pair of games $(G,H)$ is \emph{locked} if for all $(\mathbb{R},\mathbb{L})\in\mathfrak{R}(G)\times\mathfrak{L}(H)$ there are two games $G'\in \mathbb{R}$ and $H'\in \mathbb{L}$ for which
\begin{itemize}
  \item there exists a Right-silent game $H'^{R}$ such that $G'\geqslant H'^{R}$ or
  \item there exists a Left-silent $G'^{L}$ such that $G'^{L}\geqslant H'$.
\end{itemize}
\end{definition}

\begin{definition}[$\infty$-Maintenance Property]\label{def:maintenanceinf}
Let $G,H\in\Np$ be two affine games. We say that $(G,H)$ \emph{satisfies the $\infty$-maintenance property}, denoted by $\mathbf{M}^\infty(G,H)$, if $(G^R,H)$ is locked for all $G^R\in \GR$, and $(G,H^L)$ is locked for all $H^L\in \HL$.\\
\end{definition}

\begin{observation}
If $(G,H)$ satisfies the $\infty$-maintenance property, then, for all $G^R$ and for each pair $(\mathbb{R}(G^R),\mathbb{L}(H))$, there are two games $\overrightarrow{G^R}_a\in\mathbb{R}(G^R)$ and $\overleftarrow{H}_b\in \mathbb{L}(H)$ ``compatible with the traditional maintenance property''. That means that there exists a Right-silent game $(\overleftarrow{H}_b)^R$ for which $\overrightarrow{G^R}_a\geqslant (\overleftarrow{H}_b)^R$ or  there exists a Left-silent game $(\overrightarrow{G^R}_a)^L$ for which $(\overrightarrow{G^R}_a)^L\geqslant \overleftarrow{H}_b$. The same happens with any $H^L$. In that sense, we have a weaker version of the traditional maintenance property, as will be stated in Theorem \ref{thm:traditional}.\\
\end{observation}

\begin{example}
Consider again $G$ and $H$ from Subsection \ref{ssec:challenging}, to check that $H^L$ poses no issues regarding the $\infty$-maintenance property. We have 2 strategies of Right responses, $\mathbb{R}_1(G)=\{G,\{0\,|\,\overline{\infty}\}\}$ and $\mathbb{R}_2(G)=\{G,\{*2\,|\,\overline{\infty}\}\}$, and 2 strategies of Left responses, $\mathbb{L}_1(H^L)=\{H^L,0,*2\}$ and $\mathbb{L}_2(H^L)=\{H^L,0,*3\}$. There are $4$ pairs to consider.
\begin{itemize}
  \item For $\left(\mathbb{R}_1(G),\mathbb{L}_1(H^L)\right)$, we have the pair $(\{0\,|\,\overline{\infty}\},0)$ and the Left-silent game \linebreak$0\in \{0\,|\,\overline{\infty}\}^\mathcal{L}$ such that $0\geqslant 0$.
  \item For $\left(\mathbb{R}_1(G),\mathbb{L}_2(H^L)\right)$, we have the pair $(\{0\,|\,\overline{\infty}\},0)$ and the Left-silent game \linebreak$0\in \{0\,|\,\overline{\infty}\}^\mathcal{L}$ such that $0\geqslant 0$.
  \item For $\left(\mathbb{R}_2(G),\mathbb{L}_1(H^L)\right)$, we have the pair $(\{*2\,|\,\overline{\infty}\},*2)$ and the Left-silent game \linebreak$*2\in \{*2\,|\,\overline{\infty}\}^\mathcal{L}$ such that $*2\geqslant *2$.
  \item For $\left(\mathbb{R}_2(G),\mathbb{L}_2(H^L)\right)$, we have the pair $(G,*3)$ and the Left-silent game $*3\in G^\mathcal{L}$ such that $*3\geqslant *3$.\\
\end{itemize}
\end{example}

It is time to prove the two implications that underpin the main result of this section.\\

\begin{theorem} \label{thm:firstimpli}Let $G,H \in\Np\setminus\{\infty,\overline{\infty}\}$. If $G\geqslant H$ then $\mathbf{M}^\infty(G,H)$.
\end{theorem}

For the proof of Theorem \ref{thm:firstimpli}, a preliminary lemma is required.\\

\begin{lemma}\label{lem:downlink} Let $G,H\in \Np$.
\begin{itemize}
  \item If $G\not\geqslant H$ and $G$ is Left-silent, then there is some $X\in \Np\setminus\{\infty,\overline{\infty}\}$ such that $o^L(G+X)=\mathbf{R}<o^L(H+X)=\mathbf{L}$.
  \item  If $G\not\geqslant H$ and $H$ is Right-silent, then there is some $X\in \Np\setminus\{\infty,\overline{\infty}\}$ such that $o^R(G+X)=\mathbf{R}<o^R(H+X)=\mathbf{L}$.
\end{itemize}
\end{lemma}

\begin{proof}
It is enough to prove the first item, as the proof for the second is entirely analogous. By Definition \ref{def:orderequivalence}, $G\not\geqslant H$ implies that
\begin{itemize}
  \item $\exists X\in \Np\setminus\{\infty,\overline{\infty}\}$ such that $o^L(G+X)=\mathbf{R}<o^L(H+X)=\mathbf{L}$ or
  \item $\exists Y\in \Np\setminus\{\infty,\overline{\infty}\}$ such that $o^R(G+Y)=\mathbf{R}<o^R(H+Y)=\mathbf{L}$.
\end{itemize}

If the first case is satisfied, then the proof is complete. If the second case is satisfied, then let $X$ be the game $\{Y\,|\,\overline{\infty}\}$. Cases where $G=\overline{\infty}$ or $H=\infty$ are very simple to handle. Otherwise, we have $o^L(G+X)=\mathbf{R}<o^L(H+X)=\mathbf{L}$. These individualized outcomes arise from the fact that the moves to $G+Y$ and $H+Y$ are mandatory, and $o^R(G+Y)=\mathbf{R}<o^R(H+Y)=\mathbf{L}$.\\
\end{proof}

\begin{proof} (Theorem \ref{thm:firstimpli})
For the sake of contradiction, suppose that we have $G\geqslant H$ and that $(G,H)$ does not satisfy the $\infty$-maintenance property as stated in Definition \ref{def:maintenanceinf}.\\

Without loss of generality, suppose that the $\infty$-maintenance property fails due to the existence of some $H^L \in H^\mathcal{L}$ opposing it. If the failure were due to the existence of some $G^R \in G^\mathcal{R}$ opposing it, the proof would be similar.

Saying that is the same as saying that there is a particular \emph{pernicious pair} $(\mathbb{R}(G),\mathbb{L}(H^L))=(\{\overrightarrow{G}_1,\ldots,\overrightarrow{G}_a,\ldots\},\{\overleftarrow{H^L}_1,\ldots,\overleftarrow{H^L}_b,\ldots\})$ where the following holds for each pair $(\overrightarrow{G}_a,\overleftarrow{H^L}_b)$:\\

\noindent
there are no Left-silent games $(\overrightarrow{G}_a)^{L_i}$ or, for all $i$, $\underbrace{(\overrightarrow{G}_a)^{L_i}}_{Left-silent}\not\geqslant \overleftarrow{H^L}_b$
\begin{center}and\end{center}
there are no Right-silent games $(\overleftarrow{H^L}_b)^{R_j}$ or, for all $j$, $\overrightarrow{G}_a\not\geqslant \underbrace{(\overleftarrow{H^L}_b)^{R_j}}_{Right-silent}$.\\\\

Note that there is a slight abuse of language in the previous expressions. The games $(\overrightarrow{G}_a)^{L_i}$ should ideally be expressed as $(\overrightarrow{G}_a)^{L_i}$, where $i\in\{1,\ldots,k_a\}$, given that we are enumerating the Left-silent options of $\overrightarrow{G}_a$. However, to prevent an overwhelming use of subscripts and superscripts, we have chosen a simplified writing style.\\

\noindent
Fact 1: Regarding the Left-silent games, Lemma \ref{lem:downlink} ensures that, for all $a, b, i$, there exist distinguishing games $X_{a, b}^i$ such that

$$o^L(\underbrace{(\overrightarrow{G}_a)^{Li}}_{Left-silent}+X_{a,b}^i)=\mathbf{R}<o^L(\overleftarrow{H^L}_b+X_{a,b}^i)=\mathbf{L}.$$

\vspace{0.3cm}
\noindent
Fact 2: Regarding the Right-silent games, Lemma \ref{lem:downlink} ensures that, for all $a, b, j$, there exist distinguishing games $Y_{a, b}^j$ such that

$$o^R(\overrightarrow{G}_a+Y_{a,b}^j)=\mathbf{R}<o^R(\underbrace{(\overleftarrow{H^L}_b)^{R_j}}_{Right-silent}+Y_{a,b}^j)=\mathbf{L}.$$

With the aim of constructing a distinguishing game leading to the desired contradiction, for each $b$, let $\overline{X}_b=\bigcup_{a,i} \{X_{a,b}^i\}$ , and for each $a$, let $\overline{Y}_a=\bigcup_{b,j} \{Y_{a,b}^j\}$. The slight abuse of language is repeated. The set $\bigcup_{a,i} \{X_{a,b}^i\}$  is $\bigcup_{a}\bigcup_{i\in\{1,\ldots,k_a\}}\{X_{a,b}^i\}$, and the same for $\bigcup_{b,j} \{Y_{a,b}^j\}$. Supported by these sets, let\\

$$X^L=\left\{\begin{array}{ll}
               \overline{\infty} & \text{if there are no Right-silent }(\overleftarrow{H^L}_b)^{R_j} \\
               \{\infty\,|\,\{\overline{Y}_1\,|\,\overline{\infty}\},\{\overline{Y}_2\,|\,\overline{\infty}\},\ldots\} & \text{if there are Right-silent }(\overleftarrow{H^L}_b)^{R_j}
             \end{array}
\right.$$

\vspace{-0.1cm}
and

\vspace{-0.3cm}
$$X^R=\left\{\begin{array}{ll}
               \infty & \text{if there are no Left-silent }(\overrightarrow{G}_a)^{Li} \\
               \{\{\infty\,|\,\overline{X}_1\},\{\infty\,|\,\overline{X}_2\},\ldots\,|\,\overline{\infty}\} & \text{if there are Left-silent }(\overrightarrow{G}_a)^{Li}
             \end{array}
\right.$$

and consider the distinguishing game $X=\{X^L\,|\,X^R\}$.\\\\\\

\noindent
Claim 1: Left, playing first, loses $G+X$. Right follows his strategy of the pernicious pair, described in the following picture. The dashed line means ``forcing sequence''. After that, Left either chooses a Left-silent game on the first component or makes a move on the second component. The ``Right wins'' are supported by Facts 1~and~2. If there are no Left-silent games $(\overrightarrow{G}_a)^{Li}$, then only the branch on the right matters. According to Definition \ref{def:strategy}, Right never moves to $\infty$, and, if Right moves to $\overline{\infty}$, then Left loses anyway. If $X^L=\overline{\infty}$, then Left loses when playing to $(\overrightarrow{G}_a)+X^L$.\\

\vspace{1cm}
\begin{center}
\scalebox{0.75}{
\includegraphics[scale=1]{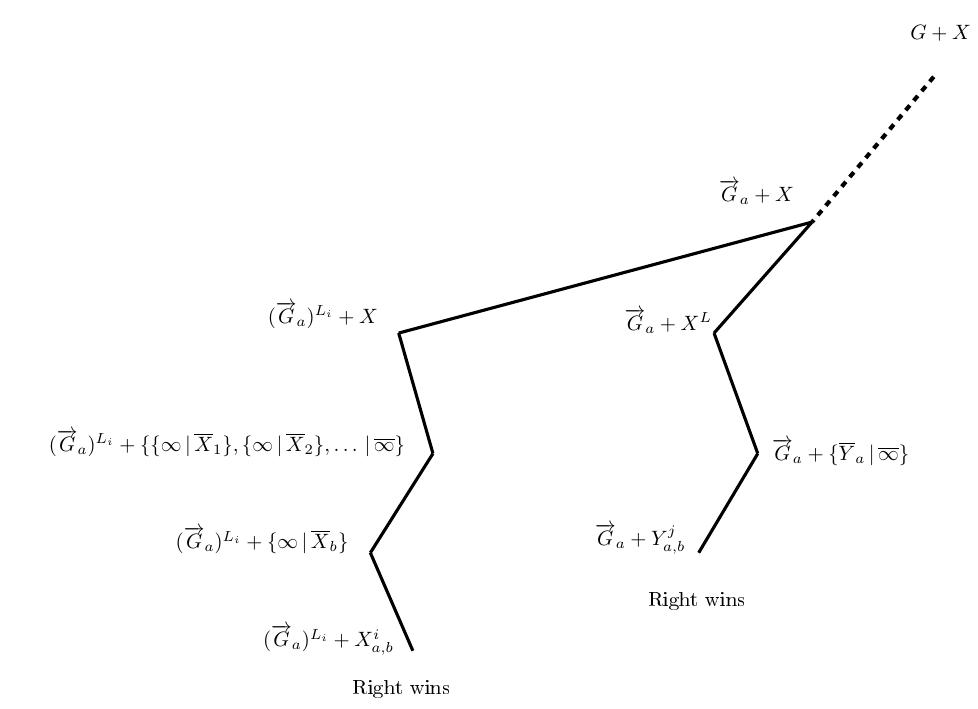}}
\end{center}

\vspace{1cm}
Claim 2: Left, playing first, wins $H+X$ by moving to $H^L+X$. Left follows her strategy of the pernicious pair, described in the following picture. The dashed line means ``forcing sequence''. After that, Right either makes a move on the second component or chooses a Right-silent game on the first component. The ``Left wins'' are supported by Facts 1~and~2. If there are no Right-silent games $(\overleftarrow{H^L}_b)^{R_j}$ then only the branch on the left matters. According to Definition \ref{def:strategy}, Left never moves to $\overline{\infty}$, and, if Left moves to $\infty$, then Left wins anyway. If $X^R=\infty$ then Right loses when playing to $(\overleftarrow{H^L}_b)+X^R$.

\begin{center}
\scalebox{0.75}{
\includegraphics[scale=1]{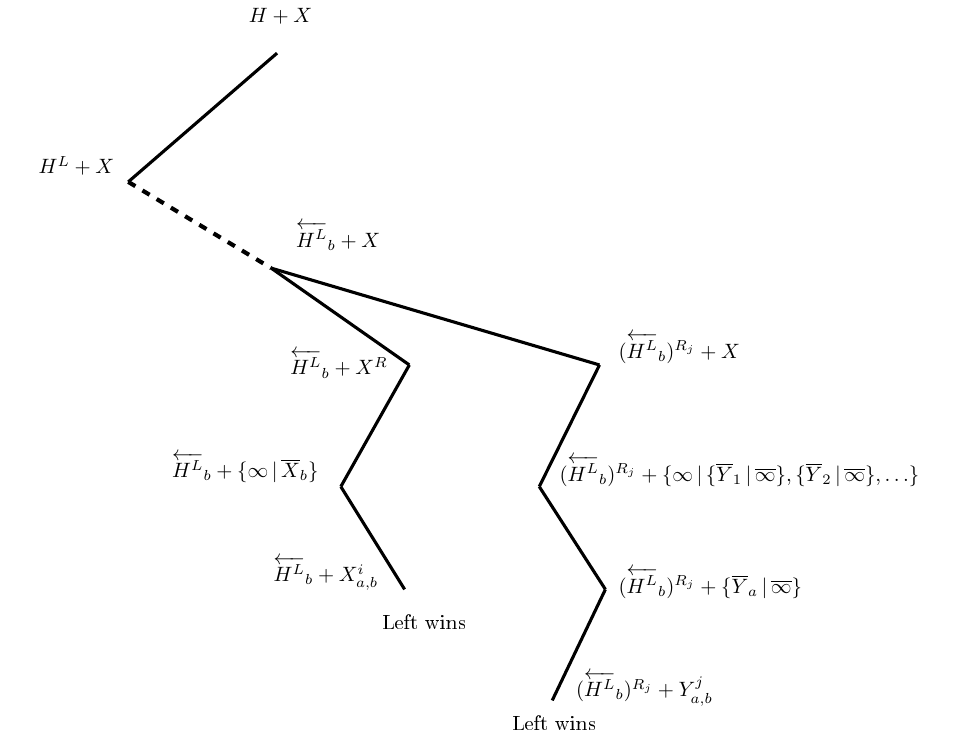}}
\end{center}

It follows from Claims 1 and 2 that $G\not\geqslant H$, resulting in a contradiction.
\end{proof}

The opposite direction is also true, as the following Theorem shows.

\begin{theorem} \label{thm:secondimpli}Let $G,H \in\Np\setminus\{\infty,\overline{\infty}\}$. If $\mathbf{M}^\infty(G,H)$ then $G\geqslant H$.
\end{theorem}

\begin{proof}
For the sake of contradiction, suppose that $\mathbf{M}^\infty(G,H)$ and $G \not\geqslant H$. According to Definition \ref{def:orderequivalence}, there exists a distinguishing game $X\in\Np\setminus\{\infty,\overline{\infty}\}$ such that either $o^L(G+X)=\mathbf{R}<o^L(H+X)=\mathbf{L}$ or $o^R(G+X)=\mathbf{R}<o^R(H+X)=\mathbf{L}$.

Without loss of generality, assume that $X$ has the smallest formal birthday, and such that $o^L(G+X)=\mathbf{R}<o^L(H+X)=\mathbf{L}$. We can establish the following three facts:

\begin{enumerate}[(a)]

 \item $X$ is Left-silent since $o^L(G+X)=\mathbf{R}$.

  \item If $o^L(H+X)=o^R(H^L+X)=\mathbf{L}$ for some $H^L\in H^{\mathcal{L}}$, then there exists a strategy $\mathbb{L}(H^L)$ such that $o^R(\overleftarrow{H^L}+X)=\mathbf{L}$ for all $\overleftarrow{H^L}\in \mathbb{L}(H^L)$. If this were not the case, Right could maneuver with checks, and we would not have $o^R(H^L+X)=\mathbf{L}$. We can also conclude that, for all $\overleftarrow{H^L}\in \mathbb{L}(H^L)$, we have $o^L((\overleftarrow{H^L})^R+X)=\mathbf{L}$, regardless of the Right option $(\overleftarrow{H^L})^R$. On the other hand, there exists a strategy $\mathbb{R}(G)$ such that $o^L(\overrightarrow{G}+X)=\mathbf{R}$ for all $\overrightarrow{G}\in\mathbb{R}(G)$. If this were not the case, Left could manoeuver with checks and we would not have $o^L(G+X)=\mathbf{R}$.  We can also conclude that, for all $\overrightarrow{G}\in\mathbb{R}(G)$, we have $o^R(\overrightarrow{G}^L+X)=\mathbf{R}$, regardless of the Left option $\overrightarrow{G}^L$. Thus, if $\overrightarrow{G}\in\mathbb{R}(G)$ and $\overleftarrow{H^L}\in\mathbb{L}(H^L)$, there are neither  $\overrightarrow{G}^L\geqslant \overleftarrow{H^L}$ nor $\overrightarrow{G}\geqslant (\overleftarrow{H^L})^R$. This is equivalent to saying that the pair $(\mathbb{R}(G),\mathbb{L}(H^L))$ breaks the $\infty$-maintenance property, which constitutes a contradiction.

  \item According to the initial smallest birthday assumption, $o^R(G+X^L)\geqslant o^R(H+X^L)$.\linebreak If $o^L(H+X)=o^R(H+X^L)=\mathbf{L}$ for some game $X^L\in X^{\mathcal{L}}$, then we conclude that $o^L(G+X)\geqslant o^R(G+X^L)\geqslant o^R(H+X^L)=o^L(H+X)$. Consequently, in that case, $o^L(G+X)=\mathbf{L}$, which is once again a contradiction.
\end{enumerate}

A contradiction is unavoidable given the assumptions. Hence, we have shown that $G\geqslant H$.
\end{proof}

Finally, the main result of this section follows.

\begin{theorem} \label{thm:constructive}Let $G,H \in\Np\setminus\{\infty,\overline{\infty}\}$. We have that $G\geqslant H$ if and only if $\mathbf{M}^\infty(G,H)$.
\end{theorem}

\begin{proof}
Combine Theorem \ref{thm:firstimpli} and Theorem \ref{thm:secondimpli}.
\end{proof}

Situations with many checks are relatively rare in game practice. In most cases, the traditional maintenance property, as stated in Definition \ref{def:maintenance}, is enough to verify if $G\geqslant H$.

\begin{theorem} \label{thm:traditional}Let $G,H \in\Np\setminus\{\infty,\overline{\infty}\}$. If $\mathbf{M}(G,H)$ then $G\geqslant H$.
\end{theorem}

\begin{proof}
If $\mathbf{M}(G,H)$, then $\mathbf{M}^\infty(G,H)$, since $G$ belongs to every strategy $\mathbb{R}(G)$, $H^L$ belongs to every strategy $\mathbb{L}(H^L)$, $G^R$ belongs to every strategy $\mathbb{R}(G^R)$, and $H$ belongs to every strategy $\mathbb{L}(H)$.

As $\mathbf{M}^\infty(G,H)$, Theorem \ref{thm:constructive} ensures that $G\geqslant H$.
\end{proof}

\section{Reductions and reduced forms}
\label{sec:reductions}

\vspace{0.3cm}
In $\Np$, the game reductions are essentially the same as in $\Npc$. The following three theorems are stated from Left's perspective, but, of course, there are dual statements from Right's perspective.

\emph{Domination} captures the idea that if a player has an option at least as good as another, that other option becomes superfluous.

\begin{theorem}[Domination]\label{thm:domination}
If $G\in \Np$  is a game such that $G^{L_1}, G^{L_2}\in \GL$ and\linebreak $G^{L_2}\geqslant G^{L_1}$, then $G=\{\GL\setminus\{G^{L_1}\}\mid \GR\}$.
\end{theorem}

\begin{proof}
Let $G'=\{\GL\setminus\{G^{L_1}\}\mid \GR\}$. We want to verify that $o^L(G + X)= o^L(G' + X)$ and $o^R(G + X)= o^R(G' + X)$ for every game $X$ different from $\infty$ and $\overline{\infty}$.

If Left wins $G'+X$ with a move $G'+X^L$, then, by induction, she also wins $G+X$ with the move $G+X^L$. Analogously, if Left wins $G+X$ with a move $G+X^L$, then, by induction, she also wins $G'+X$ with the move $G'+X^L$. The argument for the case where Right plays first in the component $X$ is similar.

Regarding moves other than in the component $X$, the only situation that escapes direct mimicry pertains to a possible winning move by Left in $G+X$ to $G^{L_1}+X$. In that case, playing first, she also wins $G'+X$ by moving to $G^{L_2}+X$, since $G^{L_2}\geqslant G^{L_1}$.
\end{proof}

\emph{Open reversibility} captures the idea that if a player has an option in a component against which the opponent, in a worst-case scenario, can respond in a way to improve their situation, the player chooses that option only if they intend to continue playing locally. If the intention is to play in another component, it is better to play in that component right from the start, thus avoiding the opportunity for the mentioned opponent's improvement.

\begin{theorem}[Open Reversibility]\label{thm:oreversibility}
If $G\in \Np$ is a game with $G^{L}\in \GL$ and $G^{LR}\neq \overline{\infty}$ such that $G\geqslant G^{LR}$, then $G=\{\GL\setminus\{G^{L}\},G^{LR\mathcal{L}} \mid \GR\}$.
\end{theorem}

\begin{proof}

We note that $G^{LR\mathcal{L}}$ is not empty. This happens because $G^{LR}$ cannot be $\infty$ due to the condition $G\geqslant G^{LR}$ and it cannot be $\overline{\infty}$ due to the condition $G^{LR}\neq \overline{\infty}$. Let $H=\{\GL\setminus\{G^{L}\},G^{LR\mathcal{L}} \mid \GR\}$. We want to prove that $H= G$, i.e., $G\geqslant H$ and  $H\geqslant G$.

First, we prove that $H\geqslant G^{LR}$. Let $X$ be a game different from $\infty$ and $\overline{\infty}$. If Left wins $G^{LR}+X$ with $G^{LR}+X^L$, then, by induction, she also wins $H+X$ with the move $H+X^L$. If Left wins $G^{LR}+X$ with a move $G^{LRL}+X$, she also wins $H+X$ with the move $G^{LRL}+X$, since $G^{LRL}\in H^\mathcal{L}$. If Right, playing first, wins $H+X$ with $H+X^R$, then, by induction, she also wins $G^{LR}+X$ with the move $G^{LR}+X^R$. Suppose now that Right, playing first, wins $H+X$ with $H^R+X$. Since $H^R\in \GR$, we have that $H^R+X\in(G+X)^\mathcal{R}$. Consequently, Right has a winning move in $G+X$, and due to the fact that $G\geqslant G^{LR}$, he also has a winning move in $G^{LR}+X$.

Second, we prove that $H\geqslant G$. The only case whose analysis is not straightforward is a possible winning move for Left from $G+X$ to $G^L+X$. In that case, it is mandatory for Left, playing first, to have a winning move in $G^{LR}+X$. As we already know that $H\geqslant G^{LR}$, Left must have a winning move in $H+X$ as well.

Third, we prove that $G\geqslant H$. The only cases whose analysis is not straightforward are possible winning moves for Left from $H+X$ to $G^{LRL}+X$. Since these winning moves are also available in $G^{LR}+X$, the fact that $G\geqslant G^{LR}$ implies that Left must also have a winning move in $G+X$.
\end{proof}

\emph{Absorbing reversibility} captures the idea that if a player has an option that allows the opponent to make a winning terminating move, in practice, that option works as if the player were directly committing suicide.

\begin{theorem}[Absorbing Reversibility]\label{thm:areversibility}
If $G\in \Np$ is a game with a Left option \linebreak $G^{L}=\{G^{L\mathcal{L}}\,|\,\overline{\infty}\}$, then $G=\{\GL\setminus\{G^{L}\},\overline{\infty} \mid \GR\}$.
\end{theorem}

\begin{proof}
Let $H=\{\GL\setminus\{G^{L}\},\overline{\infty} \mid \GR\}$. We want to prove that $H=G$, i.e., $G\geqslant H$ and  $H\geqslant G$.

First, we prove that $G\geqslant H$. Let $X$ be a game different from $\infty$ and $\overline{\infty}$. If Left wins $H+X$ with $H+X^L$, then, by induction, she also wins $G+X$ with the move $G+X^L$. Regarding moves other than in the component $X$, the only Left option in $H+X$ that escapes direct mimicry is $\overline{\infty}+X$. However, there is no need to consider that possibility, as it is a suicidal move for Left. To prove that the existence of a winning move for Right in $G+X$ implies the existence of a winning move for Right in $H+X$, it is sufficient to use induction and mimicry.

Second, we prove that $H\geqslant G$. Once again, winning moves in component $X$ are addressed by induction. Regarding moves other than in the component $X$, the only Left option in $G+X$ that escapes direct mimicry is $\{G^{L\mathcal{L}}\,|\,\overline{\infty}\}+X$. However, there is no need to consider that possibility, as Right can win by responding to $\overline{\infty}+X$. To prove that the existence of a winning move for Right in $H+X$ implies the existence of a winning move for Right in $G+X$, it is sufficient to use induction and mimicry.
\end{proof}

Similar to what is formalized in classical theory, a game form $G$ is reduced if none of its followers admits any of the three reductions stated in Theorems \ref{thm:domination}, \ref{thm:oreversibility}, and \ref{thm:areversibility} (also in terms of dual formulations according to the Right's perspective). A very interesting fact is that there can be two reduced forms, $G$ and $G'$, such that $G=G'$ and $G\not\cong G'$, meaning the uniqueness of reduced forms is lost. Once again, it is a consequence of checks and forcing sequences. Still, this fact is not a serious issue, since reduced forms can still be used in proofs and for the purpose of simplifications.
The following games constitute an example of this occurrence.\\

\vspace{-1cm}
$$G=\{\{\infty\,|\,\{0,*2\,|\,\overline{\infty}\},\{*,*3\,|\,\overline{\infty}\}\},\{\infty\,|\,\{0,*\,|\,\overline{\infty}\},\{*2,*3\,|\,\overline{\infty}\}\}\,\|\,\overline{\infty}\}\text{, and}$$ $$H=\{\{\infty\,|\,\{0,*2\,|\,\overline{\infty}\},\{*,*3\,|\,\overline{\infty}\}\},\{\infty\,|\,\{0,*3\,|\,\overline{\infty}\},\{*,*2\,|\,\overline{\infty}\}\}\,\|\,\overline{\infty}\}.$$

\vspace{0.3cm}
We have that $G\not\cong H$ because, although the game forms share a Left-check, there is a pair of distinct Left-checks. On the other hand, a careful application of Theorem \ref{thm:constructive} allows us to verify that $G\geqslant H$ and $H\geqslant G$, i.e., $G=H$. The same theorem allows us to check that there are no reducible followers in either of the two game forms.

In practice, what happens is that, although there are different checks, the control provided by its forcing nature is essentially the same. The following definition formalizes the concept of reduced form.

\begin{definition}[Reduced Forms]\label{def:reduced}
An affine game form $G$ is a \emph{reduced form} if the following two conditions are satisfied:
\begin{itemize}
  \item None of the followers of $G$ admits any of the three reductions stated in Theorems~\ref{thm:domination}, \ref{thm:oreversibility}, and \ref{thm:areversibility};
  \item There is no $G'$ such that $G'=G$ and $\tilde{b}(G')<\tilde{b}(G)$.
\end{itemize}
\end{definition}

\section{Conway-embedding and invertibility}
\label{sec:cembedding}

\vspace{0.3cm}
It has been emphasized that the effect of the absence of moves (empty set) can be achieved in $\Np$ by choosing $G^\mathcal{L}=\{\overline{\infty}\}$ or $G^\mathcal{R}=\{\infty\}$ for the player without moves, as a single suicidal move is the same as having no moves at all. In this section, we formalize this idea, proving that the traditional Conway values are order embedded in the affine structure. In fact, we will prove that the group structure of Conway values is effectively the substructure of the invertible elements of $\Np$.

\begin{definition}[Conway Forms and Games]\label{def:conw}
A game form $G\in \Np$ is a \emph{Conway form} if $G$ is distinct from $\infty$ and $\overline{\infty}$, and has no checks as followers. Let $\Nc\subseteq\Np$ denote the substructure of Conway forms. A game form $G\in \Np$ is a \emph{Conway game} if it is equal to a Conway form in terms of equivalence of games. A game form $G\in \Np$ is a \emph{pressing game} if it is not a Conway game.
\end{definition}

\begin{example}
The game $G=\{\{\overline{\infty}\,|\,\infty\}\,|\,\{\overline{\infty}\,|\,\infty\}\} =\{0\,|\,0\}=*$ is a Conway form (the game form has no checks as followers). The game $G'=\{\{\infty\,|\,*\}\,|\,\{*\,|\,\overline{\infty}\}\}$ is not a Conway form because it has checks as followers. However, since $G'=G$, it is a Conway game.
\end{example}

\begin{definition}\label{def:conmap}
The \emph{Conway map} $c:\Nc\rightarrow\Npc$ is recursively defined in the following way:
\[c(G)=  \begin{cases}
0=\{\emptyset\,|\,\emptyset\} \quad\quad\quad\quad\quad\textrm{ if }G=0=\{\overline{\infty}\,|\,\infty\};\\
\{\emptyset\,|\,c(\GR)\} \quad\quad\quad\quad\,\,\,\,\textrm{ if }G=\{\overline{\infty}\,|\,\GR\}\textrm{ and }\GR\neq\{\infty\};\\
\{c(\GL)\,|\,\emptyset\} \quad\quad\quad\quad\,\,\,\,\,\textrm{ if }G=\{\GL\,|\,\infty\}\textrm{ and }\GL\neq\{\overline{\infty}\};\\
\{c(\GL)\,|\,c(\GR)\} \quad\quad\,\,\,\,\textrm{ otherwise.}
\end{cases}
\]

Here, $c(\GL)$ means $\{c(G^L)|G^L\in\GL\}$, and the same for $c(\GR)$.
\end{definition}

\begin{lemma}\label{lem:conjcon}
If $G\in\Nc$, then $c(-G)=-c(G)$.
\end{lemma}

\begin{proof}
It is an immediate consequence of the Definitions \ref{def:conjugate} and \ref{def:conmap}.
\end{proof}

\begin{theorem}\label{thm:onetoone}
The Conway map is a one-to-one map.
\end{theorem}

\begin{proof}
We define first the following map from $\Npc$ to $\Npc^{\text{C}}$:
\[c^{-1}(G)=  \begin{cases}
0=\{\overline{\infty}\,|\,\infty\} \quad\quad\quad\quad\quad\,\,\,\;\textrm{ \emph{if} }G=0=\{\emptyset\,|\,\emptyset\};\\
\{\overline{\infty}\,|\,c^{-1}(\GR)\} \quad\quad\quad\quad\,\,\,\,\textrm{ \emph{if} }G=\{\emptyset\,|\,\GR\}\textrm{ \emph{and} }\GR\neq\emptyset;\\
\{c^{-1}(\GL)\,|\,\infty\} \quad\quad\quad\quad\,\,\,\,\,\textrm{ \emph{if} }G=\{\GL\,|\,\emptyset\}\textrm{ \emph{and} }\GL\neq\emptyset;\\
\{c^{-1}(\GL)\,|\,c^{-1}(\GR)\} \quad\quad\textrm{ \emph{otherwise}.}
\end{cases}
\]

Now, we prove that $c^{-1}c=\mathbf{1}$, where $\mathbf{1}$ is the identity map. Proving that $cc^{-1}$ is also the identity is entirely analogous, so we refrain from doing it.  The proof is carried out for each of the four cases of the definitions of the maps, using induction:\\

\noindent
$\bullet\,\,cc^{-1}(0)=cc^{-1}(\{\emptyset\,|\,\emptyset\})=c(\{\overline{\infty}\,|\,\infty\})=\{\emptyset\,|\,\emptyset\}=0$;\\

\noindent
$\bullet\,\,cc^{-1}(\{\emptyset\,|\,\GR\})=c(\{\overline{\infty}\,|\,c^{-1}(\GR)\})=\{\emptyset\,|\,cc^{-1}(\GR)\}\underbrace{=}_{\text{ind.}}\{\emptyset\,|\,\GR\}$;\\

\noindent
$\bullet\,\,cc^{-1}(\{\GL\,|\,\emptyset\})=c(\{c^{-1}(\GR)\,|\,\infty\})=\{cc^{-1}(\GL)\,|\,\emptyset\}\underbrace{=}_{\text{ind.}}\{\GL\,|\,\emptyset\}$;\\

\noindent
$\bullet\,\,cc^{-1}(\{\GL\,|\,\GR\})=c(\{c^{-1}(\GL)\,|\,c^{-1}(\GR)\})=\{cc^{-1}(\GL)\,|\,cc^{-1}(\GR)\}\underbrace{=}_{\text{ind.}}\{\GL\,|\,\GR\}$.\\
\end{proof}

\begin{observation}
The game tree of a Conway form has no small arrows and is identical to the game tree of the corresponding form in $\Npc$.
\end{observation}

The following theorem constitutes a fundamental result, directly related to Theorem~\ref{thm:invcomparison}. Since Conway games are invertible, with their inverses being their conjugates, the comparison involving these elements can be done by playing, in the traditional way. From now on, for ease, we write $G-H$ instead of $G+(-H)$, where $-H$ is the conjugate of $H$.

One can also establish a third parallel with what happens on the extended real number line. In that structure, infinities are the only non-invertible elements. In $\Np$, non-invertible elements are those where there is some pressure exerted by an infinity in one of their followers (this also explains the choice of the term ``pressing game'').

\begin{theorem}\label{thm:invertible}
Let $G\in \Np$. We have that $G$ is invertible if and only if $G$ is a Conway game. Moreover, if $G$ is a Conway game, then $G-G=0$.
\end{theorem}

\begin{proof}
($\Leftarrow:$) Suppose first that $G$ is a Conway form. If $G=0$, then $-G=0$ and $G-G=0$. Otherwise, according to Corollary~\ref{cor:ftnp2}, it is sufficient to verify that $G-G$ is a $\mathcal{P}$-position. If Left, playing first, chooses $G^L-G$, because this game is not $\infty$ ($G$ is not a check), Right can answer with $G^L-G^L$ and, by induction, because $G^L$ is a Conway form with no checks as followers, that option is equal to zero. Because of that, by Corollary~\ref{cor:ftnp2}, that option is a $\mathcal{P}$-position, and Right wins. Analogous arguments work for the other options of the first player, and so, $G-G$ is a $\mathcal{P}$-position. Again, by Corollary~\ref{cor:ftnp2}, $G-G=0$.

Suppose now that $G$ is a Conway game, but not a Conway form. Because it is a Conway game, by definition, it is equal to some $G'\in\Npc^{\text{C}}$. The first paragraph proved that $G'-G'=0$. Also, by symmetry, $-G$ is equal to $-G'$. Therefore, $G'-G'=0$ implies $G-G=0$.\\

\noindent
($\Rightarrow:$) Let $G\in \Np$ be a reduced form of a pressing game. Suppose that there is $H\in \Np$ such that $G+H=0$, i.e., $G+H\in\mathcal{P}$. Assume that $H$ is also a reduced form and that $\tilde{b}(G)+\tilde{b}(H)$ is the smallest possible for $G$ and $H$ satisfying these conditions.

If $G$ has no options, $G$ is either $\infty$ or $\overline{\infty}$, and $G+H\not\in\mathcal{P}$, which is a contradiction. If all options of $G$ are Conway games, then $G$ itself must be a Conway game, leading to a contradiction as well. If $G=\{\overline{\infty}\,|\,\infty\}$, then $G=0$, leading to another contradiction. So, without loss of generality, let $G^{L_1}+H$ be a move such that $G^{L_1}\neq \overline{\infty}$ is a pressing game.

Against $G^{L_1}+H$, Right must have a winning response less than or equal to zero. Suppose that Right's winning answer is some $G^{L_1R}+H\leqslant 0$. If so, by adding $G$ to both sides, $G^{L_1R}\leqslant G$ is obtained. This contradicts the assumption (reduced forms), as if $G^{L_1R}$ is $\overline{\infty}$, then $G^{L_1}$ is absorbing reversible, and if $G^{L_1R}$ is not $\overline{\infty}$, then $G^{L_1}$ is open reversible.

Suppose now Right's winning response is some $G^{L_1}+H^{R_1}\leqslant 0$. If $G^{L_1}+H^{R_1}=0$, then it contradicts the assumption once more (smallest possible $\tilde{b}(G)+\tilde{b}(H)$), since $G^{L_1}$ is a pressing game. Therefore, we have $G^{L_1}+H^{R_1}< 0$.

In order to initiate a ``carousel argument'', often used in CGT,
consider the first move by Right in $G+H$ to $G+H^{R_1}$. By similar reasons, Left must have an option $G^{L_2}+H^{R_1}\geqslant 0$. If $H^{R_1}$ is a Conway game, it is invertible, and we have $G^{L_1}< -H^{R_1}\leqslant G^{L_2}$. In this case, $G^{L_1}$ is a dominated option, contradicting the assumption (reduced forms). On the other hand, due to the fact that $H^{R_1}$ is not a Conway game, we cannot have $G^{L_2}+H^{R_1}=0$ as it once again contradicts the assumption (smallest possible $\tilde{b}(G)+\tilde{b}(H)$). Thus, we have $G^{L_2}+H^{R_1}> 0$. Turning the carousel, a first move by Left in $G+H$ to $G^{L_2}+H$ must be answered with some $G^{L_2}+H^{R_2}< 0$. And turning again, a first move by Right in $G+H$ to $G+H^{R_2}$ must be answered with some $G^{L_3}+H^{R_2}> 0$. Note that, in this case, it is not possible to have $G^{L_1}+H^{R_2}> 0$, since we would have
\begin{itemize}
  \item $G^{L_1}+H^{R_1}<0$;
  \item $-G^{L_2}-H^{R_1}<0$ (since $G^{L_2}+H^{R_1}>0$);
  \item $G^{L_2}+H^{R_2}<0$;
  \item $-G^{L_1}-H^{R_2}<0$ (since $G^{L_1}+H^{R_2}>0$).
\end{itemize}
Combining these inequalities by making use of Theorem  \ref{thm:strict} would lead to the inequality $(G^{L_1}+H^{R_1}+G^{L_2}+H^{R_2})+(-G^{L_1}-H^{R_1}-G^{L_2}-H^{R_2})< 0$. That cannot occur, given that the specified game, due to its symmetry, can belong to $\mathcal{P}$ or $\mathcal{N}$, but never to $\mathcal{R}$. Entirely analogous arguments are sufficient to show that the winning responses must always be obtained with options different from the previous ones. Since the carousel cannot rotate indefinitely, it is not possible to have $G$ and $H$ such that $G+H\in\mathcal{P}$, i.e., such that $G+H=0$.
\end{proof}

We conclude this section by proving the Conway-embedding.

\begin{theorem}[Conway-Embedding]\label{thm:embedding}
For any $G,H\in \Nc$, if $G\geqslant H$ then $c(G)\geqslant c(H)$, where the first is the order relation of $\Np$ and the second is the order relation of $\Npc$. Similarly, for any $G,H\in \Npc$, if $G\geqslant H$ then $c^{-1}(G)\geqslant c^{-1}(H)$,  where the first is the order relation of $\Npc$ and the second is the order relation of $\Np$.
\end{theorem}

\begin{proof}
Let $G$ and $H$ be two Conway forms such that $G\geqslant H$. According to Theorem~\ref{thm:invertible} $H$ is invertible and, by Theorem~\ref{thm:invcomparison}, we have $G-H\in\mathcal{L}\cup\mathcal{P}$. Let us verify that, in $\Npc$, we also have $c(G)+c(-H)\in\mathcal{L}\cup\mathcal{P}$. With that verified, Lemma \ref{lem:conjcon} ensures that $c(G)+c(-H)\geqslant 0\iff c(G)-c(H)\geqslant 0\iff c(G)\geqslant c(H)$. Suppose that Right moves to some $c(G^R)+c(-H)$. Considering the game $G^R-H\in\Np$, Left must have a winning move, which can be either some $G^{RL}-H \geqslant 0$ or some $G^R-H^R \geqslant 0$.  Concerning $\Npc$, by induction, in the first case, we have $c(G^{RL}) + c(-H) \geqslant 0$, and in the second case, $c(G^R) + c(-H^R) \geqslant 0$. In other words, Left has a winning move in $c(G^R)+c(-H)$. If Right moves to some $c(G)+c(-H^L)$, the inductive argument works exactly the same way, and Left has a winning response.  We can conclude that $c(G)\geqslant c(H)$.

The proof for the second implication in the theorem can be carried out by induction in a completely analogous manner.
\end{proof}

\section{Classification of affine games}
\label{sec:classification}

It is well-known that the games of the classical structure $\Npc$ can be classified as cold, tepid, or hot. Moreover, this classification is based on the concept of number and the related Number Avoidance Theorem.

When a game is a number (cold), both players have a negative incentive, meaning they are in \emph{zugzwang}\footnote{Zugzwang (from German ``compulsion to move'') refers to a situation where a player is placed at a disadvantage because they must make a move.} in the sense that making a move leads to a ``loss'' in terms of guaranteed moves (if there are any moves available at all).  For example, $\frac{1}{2}=\{0\,|\,1\}$; if Left plays, she reduces $\frac{1}{2}$ to $0$, and if Right plays, he increases $\frac{1}{2}$ to $1$. In other words, both players prefer the opponent to play rather than playing themselves. When a game is hot, the opposite occurs. There is urgency in the sense that making a move leads to a ``gain'' of guaranteed moves. An example of a hot game is $\pm1=\{1\,|\,-1\}$. When a game is tepid, a move brings neither gains nor losses in terms of guaranteed moves. Playing may be important for the sake of making the last move \emph{per se}, but not in terms of gaining guaranteed moves by having the right to move. An example of a tepid game is $\{1\,|\,1\}$; Left already got one guaranteed move regardless of whether she moves or not. Thus, intuitively, the Number Avoidance Theorem establishes that a player should not play on numbers in order to avoid losing what is already guaranteed; playing in hot games or tepid games avoids such a loss.

In this section, we extend these concepts, considering that in $\Np$, infinities come into play. For this purpose, it is important to start with the Number Avoidance Theorem, which still holds in affine normal play.

\begin{theorem}[Number Avoidance Theorem]\label{thm:avoidance}
Let $G,x\in \Np$, and suppose $x$ is a number and $G$ is not. If Left (resp. Right) has a winning move on $G+x$, then she has a winning move of the form $G^L+x$ (resp. $G^R+x$).
\end{theorem}

\begin{proof}
We may assume that $x$ is in the classical reduced form via Conway embedding, since the outcomes of the considered sums do not depend on its form. If $G=\infty$ or $G=\overline{\infty}$ there are no moves and there is nothing to prove. Now suppose that $G+x^L$ is a winning move for Left. According to Theorem~\ref{thm:ftnp}, we have $G+x^L\geqslant 0$. But $G$ is not equal to a number, thus $G\neq -x^L$ ($x^L$ is invertible), so in fact $G+x^L> 0$ is mandatory. Again by Theorem \ref{thm:ftnp}, Left has a winning move in $G+x^L$. By induction, we may assume that it is some $G^L+x^L\geqslant 0$. Since $x$ is in the classical reduced form, we have $x^L<x$ and $G^L+x> G^L+x^L\geqslant 0$. Hence, $G^L+x$ is a winning move for Left in $G+x$. The proof from the Right's perspective is analogous.
\end{proof}

The concept of \emph{stop} in classical theory is based on the assumption that two players move in a component, and, considering the Number Avoidance Theorem, stop playing on that component as soon as a number has been reached. Left attempts to have this stopping number be as large as possible, and Right wants it to be as small as possible. Our next step consists of defining this concept for $\Np$, in order to classify affine games similarly to that done in classical theory, with just a few additions due to infinities. An obvious difference when comparing with the classical structure is that the stopping point may not be a number, but rather $\infty$ or $\overline{\infty}$. This is because there is the possibility for one of the players to force a terminating move.

\begin{definition}[Stops]\label{def:stops}
Left and Right \emph{stops} of $G$ are recursively defined by
\[LS(G)=  \begin{cases}
\infty \quad\quad\quad\quad\quad\quad\quad\quad\quad\quad\,\,\,\,\textrm{ if }G=\infty\\
\overline{\infty} \quad\quad\quad\quad\quad\quad\quad\quad\quad\quad\,\,\,\,\textrm{ if }G=\overline{\infty}\\
x \quad\quad\quad\quad\quad\quad\quad\quad\quad\quad\quad\,\textrm{ if }G=x\textrm{ is a number}\\
\max \{RS(G^L)\,|\,G^L\in G^{\mathcal{L}}\} \quad\textrm{ otherwise}
\end{cases}
\]
\[RS(G)=  \begin{cases}
\infty \quad\quad\quad\quad\quad\quad\quad\quad\quad\quad\,\,\,\,\textrm{ if }G=\infty\\
\overline{\infty} \quad\quad\quad\quad\quad\quad\quad\quad\quad\quad\,\,\,\,\textrm{ if }G=\overline{\infty}\\
x \quad\quad\quad\quad\quad\quad\quad\quad\quad\quad\quad\,\textrm{ if }G=x\textrm{ is a number}\\
\min \{LS(G^R)\,|\,G^R\in G^{\mathcal{R}}\} \quad\textrm{ otherwise.}
\end{cases}
\]
\end{definition}

\begin{definition}\label{def:classification}
If $G$ is an element of $\Np$, then
\begin{itemize}
  \item $G$ is \emph{cold} if it is a number;
  \item $G$ is \emph{tepid} if the stops are finite and $LS(G)=RS(G)$, but $G$ is not a number;
  \item $G$ is \emph{hot} if $LS(G)>RS(G)$;
  \item $G$ is \emph{bubbling} if both stops are $\infty$ or both stops are $\overline{\infty}$.
\end{itemize}
\end{definition}

\begin{example}\label{ex:classification}
The following list has an example of each of the four types of affine games.
\begin{itemize}
  \item $G=\{\{\infty\,|\,-1*\}\,|\,0\}$ is a number since $\{\infty\,|\,-1*\}$ is reversible through $-1*$, and $G$ reduces to $-\frac{1}{2}=\{-1\,|\,0\}$.
  \item $G=\{\{\infty\,|\,0\}\,|\,0\}$ is tepid since both stops are finite and $LS(G)=RS(G)=0$. It is a reduced form of a pressing game, as it is not equivalent to any Conway form.
  \item $G=\{\{\infty\,|\,1\}\,|\,0\}$ is hot since $LS(G)=1$, $RS(G)=0$, and $LS(G)>RS(G)$. The game $G$ is a pressing game, as it is not equivalent to any Conway form.
  \item $G=\{\infty\,|\,\infty\}$ is bubbling since $LS(G)=RS(G)=\infty$. It is inevitable that $G$ ends with Left delivering a checkmate, regardless of whether she plays first or not. In this sense, $G$ has already reached an irreversible boiling state, which explains the choice of the word ``bubbling''.
\end{itemize}
\end{example}

The class of \emph{mate games} is a subclass of bubbling games. Its formalization is needed given the frequency with which these situations occur. In an $n$-Leftmate ($n$-Rightmate), Left (Right) is guaranteed the possibility of checkmating in $n$ moves, without the opponent being able to defend. However, the checkmate cannot occur more quickly.

\begin{definition}[Mate Games]\label{def:mategame}
Let $n$ be a nonnegative integer. An $n$-\emph{Leftmate} $\infty_n$ is recursively defined by
\[\infty_n=  \begin{cases}
\infty, \quad\quad\quad\quad\quad\textrm{ if }n=0;\\
\{\infty_{n-1}\,|\,\infty\}, \quad\textrm{ if }n>0.
\end{cases}
\]
The definition of $n$-\emph{Rightmate}, denoted by $\overline{\infty}_n$, is symmetric.
\end{definition}

\begin{observation}
The game $\{\infty_3\,|\,\infty_8\}$ is a bubbling game, but it is not a $n$-Leftmate. The reason for this is that the inevitable checkmate can occur at different speeds, depending on whether it is Left or Right who starts. If Right starts, he can delay his fateful destiny in the component a little, and this can be crucial in certain disjunctive sums.
\end{observation}

\noindent
\textbf{Nomenclature}: When a game is tepid and the stops are equal to zero, it is an \emph{infinitesimal}. When a game is tepid and the stops are not zero, it is a \emph{translation of an infinitesimal}. When a game is hot and at least one of the stops is infinite, it is \emph{scalding}.

\begin{observation}
There are infinitesimals that are not Conway games, such as $\{\{\infty\,|\,0\}\,|\,0\}$  already used in Example \ref{ex:classification} and materialized in Figure \ref{fig9}. On the other hand, there are no scalding games that are Conway games. For example, $\{\infty_3\,|\,0\}$ is a scalding game. In this game, Right is not yet condemned to being checkmated by the opponent as long as he retains the right to make the move. The temperature is infinite, but the boiling state of a bubbling game has not been reached yet, as in a scalding game neither player is definitively condemned to be checkmated.
\end{observation}

Now, let us move on to an important concept regarding affine normal play. If we consider the classical structure, all reductions are based on the fundamental idea of \emph{moves that may never be made}, given their uselessness. A dominated option \emph{never} needs to be used. Regarding a reversible option, since it is only chosen with the intention of continuing to play locally, all moves that are not related to this local continuation are \emph{never} used. For the affine structure $\Np$, something entirely new happens. There are checks that must \emph{always} be used. In other words, as long as the player is not mortally threatened in another component, the first thing they must do is deliver a check in a specific component. In affine normal play, in addition to ``never'', ``always'' also comes into play. Thus, there are forms that, although they cannot be reduced, already contain a certain type of ``intrinsic  reduction'', since they will be the target of a move at the first opportunity. In a way, these components are ephemeral. The following definition formalizes this idea, capturing the concept of \emph{hammerzug}\footnote{Hammerzug (from German ``move with a hammer'') is a move that, when not under lethal threat, a player should always  make, as it can only help without having any negative consequences.}. When a hammerzug is a component of a disjunctive sum, a player must deliver a check in that component as long as they are not threatened.

\begin{definition}[Hammerzug]\label{def:hammerzug} A game $G\in \Np$ is a \emph{Left-hammerzug} if there exists a Left-check $G^L=\{\infty\,|\,G^{L\mathcal{R}}\}$ such that $G^{LR}\geqslant G$ for all $G^{LR} \in G^{L\mathcal{R}}$.  The definition of Right-hammerzug is similar. A game $G\in \Np$ is a \emph{hammerzug} if it is both a Left-hammerzug and a Right-hammerzug.\end{definition}

\begin{example}
The game $G=\{\{\infty\,|\,1\}\,|\,0\}=\{\circlearrowright^{1}\,|\,0\}$ is a Left-hammerzug since $1>G$. The game $G=\{\{\infty\,|\,1\}\,|\,\{-1\,|\,\overline{\infty}\}\}=\{\circlearrowright^{1}\,|\,\circlearrowleft^{-1}\}$ is a hammerzug since $1>G>-1$.
\end{example}

\begin{theorem}[Extreme Urgency]\label{thm:extremeham} If $G\in \Np$ is a Left-hammerzug, then, for every Right-quiet game $X$, $o^L(G+X)=o^R(G^L+X)$ where $G^L$ is the Left-check specified in Definition~\ref{def:hammerzug}. There is a dual result in terms of Right-hammerzugs.
\end{theorem}

\begin{proof}
We prove that $o^L(G+X)\geqslant o^R(G^L+X)$ and  $o^R(G^L+X)\geqslant o^L(G+X)$. The first inequality is evident, since, in a worst-case scenario, Left can move to $G^L+X$ in $G+X$. Regarding the second inequality, we have $o^R(G^L+X)=o^L(G^{LR}+X)$ where $G^{LR}$ is the best defense for Right against the Left-check. According to Definition \ref{def:hammerzug}, $o^L(G^{LR}+X)\geqslant o^L(G+X)$, thus $o^R(G^L+X)\geqslant o^L(G+X)$.
\end{proof}

Returning to the classical structure $\Npc$, there are classes of infinitesimals of great relevance. The powers of up $\uparrow\,=\{0\,|\,*\}$, $\uparrow^2\,=\{0\,|\,\downarrow+*\}$, $\uparrow^3\,=\{0\,|\,\downarrow+\downarrow^2+*\}$,\,\ldots,\;are infinitesimals with respect to the previous ones. This means that $n.\!\!\uparrow^2<\uparrow$, regardless of the number of copies $n$. The same applies to $n.\!\!\uparrow^3<\uparrow^2$, $n.\!\!\uparrow^4<\uparrow^3$, and so on. The tinies and minies exhibit a similar behavior. The game tiny, denoted by $\cgtiny{1}$, is the positive game $\{0\,|\,\{0\,|\,-1\}\}$. Similarly, the game miny, denoted by $\cgminy{1}$, is the negative game $\{\{1\,|\,0\}\,|\,0\}$. In general, given a positive integer $n$, we have $\cgtiny{n}=\{0\,|\,\{0\,|\,-n\}\}$ and $\cgminy{n}=\{\{n\,|\,0\}\,|\,0\}$. These are infinitesimals of a ``magnitude'' different from the powers of up. Saying the same rigorously, in addition to the facts $n.\;\cgtiny{2}<\cgtiny{1}$, $n.\;\cgtiny{3}<\cgtiny{2}$, and so on, $\cgtiny{1}$ is infinitesimal with respect to $\uparrow^k$, regardless of $k$. Now, for a better understanding of what follows, it is helpful to make a brief observation about tinies and minies.

In a game like $\{0\,|\,-80\}+\cgminy{100}$, Left, playing first, wins by playing in the infinitesimal component and not in the hot component. This happens because Left can make a significant threat in $\cgminy{100}$. This threat is powerful enough to counteract the gain that Right can achieve in the other component. However, in game practice, it is not common to have a component with the value $\cgminy{100}$. This is because it would require a threat of $100$ in ``territorial'' terms to be defended with a move to $0$, a situation that would require a ruleset with very special characteristics. However, in $\Np$, ``powerful threat'' can be replaced by ``terminal threat'', and given that mate threats are very frequent, there is a certain type of games that are very common. The ultimate tiny (miny) can be defined as follows.

\begin{definition}[Pathetic Infinitesimals]\label{def:pathetic}
The game $\{0\,|\,\{0\,|\,\overline{\infty}\}\}=\{0\,|\,\circlearrowleft^{0}\}$ is called \emph{pathetic tiny} and denoted by $\cgtiny{\infty}$.  The game $\{\{\infty\,|\,0\}\,|\,0\}=\{\circlearrowright^{0}\,|\,0\}$ is called \emph{pathetic miny} and denoted by $\cgminy{\infty}$.
\end{definition}

It is trivial to verify that $\cgtiny{\infty}\!>0\!>\cgminy{\infty}$. Furthermore, the name ``pathetic'' is easily explainable. As observed in the previous paragraph, the presence of $\cgminy{100}$ in a disjunctive sum tends to be ephemeral, as in most cases, Left uses her threat to replace, to her advantage, a negative game by zero. Once again, in $\Np$, ``in most cases'' can be transformed into ``always'', as the following theorem establishes -- this is the first of three notable facts about the pathetic infinitesimals. The ``pathetic nature'' of these positive and negative values is due to their ephemerality. Both a pathetic tiny and a pathetic miny have a very limited lifespan in a disjunctive sum, as, at the first opportunity, one of the players makes them disappear.

\begin{theorem}\label{thm:tinyham} The pathetic tiny $\cgtiny{\infty}$ is a Right-hammerzug. Similarly, the pathetic miny $\cgminy{\infty}$ is a Left-hammerzug.
\end{theorem}

\begin{proof}
This is a direct consequence of Definition \ref{def:hammerzug} and $\cgtiny{\infty}\!>0\!>\cgminy{\infty}$.
\end{proof}

The second notable fact concerns a certain ``lack of continuity'' exhibited by $\Np$. The smallest positive element and the largest negative element are effectively reached.

\begin{theorem}\label{thm:smallest} The game $\cgtiny{\infty}$ is the smallest positive game of all, i.e., if $G\in \Np$ is such that $G>0$, then $G\geqslant \cgtiny{\infty}$. Similarly, the game $\cgminy{\infty}$ is the largest negative game of all.\end{theorem}

\begin{proof} If $G=\infty$, there is nothing to prove, so assume otherwise. According to Theorem~\ref{thm:ftnp}, $G>0$ implies the existence of some $G^L$ such that $G^L\geqslant 0$ (Fact 1). On the other hand, the same theorem guarantees that for each Right option $G^R\neq\infty$, there exists $G^{RL}$ such that $G^{RL}\geqslant 0$. For these cases, it is straightforward to verify $\mathbf{M}(G^R,\{0\,|\,\overline{\infty}\})$, since for the Left option to $0$ in $\{0\,|\,\overline{\infty}\}$, there exists some $G^{RL}\geqslant 0$. Consequently, according to Theorem \ref{thm:traditional}, we have $G^R\geqslant\{0\,|\,\overline{\infty}\}$ (Fact 2).

With the help of these two facts, let us check that the pair $(G,\cgtiny{\infty})$ satisfies also the maintenance property as established in Definition \ref{def:maintenance}, and, therefore, by Theorem \ref{thm:traditional}, $G\geqslant \cgtiny{\infty}$. For the Left move to $0$ in $\cgtiny{\infty}$, there exists some $G^L\geqslant 0$ (Fact 1). For each Right option $G^R$, we have $G^R\geqslant\{0\,|\,\overline{\infty}\}$ (Fact 2).

Proving that $\cgminy{\infty}$ is the largest negative game of all is analogous.
\end{proof}

To better understand the third notable fact, recall that $\uparrow^2$ is infinitesimal with respect to $\uparrow$, i.e., any number of copies of $\uparrow^2$ is less than a single copy of $\uparrow$. But, it is important to note that this is entirely different from saying that having two copies or just one copy of $\uparrow^2$ is irrelevant. If only $\downarrow$ is present, Left loses all disjunctive sums of the form $n.\!\!\uparrow^2+\downarrow$. However, in other sums, having two copies instead of one can be the difference between winning or losing. In fact, we have $2.\!\!\uparrow^2>\uparrow^2$. Furthermore, in the classical structure $\Npc$, if $G>0$, then $2.G>G$.

The terminating games $\infty$ and $\overline{\infty}$ bring the existence of nontrivial finite indempotents. Having many or few copies of a pathetic tiny is irrelevant since, at the slightest opportunity, Right disposes of them all at once through a sequence of checks.

\begin{theorem}[Pathetic Idempotents]\label{thm:idempotents} The games $\cgtiny{\infty}$ and $\cgminy{\infty}$ are idempotents.
\end{theorem}

\begin{proof}
We prove only the equality $\cgtiny{\infty}+\cgtiny{\infty}=\cgtiny{\infty}$, since the proof for $\cgminy{\infty}+\cgminy{\infty}=\cgminy{\infty}$ is symmetric.

The inequality $\cgtiny{\infty}+\cgtiny{\infty}\geqslant\cgtiny{\infty}$ is a trivial consequence of Theorem \ref{thm:compadd}, since $\cgtiny{\infty}>0$.

Regarding the inequality $\cgtiny{\infty}+\cgtiny{\infty}\leqslant\cgtiny{\infty}$, we begin by observing that the game $\cgtiny{\infty}+\cgtiny{\infty}$ is $\{\cgtiny{\infty}\,|\,\{\cgtiny{\infty}\,|\,\overline{\infty}\}\}$. Additionally, $\mathbf{M}(\{0\,|\,\overline{\infty}\},\{\cgtiny{\infty}\,|\,\overline{\infty}\})$ because, for the Left move to $\cgtiny{\infty}$ in $\{\cgtiny{\infty}\,|\,\overline{\infty}\}$, there is $\{0\,|\,\overline{\infty}\}\geqslant \{0\,|\,\overline{\infty}\}$, and for the Right move to $\overline{\infty}$ in $\{0\,|\,\overline{\infty}\}$, there is $\overline{\infty}\geqslant \overline{\infty}$. Hence, according to Theorem \ref{thm:traditional}, we have $\{0\,|\,\overline{\infty}\}\geqslant\{\cgtiny{\infty}\,|\,\overline{\infty}\}$ (Fact 1). Moreover,
$\mathbf{M}(\cgtiny{\infty},\{0\,|\,\overline{\infty}\})$  because, for the Left move to $0$ in $\{0\,|\,\overline{\infty}\}$, there is $0\geqslant 0$, and for the Right move to $\{0\,|\,\overline{\infty}\}$ in $\cgtiny{\infty}$, there is $\{0\,|\,\overline{\infty}\}\geqslant \overline{\infty}$. Hence, according to Theorem \ref{thm:traditional}, we have $\cgtiny{\infty}\geqslant\{0\,|\,\overline{\infty}\}$ (Fact 2).

With the help of these facts, let us check that $\mathbf{M}(\cgtiny{\infty},\{\cgtiny{\infty}\,|\,\{\cgtiny{\infty}\,|\,\overline{\infty}\}\})$. For the Left move to $\cgtiny{\infty}$ in $\{\cgtiny{\infty}\,|\,\{\cgtiny{\infty}\,|\,\overline{\infty}\}\}$, there is $\cgtiny{\infty}\geqslant \{0\,|\,\overline{\infty}\}$ (Fact 2). For the Right move to $\{0\,|\,\overline{\infty}\}$ in $\cgtiny{\infty}$, there is $\{0\,|\,\overline{\infty}\}\geqslant\{\cgtiny{\infty}\,|\,\overline{\infty}\}$ (Fact 1). Hence, according to Theorem \ref{thm:traditional}, we have $\cgtiny{\infty}\geqslant\{\cgtiny{\infty}\,|\,\{\cgtiny{\infty}\,|\,\overline{\infty}\}\}$.
\end{proof}

\begin{observation}
Although, by Theorem \ref{thm:strict}, it is true in the case where $J=0$, in general, it is not true that $G>0$ and $H\geqslant J$ implies $G+H> J$. Just consider $G=H=J=\cgtiny{\infty}$ and Theorem \ref{thm:idempotents} to observe this fact.
\end{observation}

Figure \ref{fig12} contains a visualization of the affine game line. On one hand, the Conway game line is closed with $\infty$ and $\overline{\infty}$ (bold). On the other hand, the descending set of lines respecting some scales of infinitesimals is also closed with $\cgtiny{\infty}$ and $\cgminy{\infty}$ (bold). Note that the way to close this set is not with a line but only with two elements. This happens because these are idempotent. The game $\moon\!\!=\{0,\circlearrowright^{0}\,|\,0,\circlearrowleft^{0}\}$ will be discussed later.

\begin{figure}[!htb]
\begin{center}
\scalebox{0.6}{
\includegraphics[scale=1]{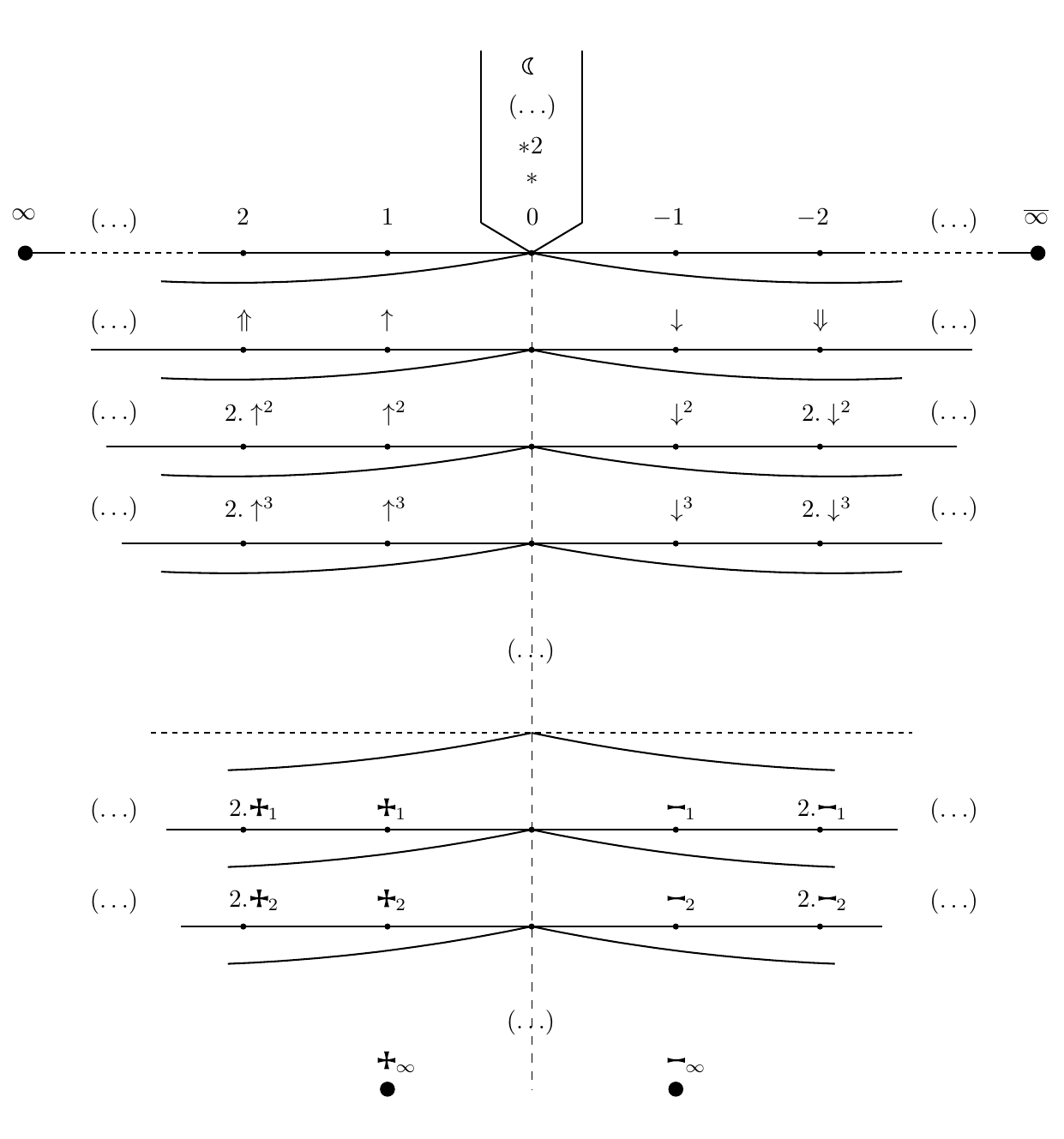}}
\end{center}
\vspace{-0.4cm}
\caption{Affine game line.}
\label{fig12}
\end{figure}
\clearpage
\section{Lattice structure of affine games born by day $n$}
\label{sec:lattice}

In this section, we discuss the lattice structure of the affine games born by day $n$. To this end, it is important to make a preliminary observation about the notion of birthday. The formalization presented in Definition \ref{def:recursion} is commonly referred to as the ``formal birthday''. Since the equality of games is an equivalence relation, one can consider the quotient space $\Np/_=$ and it makes sense to define the \emph{birthday} of $G$, $b(G)$, as the ``birthday of the equivalence class of $G$''. More precisely, the birthday of $G$ is the smallest formal birthday among the games $G'$ that are equivalent to $G$. We observe that, in $\mathbb{A}_{n}$, all games have a formal birthday less than or equal to $n$, but the formal birthday can be greater than the birthday. For example, $\{-2\,|\,2\}\in\mathbb{A}_{3}$, with a formal birthday of $3$ and a birthday of $0$, since $\{-2\,|\,2\}=0$.

We will prove that the affine games born by day $n$ form a lattice structure. This has two interpretations: the poset $(\mathbb{A}_{n},\geqslant)$ is a lattice structure where many equivalent games overlap in the same node, or the poset $(\mathbb{G}_{n},\geqslant)$, consisting of equivalence classes with a birthday less than or equal to $n$, is a lattice structure without overlaps. We opt for the second, and whenever we write $G\in\mathbb{G}_{n}$ or $G\geqslant H$, we are using representatives from the equivalence classes. That said, we begin with a definition necessary for the main proof.

\begin{definition}\label{def:joinmeet}
Let $G\in\mathbb{G}_{n}$. The sets $\lfloor G \rfloor_n$ and $\lceil G \rceil_n$ are defined in the following way:

\begin{itemize}
  \item $\lfloor G \rfloor_n=\{W\text{ such that there is some }\{\GL\,|\,\GR\}\text{ equal to }G\text{, where the}$\\

  \vspace{-1cm}
  \hspace{1.2cm}$\text{birthdays of the games in }\GL\text{ and }\GR\text{ are less than }n\text{, and }W\in\GL\};$\\

    \vspace{-0.4cm}
    \item $\lceil G \rceil_n=\{W\text{ such that there is some }\{\GL\,|\,\GR\}\text{ equal to }G\text{, where the}$\\

  \vspace{-1cm}
  \hspace{1.2cm}$\text{birthdays of the games in }\GL\text{ and }\GR\text{ are less than }n\text{, and }W\in\GR\}.$
\end{itemize}
 \end{definition}

We are now ready for the main result of this section.

\begin{theorem}[Lattice Structure]\label{thm:lattice}
Let $n\geqslant -1$, and $G,H\in\mathbb{G}_{n}\setminus\{\infty,\overline{\infty}\}$. The poset $(\mathbb{G}_{n},\geqslant)$ is a lattice where
\begin{itemize}
  \item $G\vee H=\{\lfloor G \rfloor_n\cup\lfloor H \rfloor_n\,\|\,(\lceil G \rceil_n\cap\lceil H \rceil_n)\cup\{\infty\}\}$,
  \item $G\wedge H=\{(\lfloor G \rfloor_n\cap\lfloor H \rfloor_n)\cup\{\overline{\infty}\}\,\|\,\lceil G \rceil_n\cup\lceil H \rceil_n\}$.
\end{itemize}
 \end{theorem}

\begin{proof}
If either of the games $G$ or $H$ is $\infty$ or $\overline{\infty}$, the elements $G\vee H$ and $G\wedge H$ arise from the absorbing nature of infinities. It is an easy application of Theorem \ref{thm:inflarge} that, if $G=\infty$, then $G\vee H=\infty$ and $G\wedge H=H$ regardless of $H$, and if $G=\overline{\infty}$, then $G\vee H=G$ and $G\wedge H=\overline{\infty}$ regardless of $H$. In other words, $\top=\infty$ and $\bot=\overline{\infty}$. Hence, it is enough to consider $G,H\in\mathbb{A}_{n}\setminus\{\infty,\overline{\infty}\}$ and prove the facts stated in the theorem.

Firstly, let us prove that $G\vee H\geqslant G$ and $G\vee H\geqslant H$, where the join is \linebreak$G\vee H=\{\lfloor G \rfloor_n\cup\lfloor H \rfloor_n\,\|\,(\lceil G \rceil_n\cap\lceil H \rceil_n)\cup\{\infty\}\}$. We only need to prove that $G\vee H\geqslant G=\{\GL\,|\,\GR\}$, since the proof for $G\vee H\geqslant H$ is analogous. The sets $\GL$ and $\GR$ are chosen in such a way that their elements have a birthday less than or equal to $n$. If Left, playing first, wins $G+X$ with some $G^L+X$, since $G^L\in\lfloor G \rfloor\subseteq \lfloor G \rfloor\cup\lfloor H \rfloor$, Left also wins $G\vee H+X$ by moving to $G^L+X$. If Left, playing first, wins $G+X$ with some $G+X^L$, then, by induction, Left also wins $G\vee H+X$ with $G\vee H+X^L$.  On the other hand, if Right, playing first, wins $G\vee H+X$ with some $(G\vee H)^R+X$, due to $(G\vee H)^R\in \lceil G \rceil_n$, $(G\vee H)^R$ is a Right option of some $G'$, equivalent to $G$. Thus, Right also wins $G'+X$ with the move $(G\vee H)^R+X$. Given the equivalence of $G'$ and $G$, Right, playing first, must also win $G+X$. If Right, playing first, wins $G\vee H+X$ with some $G\vee H+X^R$, then, by induction, Right also wins $G\vee H+X$ with $G\vee H+X^R$.

Secondly, for $W\in\mathbb{G}_{n}$, let us prove that if $W\geqslant G$ and $W\geqslant H$, then $W\geqslant G\vee H$. Since it is a very straightforward fact to justify if $W$ is $\infty$ or $\overline{\infty}$, let us assume that this is not the case.
Let us start by proving that any $W^R\in W^\mathcal{R}$ is an element of $\lceil G \rceil_n\cap\lceil H \rceil_n$. To do so, it is sufficient to prove that any element $W^R$ is an element of $\lceil G \rceil_n$, as doing the same in relation to $\lceil H \rceil_n$ is similar.

If $G=\{\GL\,|\,\GR\}$, then $G=\{\GL\,|\,\GR\cup\{W^R\}\}$. This happens because, on one hand, according to Theorem \ref{thm:greediness}, $\{\GL\,|\,\GR\}\geqslant \{\GL\,|\,\GR\cup\{W^R\}\}$. On the other hand, $\{\GL\,|\,\GR\cup\{W^R\}\}\geqslant \{\GL\,|\,\GR\}$ is a consequence that if $W^R+X$ is a winning move for Right, given $W\geqslant G$, there must also be a winning move for Right in $G+X$.
Once $\{\GL\,|\,\GR\cup\{W^R\}\}$ is a possible form of $G$, by Definition \ref{def:joinmeet}, $W^R\in\lceil G \rceil_n$.

Now that we have established that $W^R\in\lceil G \rceil_n\cap\lceil H \rceil_n$, we prove that $W\geqslant G\vee H$. If Left, playing first, wins $G\vee H+X$ with some $(G\vee H)^L+X$, then that move is a winning move in $G+X$ or in $H+X$, given that $(G\vee H)^\mathcal{L}=\lfloor G \rfloor_n\cup\lfloor H \rfloor_n$. Since $W\geqslant G$ and $W\geqslant H$, Left must also have a winning move in $W+X$. If Left, playing first, wins $G\vee H+X$ with some $G\vee H+X^L$, then, by induction, Left also wins $W+X$ with $W+X^L$. If Right, playing first, wins $W+X$ with some $W^R+X$, due to the fact $W^R\in\lceil G \rceil_n\cap\lceil H \rceil_n$, Right also wins $ G\vee H+X$ by moving to $W^R+X$. If Right, playing first, wins $W+X$ with some $W+X^R$, then, by induction, Right also wins $G\vee H+X$ with $G\vee H+X^R$.

The same type of reasoning applies to $G\wedge H$.
\end{proof}

We conclude this section with some examples. With two atoms instead of one, the lattice structure of the affine games born by day zero has more elements than that of classical theory (Figure \ref{fig13}).

\begin{figure}[!htb]
\begin{center}
\scalebox{1}{
\begin{tikzpicture}[thick, scale=1.5]
  \node (top) at (0,4) {$\infty$};
  \node (a) at (0,3) {$\infty_1$};
  \node (b) at (-1,2) {$\bm{\mathsf{0}}$};
  \node (c) at (1,2) {$\pm \infty$};
  \node (d) at (0,1) {$\overline{\infty}_1$};
  \node (bot) at (0,0) {$\overline{\infty}$};
  \draw[-] (top) -- (a);
  \draw[-] (a) -- (b);
  \draw[-] (a) -- (c);
  \draw[-] (b) -- (d);
  \draw[-] (c) -- (d);
  \draw[-] (d) -- (bot);
\end{tikzpicture}}
\end{center}

\vspace{-0.4cm}
\caption{The lattice structure of the poset $(\mathbb{G}_{0},\geqslant)$.}
\label{fig13}
\end{figure}
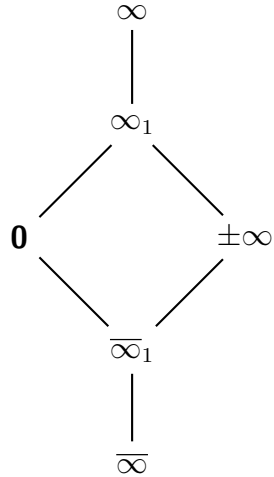

Figure \ref{fig14} shows the lattice structure of the affine games born by day one. The bold represents the substructure of Conway games born by day one, which is order embedded. Figure \ref{fig15} shows an {\sc atari go} position for each game value in $\mathbb{G}_{1}$.

\begin{figure}[!htb]
\begin{center}
\scalebox{1}{
\includegraphics[scale=1]{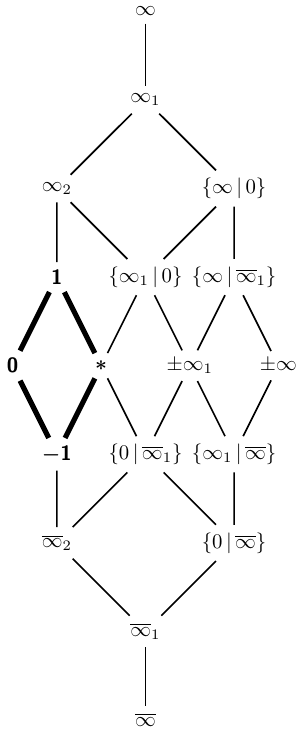}}
\end{center}

\vspace{-0.4cm}
\caption{The lattice structure of the poset $(\mathbb{G}_{1},\geqslant)$.}
\label{fig14}
\end{figure}

It is also pertinent to observe that, although this does not happen in $\Npc$, in $\Np$ a restriction can alter the game equivalence. To clarify what is meant by this, consider the games $\moon\!\!=\{0,\circlearrowright^{0}\,|\,0,\circlearrowleft^{0}\}$ and $\{*,\circlearrowright^{*}\,|\,*,\circlearrowleft^{*}\}$. These two games are not equivalent since, if $X=\{0\,|\,-1\}$, then Left, playing first, wins $\moon\!\!+X$ but loses $\{*,\circlearrowright^{*}\,|\,*,\circlearrowleft^{*}\}+X$.

On the other hand, we can define an affine game form $G$ as \emph{symmetric} if it is different from $\infty$ and $\overline{\infty}$, and $\GR=-\GL$. With the aid of the concept of symmetry, we can naturally define that $G\in\Np$ is \emph{impartial} if it is symmetric and all its quiet followers are symmetric. The class of impartial games is denoted by $\mathbb{I}$. Following the establishment of these concepts, the game equivalence \emph{modulo} $\mathbb{I}$, denoted by $=_{\mathbb{I}}$, is defined exactly the same way as in Definition \ref{def:orderequivalence}, with the difference that the distinguishing games $X$ have to be elements of $\mathbb{I}$. In \cite{Lar21,Lar23}, it was shown that the game values of this restricted structure are the well-known nimbers plus an extra value, which is $\moon$. In this restriction, we already have $\moon\!\!=_{\mathbb{I}}\{*,\circlearrowright^{*}\,|\,*,\circlearrowleft^{*}\}$. In fact, we even have $\moon\!\!=_{\mathbb{I}}\pm\infty$. In other words, the distinction between $\moon$ and $\{*,\circlearrowright^{*}\,|\,*,\circlearrowleft^{*}\}$ cannot be made with impartial games, and that is the reason why these games are equivalent modulo $\mathbb{I}$ but not equivalent in $\Np$. This also explains why $\moon\!\!$ was included in Figure \ref{fig12}.

\begin{figure}[!htb]
\begin{center}
\begin{tabular}{|c|c|c|}
  \hline
    && \\
  \raisebox{0.2cm}{$\infty$} & \raisebox{0.2cm}{$\infty_1$} & \raisebox{0.2cm}{$\infty_2$} \\
  \includegraphics[scale=1]{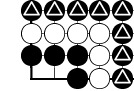} &
  \includegraphics[scale=1]{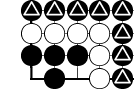} &
 \includegraphics[scale=1]{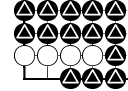}\\
    && \\
  \hline
    && \\
  \raisebox{0.2cm}{$\{\infty\,|\,0\}$} & \raisebox{0.2cm}{$1$} & \raisebox{0.2cm}{$\{\infty_1\,|\,0\}$} \\
   \includegraphics[scale=1]{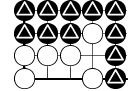} &
  \includegraphics[scale=1]{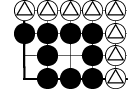} &
  \includegraphics[scale=1]{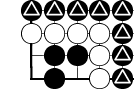}\\
    && \\
  \hline
    && \\
  \raisebox{0.2cm}{$\{\infty\,|\,\overline{\infty}_1\}$} & \raisebox{0.2cm}{$0$} & \raisebox{0.2cm}{$*$} \\
  \includegraphics[scale=1]{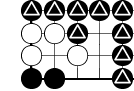} &
  \includegraphics[scale=1]{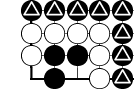} &
  \includegraphics[scale=1]{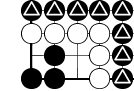}\\
    && \\
  \hline
   && \\
  \raisebox{0.2cm}{$\pm\infty_1$} & \raisebox{0.2cm}{$\pm\infty$} & \\
  \includegraphics[scale=1]{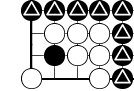} &
  \includegraphics[scale=1]{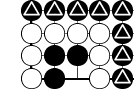} &
  \\
    && \\
  \hline
\end{tabular}
\end{center}

\vspace{-0.4cm}
\caption{Materialization of $\mathbb{G}_{1}$ through {\sc atari go} positions.}
\label{fig15}
\end{figure}

\section{Analysis of an example taken from game practice}
\label{sec:example}

\begin{figure}[!htb]
\begin{center}
\includegraphics[scale=1]{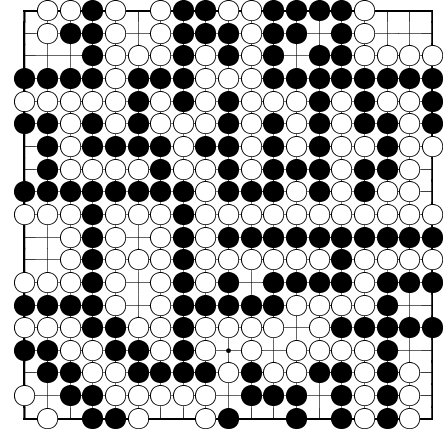}
\end{center}

\vspace{-0.4cm}
\caption{What is the outcome of this game?}
\label{fig16}
\end{figure}

The previous sections established the structure $\Np$, along with its properties. In this section, we discuss its application to game practice. Analysing combinatorial games is important not only for the games themselves, but also for the paths it opens up. For example, recent projects carried out by DeepMind, such as {\sc AlphaGo} and {\sc AlphaZero}, aimed to improve AI capabilities in playing combinatorial games. However, the goal was not only to analyze specific games, but also to push AI and mathematics, promoting algorithms capable of dealing with important real-world problems.

Nevertheless, one should never forget that games are in themselves a primary application of CGT. At this very moment as the reader is reading this sentence, \emph{hundreds of millions of people} are playing games, many of them making gaming their professional activity. Games constitute a primordial activity as ceremonial burial, arts, education, and so on, playing a role in human development \citep{Hui49}. The reason why games are often not considered ``applications'' lies in two stereotypes: sometimes they are associated with addiction (particulary games of chance), other times they are associated with pleasure. One might think that addiction and pleasure should not be considered applications of mathematics. However, there is no reason why they should not be.

We have chosen {\sc atari go} for this final section because this game, due to its characteristics, is very suitable for the purpose of showing the pertinence of affine normal play theory. However, that is not the only reason. This simplified version of {\sc go} has been used in various countries due to its educational and therapeutic value, illustrating that playing is much more than addiction or pleasure. Concerned about various types of problems in Japanese schools, Yasuda Yasutoshi (9 dan\footnote{The level of {\sc go} master players is described through the dan rank.}) has argued that the use of {\sc atari go} has measurable positive effects \citep{Yas21}.

Without further delay, let us delve into the analysis of a carefully chosen example. We propose an endgame reached in a game with many mistakes made by both players. In the position shown in Figure \ref{fig16}, it is easy to verify that Right wins when playing first. The interesting question is whether Left wins playing first, i.e., determining if the outcome class of the game is $\mathcal{N}$ or $\mathcal{R}$. As usual in CGT, the first step is to divide and conquer.

In Figure \ref{fig17}, moves on intersections marked with ``s'' are suicidal moves and should not be played by either player. Moves on intersections marked with ``t'' are also terrible, as a response in ``m'' leads to a catastrophe. Finally, moves on intersections marked with ``h'' are horrible, weakening the black group. There is a large set of alive pieces.\\

\begin{figure}[!htb]
\begin{center}
\includegraphics[scale=1]{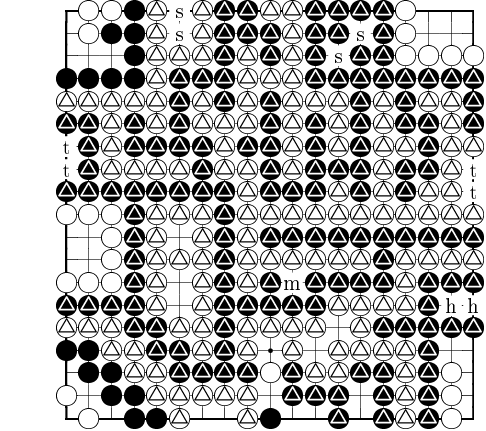}
\end{center}

\vspace{-0.4cm}
\caption{Alive pieces.}
\label{fig17}
\end{figure}

Given the set of alive pieces, it is possible to identify eight disjoint components (Figure~\ref{fig18}). The next step is to determine the game value of each of them.
\clearpage
\begin{figure}[!htb]
\hspace{1.7cm}
\scalebox{0.78}{
\includegraphics[scale=1]{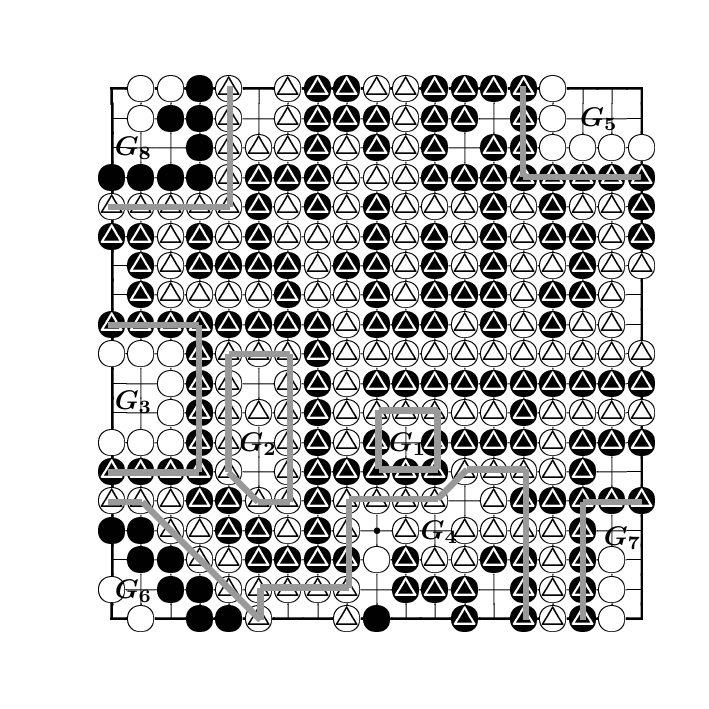}}

\vspace{-0.9cm}
\caption{Disjoint components.}
\label{fig18}
\end{figure}

Starting with the easiest components, $G_1$ is equal to $*$, since either player can place a stone at that intersection. The component $G_2$ is equal to $-1$, as Left cannot play in that region, and Right can make a last move there. Finally, $G_3$ is equal to $0$ because it is a $\mathcal{P}$-position. Left cannot play inside that region, and any move by Right can be answered with a symmetric move diagonally. In this way, these three components are Conway games, with the first one being tepid (infinitesimal) and the second and third being cold (numbers).

Advancing to the analysis of $G_4$, a move at the intersection marked with an ``m'' in Figure \ref{fig19} is mandatory for either player, dominating all others (for ease of drawing, we slightly adjusted the positions of the surrounding pieces without changing the nature of the component). Regardless of whether Left or Right makes the move, it is a check. It is reasonably straightforward to see that $G_4=\{\{\infty\,|\,2\}\,|\,\{-2\,|\,\overline{\infty}\}\}=\{\circlearrowright^{2}\,|\,\circlearrowright^{-2}\}$. Since $2>G_4>-2$, this component is a hammerzug. As Left is not in check in any other component, Theorem \ref{thm:extremeham} ensures  that the first move can be made in this component without having to think about anything else (collecting two guaranteed moves).\\

\begin{figure}[!htb]
\begin{center}
\includegraphics[scale=1]{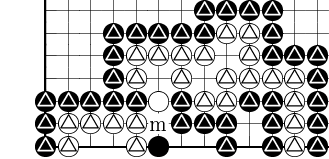}
\end{center}

\vspace{-0.4cm}
\caption{The component $G_4$ is a hammerzug.}
\label{fig19}
\end{figure}

The component $G_5$ does not have a clear analysis as the previous ones. Consider Figure \ref{fig20}. If Right starts and places a stone at intersection ``a'', he wins $3$ guaranteed moves. We leave it to the reader to conclude that if Left starts and places a stone at ``a'', then the resulting game is $\{\{\infty_1\,|\,*\}\,|\,\{*\,|\,\overline{\infty}_1\}\}$, and if she places a stone at ``b'', then the resulting game is $\{\{\infty_2\,|\,*\}\,|\,\{*\,|\,\overline{\infty}_1\}\}$.
Both games can be reduced and are equal to $*$. Thus, $G_5=\{*\,|\,-3\}$, and we have again a Conway game, but this time it is hot.

\begin{figure}[!htb]
\begin{center}
\includegraphics[scale=1]{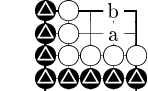}
\end{center}

\vspace{-0.4cm}
\caption{The component $G_5$ is equal to $\{*\,|\,-3\}$.}
\label{fig20}
\end{figure}

The component $G_6$ is very interesting. Both players are in a rush  to place a stone at the intersection marked with an ``m'' in Figure \ref{fig21}. If Right does it, he ensures a winning terminating move in that component. If Left does it, she manages to defend, albeit at the cost of giving one guaranteed move to the opponent. Therefore, we have $G_6=\{-1\,|\,\overline{\infty}_2\}$. This negative component is scalding.

\begin{figure}[!htb]
\begin{center}
\includegraphics[scale=1]{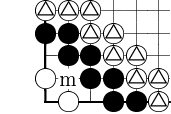}
\end{center}

\vspace{-0.4cm}
\caption{The negative component $G_6$ is scalding.}
\label{fig21}
\end{figure}

The component $G_7$ is also scalding, but it is Left who is threatening to cause a bubbling situation. If Left places a stone at the intersection marked with an ``a'' in Figure \ref{fig22}, she guarantees a winning terminating option in $3$ moves. If Right plays first, placing a stone at ``a'', ``b'' or ``c'' results in a $\mathcal{P}$-position, so Right has a move to $0$. Hence, we have $G_7=\{\infty_3\,|\,0\}$.

\begin{figure}[!htb]
\begin{center}
\includegraphics[scale=1]{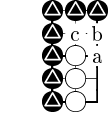}
\end{center}

\vspace{-0.4cm}
\caption{The fuzzy component $G_7$ is scalding.}
\label{fig22}
\end{figure}

The component $G_8$ has a very interesting peculiarity. If Right places a stone at the intersection marked with an ``a'' in Figure \ref{fig23}, he guarantees a winning terminating option in $2$ moves. What is different from the previous cases is that Left has \emph{two distinct ways to defend}. If Left places a stone at the intersection ''b'', then she threatens a move that ensures the victory after one move. Against that move, Right must place a stone at ``a'', also threatening to win. To defend, Left places a stone at ``c'', achieving a $\mathcal{P}$-position, i.e., $0$. It follows that placing a black stone at ''b'' is the option $\{\infty_1\,|\,\{0\,|\,\overline{\infty}_1\}\}$. On the other hand, if the first move is the placement of a black stone at ``a'', Right must defend by placing a stone at ``b'', achieving a $\mathcal{P}$-position, i.e., $0$. Thus, this second option for Left is $\{\infty_1\,|\,0\}$.
We have $G_8=\{\{\infty_1\,|\,\{0\,|\,\overline{\infty}_1\}\},\{\infty_1\,|\,0\}\,|\,\overline{\infty}_2\}$, that can be reduced to $G_8=\{0,\{\infty_1\,|\,0\}\,|\,\overline{\infty}_2\}$. If $G_8$ were the only component in the disjunctive sum, the first option (placement at ``b'') would be the winning move for Left. With more components, the second option defends \emph{in sente},\footnote{From Japanese ``before hand'', playing \emph{sente} (\begin{CJK*}{UTF8}{gbsn}先手\end{CJK*}) means to keep the initiative.} making it a potentially good choice.

\begin{figure}[!htb]
\begin{center}
\includegraphics[scale=1]{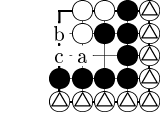}
\end{center}

\vspace{-0.4cm}
\caption{The component $G_8$ is scalding and Left can defend in sente.}
\label{fig23}
\end{figure}

Given the analysis of the components, the position shown in Figure \ref{fig16} is the following disjunctive sum: $$G=*+(-1)+0+\{\circlearrowright^{2}\,|\,\circlearrowright^{-2}\}+\{*\,|\,-3\}+\{-1\,|\,\overline{\infty}_2\}+\{\infty_3\,|\,0\}+\{0,\{\infty_1\,|\,0\}\,|\,\overline{\infty}_2\}.$$

\begin{figure}[!htb]
\begin{center}
\includegraphics[scale=1]{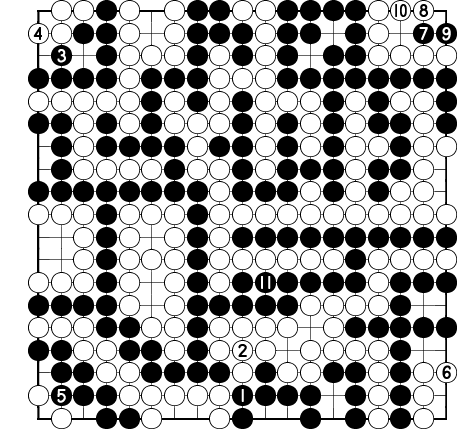}
\end{center}

\vspace{-0.4cm}
\caption{Left, playing first, wins.}
\label{fig24}
\end{figure}

The presence of multiple scalding components, as well as a hammerzug, leads to the conjecture that there must have been several previous errors. Typically, when a scalding component comes into play, it cools down in the next move as the threatened player attempts to defend. However, this is not one hundred percent certain, as the best defense can be the attack, which might be executed with a ``faster'' threat in another component. Returning to the game $G$, after the moves in the fourth and eighth components (in this order), Right has to respond, and the following sum is obtained (it is worth noting that Left employed the defense in sente in the eighth component):
$$G=*+(-1)+0+2+\{*\,|\,-3\}+\{-1\,|\,\overline{\infty}_2\}+\{\infty_3\,|\,0\}+0.$$
At that point, Left defends in the component $\{-1\,|\,\overline{\infty}_2\}$, obtaining the following sum:
$$G=*+(-1)+0+2+\{*\,|\,-3\}+(-1)+\{\infty_3\,|\,0\}+0.$$
Now that Right is under threat in the scalding component $\{\infty_3\,|\,0\}$, he has no time to play in the hot Conway game and must respond with
$$G=*+(-1)+0+2+\{*\,|\,-3\}+(-1)+0+0.$$
Left finishes gloriously by playing to
$$G=*+(-1)+0+2+*+(-1)+0+0.$$
Note that if the empty intersection concerning $G_1$ did not exist, Left would not be able to win. Figure \ref{fig24} shows the exact winning sequence.

\section*{Future work}
\label{sec:future}

With the algebraic structure $\Np$ established, it is important to illustrate its applicability through the analysis of concrete rulesets. A complete analysis of {\sc whackenbush} may be an interesting step in that direction.

From a theoretical point of view, it is also relevant to investigate how to extend atomic weight theory and temperature theory to encompass affine normal play games.

\end{document}